\documentclass[11pt]{article}

\usepackage{tikz}

\usepackage{lscape,multirow}
\usetikzlibrary{arrows}
\usetikzlibrary{matrix}
\usetikzlibrary{matrix,arrows}

\usepackage{amsmath} % eg. for \begin{equation}
\usepackage{amsxtra} % for \sphat
\usepackage{mathrsfs}
\usepackage{amssymb} % eg. for \leqslant
\usepackage{theorem}
\usepackage{rotating}
\usepackage{stmaryrd}
\usepackage[all,cmtip]{xy}
\usepackage{pict2e}
\usepackage{feyn}

\theoremstyle{change} % puts numbers IN FRONT of "Theorem"
\newtheorem{theorem}{Theorem}[section] % defines environment "Theorem".
\newtheorem{lemma}[theorem]{Lemma} % defines environment "Lemma", that
\newtheorem{proposition}[theorem]{Proposition}
\newtheorem{corollary}[theorem]{Corollary}
\theorembodyfont{\rmfamily} % has the effect that the content of Remark,
\newtheorem{remark}[theorem]{Remark}
\newtheorem{example}[theorem]{Example}

\newtheorem{definition}[theorem]{Definition}
\newtheorem{notation}[theorem]{Notation}
\newtheorem{nothing}[theorem]{} % empty Theoremumgebung.

\newtheorem{hypotheses}[theorem]{Hypotheses}

\newenvironment{proof}{\noindent{\bf Proof}\ }{\qed\bigskip}

\renewcommand{\l}{(}
\renewcommand{\r}{)}

\renewcommand{\le}{\leqslant}
\renewcommand{\leq}{\leqslant}
\renewcommand{\geq}{\geqslant}
\renewcommand{\unlhd}{\trianglelefteqslant}

\renewcommand{\marginpar}[1]{}

\newcommand{\Aut}{\mathrm{Aut}}

\newcommand{\BIGOP}[1]
  {\mathop{\mathchoice
  {\raise-0.22em\hbox{\huge $#1$}}
  {\raise-0.05em\hbox{\Large $#1$}}{\hbox{\large $#1$}}{#1}}}

\newcommand{\catC}{\catfont{C}}

\newcommand{\catfont}{\mathsf}

\newcommand{\bigdisj}{\mathop{\buildrel.\over\bigcup}}

\newcommand{\End}{\mathrm{End}}
\newcommand{\ftilde}{\tilde{f}}

\newcommand{\Hom}{\mathrm{Hom}}
\newcommand{\Hd}{\mathrm{Hd}}

\newcommand{\id}{\mathrm{id}}

\newcommand{\Ind}{\mathrm{Ind}}
\newcommand{\Inf}{\mathrm{Inf}}

\newcommand{\Inn}{\mathrm{Inn}}

\newcommand{\Mor}{\mathrm{Mor}}
\newcommand{\myiso}{\buildrel\sim\over\to}

\newcommand{\Lbar}{\overline{L}}

\newcommand{\Ob}{\mathrm{Ob}}
\newcommand{\Out}{\mathrm{Out}}

\newcommand{\qed}{\nobreak\hfill
                   \vbox{\hrule\hbox{\vrule\hbox to 5pt
                   {\vbox to 8pt{\vfil}\hfil}\vrule}\hrule}}

\newcommand{\Rad}{\mathrm{Rad}}

\newcommand{\scrJ}{\mathscr{J}}

\newcommand{\stab}{\mathrm{stab}}

\newcommand{\Soc}{\mathrm{Soc}}

\newcommand{\sgn}{\mathrm{sgn}}
\newcommand{\Ttilde}{\tilde{T}}

\title{Quasi-Hereditary Structure of Twisted Split Category Algebras Revisited\footnote{{\bf MR Subject Classification:} 16G10, 20M17, 19A22. {\bf Keywords:} split category, regular monoid, quasi-hereditary algebra, highest weight category, biset functor, Brauer algebra.} }
\author{\small Robert Boltje\\
  \small Department of Mathematics\\
  \small University of California\\
  \small Santa Cruz, CA 95064\\
  \small U.S.A.\\
  \small boltje@ucsc.edu
  \and
  \small Susanne Danz\\
  \small Department of Mathematics\\ 
  \small University of Kaiserslautern\\
  \small P.O. Box 3049\\
  \small 65653 Kaiserslautern\\
  \small Germany\\
  \small danz@mathematik.uni-kl.de}
\date{May 2, 2014}

\begin{document}

\sloppy

\maketitle

%%%%%%%%%%%%% ABSTRACT %%%%%%%%%%%%%%%%%%%%%%%%%%%%%%%%%%%%%%%

\begin{abstract}
Let $k$ be a field of characteristic $0$, let $\catC$ be a finite split category, let $\alpha$ be a 2-cocycle of $\catC$ with values in the multiplicative group of $k$, and consider the resulting twisted category algebra $A:=k_\alpha\catC$. Several interesting algebras arise that way,  for instance, the Brauer algebra. Moreover, the category of biset functors over $k$ is equivalent to a module category over a condensed algebra $\varepsilon A\varepsilon$, for an idempotent $\varepsilon$ of $A$. In \cite{BDIII} the authors proved that $A$ is quasi-hereditary (with respect to an explicit partial order $\le$ on the set of irreducible modules), and standard modules were given explicitly. Here, we improve the partial order $\le$ by introducing a coarser order $\unlhd$ leading to the same results on $A$, but which allows to pass the quasi-heredity result to the condensed algebra $\varepsilon A\varepsilon$ describing biset functors, thereby giving a different proof of a quasi-heredity result of Webb, see \cite{Webb}. The new partial order $\unlhd$ has not been considered before, even in the special cases, and we evaluate it explicitly for the case of biset functors and the Brauer algebra. It also puts further restrictions on the possible composition factors of standard modules. 
\end{abstract}

%%%%%%%%%%%%%%%%%%%%%%%%%%%%%%%%%%%%%%%%%%%%%%%%%%%%%%%%%%

\section{Introduction}\label{sec intro}

Suppose that $k$ is a field and that $\catC$ is a finite category, that is, the morphisms in $\catC$ form a finite set. Suppose
further that $\alpha$ is a $2$-cocycle of $\catC$ with values in $k^\times$. Then the {\it twisted category algebra} $k_\alpha\catC$ is
the finite-dimensional $k$-algebra with the morphisms in $\catC$ as a $k$-basis, and with multiplication induced by
composition of morphisms, twisted by the $2$-cocycle $\alpha$; for a precise definition, see \ref{nota twisted cat alg}.
In the case where the category has one object only this just recovers the notion of a twisted monoid algebra.

\medskip
In recent years, (twisted) category algebras and (twisted) monoid algebras have been intensively studied by B. Steinberg et al., for instance in \cite{GMS}, as well as 
by Linckelmann and Stolorz, who, in particular, determined the isomorphism classes of simple $k_\alpha\catC$-modules in \cite{LS}.
As a consequence of \cite[Theorem~1.2]{LS} the isomorphism classes of simple $k_\alpha\catC$-modules can be parametrized by 
a set $\Lambda$ of pairs
whose first entry varies over certain finite groups related to $\catC$ (called {\it maximal subgroups} of $\catC$), and whose second entry
varies over the isomorphism classes of simple modules over twisted group algebras of these maximal subgroups.

\medskip
For convenience, in the following we shall suppose that $k$ has characteristic 0, but this condition can be relieved, as we shall see in Theorem~\ref{thm main}.
Moreover, we suppose that the category $\catC$ is {\it split}, that is, every morphism
in $\catC$ is a split morphism, see \ref{nota twisted cat alg}(a).
It has been shown by the authors in \cite{BDIII} and, independently, by Linckelmann and Stolorz in \cite{LS2} 
that the resulting twisted category algebra $k_\alpha\catC$ is quasi-hereditary in the sense of \cite{CPS}.
In \cite{BDIII} we also determined the standard modules of $k_\alpha\catC$ with respect to a natural partial order $\leq$
on the labelling set $\Lambda$ of isomorphism classes of simple modules, which depends only on the first entries of pairs in $\Lambda$ and is explained in \ref{noth leq}.

\medskip
Since $k_\alpha\catC$ is quasi-hereditary with respect to $(\Lambda,\leq)$, it is also quasi-hereditary with respect to any refinement of $\leq$, and the corresponding standard and costandard modules are the same as those with respect to $(\Lambda,\leq)$. This is an immediate consequence of Defintion~\ref{defi qh} below.

In this paper we introduce a new partial order $\unlhd$ on $\Lambda$ such that the partial order $\leq$ is a refinement
of $\unlhd$. 
We shall then show in Theorem~\ref{thm main} that the algebra $k_\alpha\catC$ remains quasi-hereditary with respect to this new partial order. Furthermore, we shall show that the standard and costandard modules of
 $k_\alpha\catC$ with respect to the two partial orders coincide, and we shall give explicit descriptions of these modules. 
The partial order $\unlhd$ seems more natural than the initial one, since it depends on both entries of the pairs in $\Lambda$, and it allows to pass the hereditary structure to idempotent condensed algebras $\varepsilon \cdot k_\alpha\catC\cdot\varepsilon$, for $\varepsilon^2=\varepsilon\in k_\alpha\catC$, in a particular case we are interested in and that is related to the category of biset functors, see Section~\ref{sec bisets}.  %Unfortunately, we do not have a similarly concrete description of the tilting modules in this general setup.
In Sections~\ref{sec sqsubseteq} and \ref{sec duality} we shall give a number of possible reformulations and simplifications of the defining properties of the partial order $\unlhd$
 that are particularly useful when considering concrete examples. 
 
\medskip
It is known, by work of Wilcox \cite{Wil}, that diagram algebras such as Brauer algebras, Temperley--Lieb algebras, partition algebras, and relatives of these arise naturally as twisted split category algebras and twisted regular monoid algebras; for 
a list of references, see the introductions to \cite{BDIII,LS2}. 
Our initial motivation for studying the structure of twisted category algebras comes from our results in
 \cite{BDII}, where we have shown that 
the double Burnside algebra of a finite group over $k$ is isomorphic to a $k$-algebra that is obtained from
a twisted split category algebra by idempotent condensation. The latter result is also valid for more general algebras related to
the category of biset functors on a finite section-closed set of finite groups; see \cite[Section 5]{BDII} and Section~\ref{sec bisets} of this article.
Therefore, in Section~\ref{sec bisets} and in Section~\ref{sec brauer} we shall apply our general results concerning
the quasi-hereditary structure of twisted split category algebras to the algebra related to biset functors just mentioned and to the 
Brauer algebra, respectively. In doing so we shall, in particular, derive new information 
about decomposition numbers of the algebras under consideration, since
our partial order $\unlhd$ is in general strictly coarser than the well-known partial order
$\leq$ on $\Lambda$.

Moreover, in Theorem~\ref{thm condense} we 
recover and slightly improve a result of Webb in \cite{Webb} stating that the category of biset functors over $k$ is a highest weight
category, when the underlying category of finite groups is finite. 
%we shall also comment on infinite categories in 
%Remark~\ref{rem conn with biset functors}. 
The key step towards the latter result will be showing that our newly introduced
partial order $\unlhd$ behaves particularly well with respect to this particular idempotent condensation, whereas the finer order $\leq$ does not; see Example~\ref{expl S_4}. 
%This in fact was the main motivation to introduce the more complicated but also more natural partial order $\unlhd$ for general twisted split category algebras.

\medskip
Many known examples of twisted split category algebras $k_\alpha\catC$ arise from categories  equipped with a contravariant functor
that is the identity on objects and that gives rise to a duality on the category of left $k_\alpha\catC$-modules. 
Such a  duality, in particular, sends standard modules to costandard modules, and allows for further simplifications of the partial order
$\unlhd$ on the set $\Lambda$. We shall analyze this duality in detail in Section~\ref{sec duality}, and apply these general 
results to our concrete examples in Section~\ref{sec bisets} and Section~\ref{sec brauer}.

%\medskip
%
%The present paper is organized as follows: In Section~\ref{sec qh}
%we begin by recalling the definition and some basic properties of quasi-hereditary algebras, and in Section~\ref{sec twisted cat} we 
%briefly summarize some known results concerning twisted category algebras that will be needed in the course of this paper.

\bigskip\noindent
{\bf Acknowledgements:}\, The second author's research has been supported by  the DFG Priority Programme
`Representation Theory' (Grant \# DA1115/3-1), and a Marie Curie Career Integration Grant
(PCIG10-GA-2011-303774).

%%%%%%%%%%%%%%%%%%%%%%%%%%%%%%%%%%%%%%%%%%%%%%%%%%%%%%%%%%

\section{Quasi-Hereditary Algebras}\label{sec qh}

In the following, $k$ will denote an arbitrary field.
We begin by briefly recalling the definition and some basic  properties of a quasi-hereditary $k$-algebra needed in this article. For more details and background we refer the reader to \cite{CPS} and \cite[Appendix]{D}. Unless specified otherwise, 
modules over any finite-dimensional $k$-algebra are understood to be finite-dimensional {\em left} modules.

\begin{definition}[Cline--Parshall--Scott \cite{CPS}]\label{defi qh}
Let $k$ be a field, and let $A$ be  a finite-dimensional $k$-algebra. Let further $\Lambda$ be a finite set parametrizing the isomorphism classes of simple $A$-modules, and let $\leq $ be a partial order on $\Lambda$.  For $\lambda\in\Lambda$, let $D_\lambda$ be  a simple $A$-module labelled by $\lambda$, let $P_\lambda$ be a projective cover of $D_\lambda$, and let $I_\lambda$ be an injective envelope of $D_\lambda$.

\smallskip

(a)\, For $\lambda\in\Lambda$, let $\Delta_\lambda$ be the unique maximal quotient module $M$ of $P_\lambda$ such that
all composition factors of $\Rad(M)$ belong to the set $\{D_\mu\mid \mu<\lambda\}$. Then $\Delta_\lambda$ is called the
{\it standard module} of $A$ with respect to $(\Lambda,\leq)$ labelled by $\lambda$.

\smallskip

(b)\, For $\lambda\in\Lambda$, let $\nabla_\lambda$ be the unique maximal submodule $N$ of $I_\lambda$ such that all composition factors of $N/\Soc(N)$ belong to the set $\{D_\mu\mid \mu<\lambda\}$. Then $\nabla_\lambda$ is called the {\it 
costandard module} of $A$ with respect to $(\Lambda,\leq)$ labelled by $\lambda$.
\smallskip

(c)\, The $k$-algebra $A$ is called {\it quasi-hereditary} with respect to $(\Lambda,\leq)$ if, for each $\lambda\in\Lambda$, the projective module $P_\lambda$ admits a filtration
$$0=P_\lambda^{(0)}\subset P_\lambda^{(1)}\subset\cdots\subset P_\lambda^{(m_\lambda)}=P_\lambda$$
satisfying the following properties:

\smallskip
\quad (i)\, $P_\lambda^{(m_\lambda)}/P_\lambda^{(m_\lambda-1)}\cong \Delta_\lambda$, and

\smallskip
\quad (ii)\, if $1\leq q<m_\lambda$ then $P_\lambda^{(q)}/P_\lambda^{(q-1)}\cong \Delta_{\mu}$, for some $\mu\in\Lambda$ with $\lambda<\mu$.
\end{definition}

\bigskip
The notion of a quasi-hereditary algebra can be defined equivalently in terms of costandard modules:

\begin{proposition}[\protect{\cite[Definition~A2.1, Lemma~A3.5]{D}}]\label{prop co standard}
With the notation as in Definition~\ref{defi qh}, the $k$-algebra $A$ is quasi-hereditary with respect to $(\Lambda,\leq)$ if and only if, for each $\lambda\in\Lambda$, the injective module $I_\lambda$ admits a filtration
$$0=I_\lambda^{(0)}\subset I_\lambda^{(1)}\subset\cdots\subset I_\lambda^{(n_\lambda)}=I_\lambda$$
such that

\quad {\rm (i)}\, $I_\lambda^{(1)}\cong \nabla_\lambda$, and

\quad {\rm (ii)}\, if $1< q\leq n_\lambda$ then $I_\lambda^{(q)}/I_\lambda^{(q-1)}\cong \nabla_{\mu}$, for some $\mu\in\Lambda$ with $\lambda<\mu$.
\end{proposition}

\begin{nothing}\label{noth duals}
{\bf Dual modules.}\, 
(a) Suppose again that $A$ is a finite-dimensional $k$-algebra. For any left $A$-module $M$, we set $M^*:=\Hom_k(M,k)$ and view $M^*$ as a right $A$-module via $(\lambda\cdot a)(m):=\lambda(am)$, for $\lambda\in M^*$, $a\in A$, and $m\in M$. Similarly, if $N$ is a right $A$-module, the $k$-linear dual $N^*$ is a left $A$-module. If $\Lambda$ labels the isomorphism classes of simple left $A$-modules then it also labels the isomorphism classes of simple right $A$-modules: if $D_\lambda$ is the simple left $A$-modules labelled by $\lambda\in\Lambda$ then we choose $D_\lambda^*$ to be the simple right $A$-module labelled by $\lambda$. 
The dual $P_\lambda^*$ of the projective cover of $D_\lambda$ is an injective envelope of $D_\lambda^*$, and the dual $I_\lambda^*$ of the injective envelope of $D_\lambda$ is a projective cover of $D_\lambda^*$.
Moreover, if $\leq$ is a partial order on $\Lambda$ and if $\Delta_\lambda$ (respectively, $\nabla_\lambda$) is the corresponding standard (respectively, costandard) left $A$-module labelled by $\lambda\in\Lambda$, then $\nabla_\lambda^*$ (respectively, $\Delta_\lambda^*$) is the corresponding standard (respectively, costandard) right $A$-module labelled by $\lambda$. Using Proposition~\ref{prop co standard} one sees that $A$ is quasi-hereditary with respect to $\leq$ in the left module formulation if and only if $A$ is quasi-hereditary in the similar right module formulation.

\smallskip
(b) Suppose further that there is a $k$-algebra anti-involution $-^\circ:A\to A$, that is, $-^\circ$ is a $k$-linear isomorphism of order 2 that satisfies $(ab)^\circ=b^\circ a^\circ$, for all $a,b\in A$.
Then, given any finite-dimensional left $A$-module $M$, its $k$-linear dual $M^*$ also carries a left $A$-module structure via
\begin{equation}\label{eqn dual module}
(a\cdot f)(m):=f(a^\circ\cdot m)\quad (a\in A, \, f\in\Hom_k(M,k),\, m\in M)\,.
\end{equation}
We denote the resulting left $A$-module by $M^\circ$. Note that $(M^\circ)^\circ\cong M\cong (M^*)^*$ as left $A$-modules. 

\smallskip
If $A$ is the group algebra $kG$ for a finite group $G$, then we have a canonical anti-involution induced by the map $G\to G,\; g\mapsto g^{-1}$. By abuse of notation, in this particular case we still write $M^*$ for the left $A$-module $M^\circ$. So, in this case, $M^*$ can mean a left or a right $kG$-module, and we shall clarify this when necessary.

\smallskip
Let $\Lambda$ be a set labelling the isomorphism classes of simple left $A$-modules, and let $\leq$ be a partial order on $\Lambda$.
For $\lambda\in\Lambda$, the dual $D_\lambda^\circ$ of the simple left $A$-module $D_\lambda$ is again a simple left $A$-module, so that there is some $\lambda^\circ\in\Lambda$ with $D_\lambda^\circ\cong D_{\lambda^\circ}$. Moreover, the dual $P_\lambda^\circ$ of the projective cover $P_\lambda$ of $D_\lambda$ then is an injective envelope of $D_\lambda^\circ$. Imposing an additional assumption on the poset $(\Lambda,\leq)$, the duality also relates standard modules to costandard modules as follows. 
\end{nothing}

\begin{proposition}\label{prop dual standard modules}
Retain the assumptions from \ref{noth duals}. Suppose further that, for all $\lambda,\mu\in\Lambda$, one has $\lambda\leq \mu$ if and only if $\lambda^\circ\leq \mu^\circ$. Then, for each $\lambda\in\Lambda$, one has $A$-module isomorphisms
$$\Delta_\lambda^\circ\cong \nabla_{\lambda^\circ}\quad \text{ and }\quad \nabla_\lambda^\circ\cong \Delta_{\lambda^\circ}\,.$$
%
%\smallskip
%{\rm (b)}\, Assume that $k$ is a splitting field for $A$ and that $A$ is quasi-hereditary with respect to $(\Lambda,\leq)$. Let $\unlhd$ be a partial order on $\Lambda$ such that $\leq$ is a refinement of $\unlhd$ with the property that
%$\lambda\unlhd\mu$ holds if and only if $\lambda^\circ\unlhd\mu^\circ$, for all
%$\lambda,\mu\in\Lambda$.
%Furthermore, suppose that, for each $\lambda\in\Lambda$, every composition factor of $\Rad(\Delta_\lambda)$ already belongs to the set $\{D_\mu\mid \mu\lhd\lambda\}$. Then $A$ is also quasi-hereditary with respect
%to $(\Lambda,\unlhd)$, and the corresponding standard modules and costandard modules do not change.
\end{proposition}

\begin{proof}
Let $\lambda\in \Lambda$. By definition, $\Delta_\lambda$ is the largest quotient module of $P_\lambda$ such that all composition factors of its radical belong to $\{D_\mu\mid \mu<\lambda\}$. Hence $\Delta_\lambda^\circ$ is the largest submodule of $P_\lambda^\circ\cong I_{\lambda^\circ}$ such that all composition factors of its cosocle belong to the set $\{D_\mu^\circ\mid \mu<\lambda\}=\{D_{\mu^\circ}\mid \mu<\lambda\}$.
Since, by our hypothesis, $\mu<\lambda$ if and only if $\mu^\circ<\lambda^\circ$, this shows that $\Delta_\lambda^\circ\cong \nabla_{\lambda^\circ}$.

The second isomorphism in (a) follows analogously.
%
%\smallskip
%To prove (b), let $\lambda\in\Lambda$ be arbitrary.
%By Lemma~\ref{lemma coarser order}, we already know that the module $\Delta_\lambda$ is also the standard module of $A$ labelled by $\lambda$ with respect to $(\Lambda,\unlhd)$. Also by Lemma~\ref{lemma coarser order} and (a),
%$\nabla_\lambda$  is the costandard module of $A$ labelled by $\lambda$ with respect to $(\Lambda,\unlhd)$. Since $A$ is quasi-hereditary with respect to $(\Lambda,\leq)$, the projective module $P_\lambda$ has a filtration whose top factor is isomorphic to $\Delta_\lambda$ and whose remaining factors belong to $\{\Delta_\nu\mid \lambda<\nu\}$. Now, by BGG reciprocity, see \cite[Proposition~A2.2(iv)]{D}, for every $\lambda<\nu$, the multiplicity of $\Delta_\nu$ in any such filtration is well defined and equals
%$$[\nabla_\nu:D_\lambda]\,.$$
%By part~(a), we in turn have
%$$[\nabla_\nu:D_\lambda]=[\nabla_\nu^*:D_\lambda^*]=[\Delta_{\nu^*}:D_{\lambda^*}]\,,$$
%that is, $\lambda^*\lhd\nu^*$, thus also $\lambda\lhd\nu$. So, altogether this shows that $P_\lambda$ has indeed a filtration
%whose top factor is isomorphic to $\Delta_\lambda$ and whose remaining factors belong to $\{\Delta_\nu\mid \lambda\lhd\nu\}$. 
%Hence $A$ is quasi-hereditary with respect to $(\Lambda,\unlhd)$.
\end{proof}

\begin{remark}\label{rem BGG}
Keep the assumptions
as in \ref{noth duals}(b).
In the case where, for every $\lambda\in\Lambda$, one has $D_\lambda^\circ\cong D_\lambda$, the quasi-hereditary algebra
$A$ is usually called a {\it BGG-algebra}, see for instance \cite{CPS2, I}.
\end{remark}

Idempotent condensation of a quasi-hereditary algebra results again in a quasi-hereditary algebra if the simple modules annihilated by the idempotent form a subset of $\Lambda$ that is closed from above. The following proposition makes this more precise. It is an immediate consequence of Green's idempotent condensation theory, \cite[Section~6.2]{Gr}, and \cite[Proposition A.3.11]{D}. We shall use it in Section~\ref{sec bisets} to show that, in certain situations, biset functor categories are equivalent to module categories of quasi-hereditary algebras.

\begin{proposition}\label{prop qh condensed}
Assume that $A$ is a quasi-hereditary $k$-algebra with respect to $(\Lambda,\le)$, and let $e$ be an idempotent satisfying the following condition: If $\lambda\le \mu$ are elements of $\Lambda$ and if $eD_\lambda\neq\{0\}$ then also $eD_\mu\neq\{0\}$. Set $\Lambda':=\{\lambda\in\Lambda\mid eD_\lambda\neq\{0\}\}$ and set $A':=eAe$. Then, as $\lambda$ varies over $\Lambda'$, the $A'$-modules $eD_\lambda$ form a complete set of representatives of the isomorphism classes of simple $A'$-modules, and $A'$ is a quasi-hereditary algebra with respect to $(\Lambda',\le)$. Moreover, for $\lambda\in\Lambda'$, the corresponding standard $A'$-module  is given by $e\Delta(\lambda)$, and the corresponding costandard $A'$-module is given by $e\nabla(\lambda)$.
\end{proposition}

%%%%%%%%%%%%%%%%%%%%%%%%%%%%%%%%%%%%%%%%%%%%%%%%%%%%%%%%%%

\section{Twisted Category Algebras}\label{sec twisted cat}

In the following we recall some properties of twisted category algebras from \cite{BDIII,LS,LS2}.
Throughout this paper we shall choose our notation in accordance with \cite{BDIII}. In particular, we shall use the following notation and situation repeatedly.

\begin{notation}\label{nota twisted cat alg}
(a)\, Let $\catC$ be a finite category, that is, the morphisms in $\catC$ form a finite set $S:=\Mor(\catC)$.
We shall henceforth suppose that $\catC$ is {\it split}, that is, for every morphism $s\in S$, there is some $t\in S$ with $s\circ t\circ s=s$.

Let $k$ be a field. A {\it $2$-cocycle $\alpha$ of $\catC$ with values in $k^\times$} assigns to each pair of morphisms $s,t\in S$ such that $s\circ t$ exists in $S$
an element $\alpha(s,t)\in k^\times$ such that the following holds: whenever $s,t,u\in S$ are such that $u\circ t\circ s$ exists, one has
$\alpha(u\circ t,s)\alpha(u,t)=\alpha(u,t\circ s)\alpha(t,s)$. The {\it twisted category algebra} $k_\alpha\catC$ is the $k$-vector space with $k$-basis $S$ and with multiplication
$$t\cdot s:=\begin{cases} 
\alpha(t,s)\cdot t\circ s &\text{ if } t\circ s \text{ exists,}\\
0 & \text{ otherwise,}
\end{cases}$$
for $s,t\in S$. For the remainder of this section, we set $A:=k_\alpha\catC$. The set of all 2-cocycles of $\catC$ with values in $k^\times$
will be denoted by $Z^2(\catC,k^\times)$. If $\catC'$ is a skeleton of $\catC$ and $\alpha'$ is the restriction of $\alpha$ to $\catC'$ then $k_{\alpha'}\catC'$ is Morita equivalent to $k_\alpha\catC$.

\smallskip
(b)\, Following Green in \cite{Gr1}, one has an equivalence relation $\scrJ$ on $S$ defined by
$$s\scrJ t:\Leftrightarrow S\circ s\circ S=S\circ t\circ S\,,$$
for $s,t\in S$. The corresponding equivalence class of $s$ is denoted by $\scrJ(s)$, and is called a {\it $\scrJ$-class of $\catC$}.
One also has a partial order $\leq_\scrJ$ on the set of $\scrJ$-classes of $\catC$ defined by
\begin{equation}\label{eqn order J-class}
\scrJ(s)\leq_\scrJ \scrJ(t):\Leftrightarrow S\circ \scrJ(s)\circ S\subseteq S\circ \scrJ(t)\circ S\,,
\end{equation}
for $s,t\in S$. 
Note also that $\scrJ(s)\leq_\scrJ\scrJ(t)$ if and only if $S\circ s\circ S\subseteq S\circ t\circ S$.
From now on, let $S_1,\ldots,S_n$ denote the $\scrJ$-classes of $\catC$, ordered such that $S_i<_\scrJ S_j$ implies $i<j$.

Since $\catC$ is split, every $\scrJ$-class $S_i$ of $\catC$ contains an idempotent endomorphism $e_i$, that is,
$e_i\in\End_\catC(X_i)$, for some $X_i\in\Ob(\catC)$, and $e_i\circ e_i=e_i$: one can, for instance, take $e_i:=s_i\circ t_i$, for any $s_i\in S_i$ and any $t_i\in S$ such that $s_i\circ t_i\circ s_i=s_i$. 
%Since $\catC$ is split, any two idempotents in $S_i$ are equivalent with respect to a certain equivalence relation that we do not need in the sequel.
%One thus has an equivalence relation on the
%set of idempotent endomorphisms of $\catC$ via
%$$e\sim f:\Leftrightarrow \scrJ(e)=\scrJ(f)\,.$$

Note that 
$$e_i':=\alpha(e_i,e_i)^{-1}\cdot e_i$$
is an idempotent in the algebra $A$, whereas $e_i$ itself is in general not.

\smallskip
(c)\, For $i=1,\ldots,n$, denote by $\Gamma_{e_i}$ the group of units in the monoid $e_i\circ \End_\catC(X_i)\circ e_i$, and set
$J_{e_i}:=(e_i\circ \End_\catC(X_i)\circ e_i)\smallsetminus \Gamma_{e_i}$. The 2-cocycle $\alpha$ restricts to a 2-cocycle of the group $\Gamma_{e_i}$, so that one can regard the twisted group algebra $k_\alpha\Gamma_{e_i}$ as a (non-unitary) subalgebra of $A$. Moreover, for each $i=1,\ldots,n$, one has the following $k$-vector space decomposition of the $k$-algebra $e_i'Ae_i'$:
\begin{equation}\label{eqn eAe}
e_i'Ae_i'= e_iAe_i=k_\alpha\Gamma_{e_i}\oplus kJ_{e_i}\,;
\end{equation}
note that here $k_\alpha\Gamma_{e_i}$ is a unitary subalgebra and $kJ_{e_i}$ is a two-sided ideal of $e_i'Ae_i'$.

\smallskip
(d) In accordance with \cite{BDIII}, we also define 
\begin{equation}\label{eqn J ideals}
S_{\leq i}:= \bigdisj_{j\leq i} S_j\quad\text{ and }\quad    J_i:=kS_{\leq i}\,,
\end{equation}
for $i=1,\ldots,n$, and we set $J_0:=\{0\}$. By \cite[Proposition~3.3]{BDIII}, this yields a chain $J_0\subset J_1\subset \cdots\subset J_n=A$ of two-sided ideals in $A$.
\end{notation}

\begin{remark}\label{rem regular}
In the special case where $\catC$ is a category with one object, one simply recovers the notion of a twisted monoid algebra. 
The property of being split is then usually called {\it regular}, see for instance \cite{GMS}.
\end{remark}

\begin{nothing}\label{noth simple modules}
{\bf Simple modules and standard modules.}\,
For each $i\in\{1,\ldots,n\}$, let $e_i\in S_i$ be an idempotent endomorphism, and let $T_{(i,1)},\ldots,T_{(i,l_i)}$ be representatives of the isomorphism classes of simple $k_\alpha\Gamma_{e_i}$-modules. For $i\in\{1,\ldots,n\}$ and $r\in\{1,\ldots,l_i\}$, denote
by $\tilde{T}_{(i,r)}$ the inflation of $T_{(i,r)}$ to $e_i'Ae_i'$ with respect to the ideal $kJ_{e_i}$ and the decomposition (\ref{eqn eAe}). Consider the $A$-modules 
\begin{equation}\label{eqn Dir}
\Delta_{(i,r)}:=Ae'_i\otimes_{e'_iAe'_i} \tilde{T}_{(i,r)}\quad\text{ and } \quad D_{(i,r)}:=\Hd(\Delta_{(i,r)})\,.
\end{equation}
The isomorphism classes of $\Delta_{(i,r)}$ and $D_{(i,r)}$, respectively, are independent of the choice of the idempotent
$e_i\in S_i$.
\end{nothing}

\begin{theorem}[\protect{\cite[Theorem~1.2]{LS}}]\label{thm LS simple}
The modules $D_{(i,r)}$ ($i=1,\ldots,n,\, r=1,\ldots,l_i$) form a set of representatives of the isomorphism classes of simple $A$-modules.
\end{theorem}

\begin{remark}\label{rem condense}
Suppose again that $i\in\{1,\ldots,n\}$ and $r\in\{1,\ldots,l_i\}$. By the general theory of idempotent condensation, see 
\cite[Section~6.2]{Gr}, one has the following isomorphisms of $e_i'Ae_i'$-modules:
\begin{equation}\label{eqn id condense}
e'_i\cdot D_{(i,r)}\cong \tilde{T}_{(i,r)}\cong e'_i\cdot \Delta_{(i,r)}\,;
\end{equation}
in particular, the idempotent $e_i'$ of $A$ annihilates every composition factor of $\Rad(\Delta_{(i,r)})$.
\end{remark}

\begin{nothing}\label{noth leq}
{\bf A partial order.}\, By Theorem~\ref{thm LS simple}, the set $\Lambda:=\{(i,r)\mid 1\leq i\leq n,\, 1\leq r\leq l_i\}$ parametrizes the isomorphism classes of simple $A$-modules. Moreover, one has a partial order $\leq$ on $\Lambda$ that is defined as follows:
\begin{equation}\label{eqn leq}
(i,r)< (j,s):\Leftrightarrow S_j<_\scrJ S_i\,,
\end{equation}
for $(i,r),(j,s)\in\Lambda$. With this notation we recall the following result.
\end{nothing}

\begin{theorem}[\protect{\cite[Theorem~4.2]{BDIII}}]\label{thm A qh}
Suppose that the group orders $|\Gamma_{e_1}|,\ldots,|\Gamma_{e_n}|$ are invertible in $k$.
Then the $k$-algebra $A=k_\alpha\catC$ is quasi-hereditary with respect to $(\Lambda,\leq)$. Moreover, for $(i,r)\in\Lambda$, the standard $A$-module labelled by $(i,r)$ is isomorphic to $\Delta_{(i,r)}$.
\end{theorem}

An independent proof of the quasi-heredity of $k_\alpha\catC$ can be found in \cite[Corollary~1.2]{LS2}.

%%%%%%%%%%%%%%%%%%%%%%%%%%%%%%%%%%%%%%%%%%%%%%%%%%%%%%%%%%%%%%%

\section{Main Theorem}\label{sec main}

Throughout this section, we retain the situation and notation from \ref{nota twisted cat alg} and \ref{noth simple modules}. Additionally, we assume that, for each $i=1,\ldots,n$, the order of $\Gamma_{e_i}$ is invertible in $k$, so that
the twisted group algebra $k_\alpha\Gamma_{e_i}$ is semisimple. By Theorem~\ref{thm A qh}, $A=k_\alpha\catC$ is quasi-hereditary with simple modules $D_{(i,r)}$ and standard modules $\Delta_{(i,r)}$, $(i,r)\in\Lambda$, where $\Lambda$ is endowed with the partial order $\le$ from \ref{noth leq}. 
 
 \smallskip
In the following we shall introduce a new partial order $\unlhd$ on $\Lambda$ such that the partial order $\leq$  defined in \ref{noth leq} is a refinement of $\unlhd$. We aim to show that $A$ is also quasi-hereditary with respect to $(\Lambda,\unlhd)$, and that the standard  and costandard modules do not change. 

\begin{definition}\label{defi new order}
For $(i,r),(j,s)\in\Lambda$, let $f_{(i,r)}\in k_\alpha\Gamma_{e_i}$ and $f_{(j,s)}\in k_\alpha\Gamma_{e_j}$ be 
the block idempotents corresponding to the simple modules $T_{(i,r)}$ and $T_{(j,s)}$, respectively. We
set
\begin{align*}
(i,r)\sqsubset (j,s):\Leftrightarrow & \text{ (i) } S_j<_{\scrJ} S_i, \text{ and }\\
                                                           & \text{ (ii) }  f_{(i,r)}\cdot J_{j}\cdot f_{(j,s)}\not\subseteq J_{j-1} \text{ or } f_{(j,s)}\cdot J_{j}\cdot f_{(i,r)}\not\subseteq J_{j-1}\,.
\end{align*}
Here $J_{j-1}\subset J_j$ are the two-sided ideals of $A$ defined in (\ref{eqn J ideals}).
This defines a reflexive, anti-symmetric relation $\sqsubseteq$ on $\Lambda$.
The transitive closure of the relation $\sqsubseteq$ on $\Lambda$ will be denoted by $\unlhd$. This is again a partial order on $\Lambda$.
\end{definition}

\begin{remark}\label{rem transitive}
(a)\, Note that the partial order $\leq $ on $\Lambda$ defined
in (\ref{eqn leq}) is indeed a refinement of the new partial order $\unlhd$.

\smallskip
(b)\, One can show that the partial order $\unlhd$ on $\Lambda$ does neither depend on the choice of the idempotent $e_i$ in $S_i$ nor on the chosen total order $S_1, S_2,\ldots, S_n$ of the $\scrJ$-classes.

\smallskip
(c)\, We emphasize that the relation $\sqsubseteq$ in Defintion~\ref{defi new order} is in general not transitive; we shall give explicit examples later in Section~\ref{sec bisets} and Section~\ref{sec brauer}.
The condition in Definition~\ref{defi new order}(ii) can be reformulated, and can often be simplified; see Sections~\ref{sec sqsubseteq} and \ref{sec duality}. In Section~\ref{sec duality} we shall, in particular, establish criteria for the conditions $f_{(i,r)}\cdot J_{j}\cdot f_{(j,s)}\not\subseteq J_{j-1}$ and
$ f_{(j,s)}\cdot J_{j}\cdot f_{(i,r)}\not\subseteq J_{j-1}$ to be equivalent.
\end{remark}

We are now prepared to state and prove our main result:

\begin{theorem}\label{thm main}
Suppose that the group orders $|\Gamma_{e_1}|,\ldots,|\Gamma_{e_n}|$ are invertible in $k$.
Then the twisted category algebra $A=k_\alpha\catC$ is quasi-hereditary with respect to $(\Lambda,\unlhd)$. Moreover, the $A$-modules $\Delta_{(i,r)}$ ($(i,r)\in\Lambda$), as defined in (\ref{eqn Dir}), are the corresponding standard modules.
\end{theorem}

\begin{proof}
We shall use Definition~\ref{defi qh} and proceed as in \cite{BDIII}. That is, we shall show that, for each $(i,r)\in\Lambda$,
the $A$-module $\Delta_{(i,r)}$ satisfies (i) and (ii) below, and that the projective $A$-module $P_{(i,r)}$ admits a filtration
$$0=P_{(i,r)}^{(0)}\subset\cdots\subset P_{(i,r)}^{(m_{ir})}=P_{(i,r)}$$
satisfying (iii) and (iv) below:

\medskip
(i)\, $\Hd(\Delta_{(i,r)})\cong D_{(i,r)}$;

\smallskip
(ii)\, $[\Rad(\Delta_{(i,r)}):D_{(l,t)}]\neq 0 \Rightarrow (l,t)\lhd (i,r)$;

\smallskip
(iii)\, $P_{(i,r)}^{(m_{ir})}/P_{(i,r)}^{(m_{ir}-1)}\cong \Delta_{(i,r)}$;

\smallskip
(iv)\, $1\leq q<m_{(i,r)} \Rightarrow P_{(i,r)}^{(q)}/P_{(i,r)}^{(q-1)}\cong \Delta_{(j,s)}$, for some $(j,s)\in\Lambda$ with $(i,r)\lhd (j,s)$.

\medskip
Condition~(i) has already been verified in the proof of \cite[Theorem~4.2]{BDIII}. To show (ii), let $(i,r), (l,t)\in\Lambda$ be such
$[\Rad(\Delta_{(i,r)}):D_{(l,t)}]\neq 0$. Since $\Delta_{(i,r)}$ is the standard $A$-module (see Theorem~\ref{thm A qh})with respect to $(\Lambda,\leq)$ labelled by $(i,r)$, we already know that $S_i<_\scrJ S_l$. We shall show that $f_{(l,t)}\cdot J_i\cdot f_{(i,r)}\not\subseteq J_{i-1}$, so that
$(l,t)\sqsubset (i,r)$, and thus $(l,t)\lhd (i,r)$.

Recall from (\ref{eqn id condense}) that $e_l'\cdot D_{(l,t)}\cong e_l'\cdot\Delta_{(l,t)}\cong \tilde{T}_{(l,t)}$ as $e_l'Ae_l'$-modules. Moreover, note that
$f_{(i,r)}\cdot\tilde{T}_{(i,r)}=\tilde{T}_{(i,r)}$, since $f_{(i,r)}$ is the block idempotent of $k_\alpha\Gamma_{e_i}$ corresponding to the simple module $T_{(i,r)}$ and $\tilde{T}_{(i,r)}$  is just the inflation of $T_{(i,r)}$ to $e_i'Ae_i'$. Analogously, we have $f_{(l,t)}\cdot \tilde{T}_{(l,t)}=\tilde{T}_{(l,t)}$. This implies, in particular, that $f_{(l,t)}\cdot D_{(l,t)}=f_{(l,t)}e_l'\cdot D_{(l,t)}\neq \{0\}$, and thus also $f_{(l,t)}\cdot \Delta_{(i,r)}\neq \{0\}$, since
we are assuming $[\Delta_{(i,r)}:D_{(l,t)}]\neq 0$ and since multiplication by $f_{(l,t)}$ is exact.

Now, writing $f_{(i,r)}$ as a sum of pairwise orthogonal primitive idempotents in $e_i'Ae_i'$, there is a summand $\tilde{f}_{(i,r)}$ in this decomposition such that $\tilde{f}_{(i,r)}\cdot \tilde{T}_{(i,r)}\neq \{0\}$. By \cite[Theorem~1.8.10]{NT}, this in turn implies that the $e_i'Ae_i'$-module $e_i'Ae_i'\cdot \tilde{f}_{(i,r)}$ is a projective cover of $\tilde{T}_{(i,r)}$. Thus, as we have seen in the proof of \cite[Theorem~4.2]{BDIII}, the $A$-module $A\cdot \tilde{f}_{(i,r)}$ is a projective cover of $D_{(i,r)}$, so that we may from now on suppose that $A\cdot \tilde{f}_{(i,r)}=P_{(i,r)}$.

Next consider the following chain of $A$-modules from \cite[(12)]{BDIII}:
\begin{equation}\label{eqn BDIII chain}
\{0\}=J_0\cdot \tilde{f}_{(i,r)}\subseteq J_1\cdot \tilde{f}_{(i,r)}\subseteq\cdots \subseteq J_{i-1}\cdot \tilde{f}_{(i,r)}\subseteq J_i\cdot \tilde{f}_{(i,r)}=A\cdot\tilde{f}_{(i,r)}=P_{(i,r)}\,.
\end{equation}
As we have shown in the proof of \cite[Theorem~ 4.2]{BDIII}, there is an $A$-module isomorphism
$$\Delta_{(i,r)}\cong J_i\cdot \tilde{f}_{(i,r)}/J_{i-1}\cdot \tilde{f}_{(i,r)}=(J_i/J_{i-1})\cdot \tilde{f}_{(i,r)}\,.$$
Consequently, we have now proved that
$$\{0\}\neq f_{(l,t)}\cdot (J_i/J_{i-1})\cdot \tilde{f}_{(i,r)}= f_{(l,t)}\cdot (J_i/J_{i-1})\cdot f_{(i,r)}\cdot \tilde{f}_{(i,r)}\,,$$
and therefore $f_{(l,t)}\cdot J_i\cdot f_{(i,r)}\not\subseteq J_{i-1}$, implying
$(l,t)\sqsubset (i,r)$, and thus also $(l,t)\lhd (i,r)$, as desired. This settles the proof of (ii).

\medskip

It remains to verify conditions (iii) and (iv). To this end, we consider the chain (\ref{eqn BDIII chain}) of $A$-modules again. 
Since we already know that $\Delta_{(i,r)}\cong J_i\cdot \tilde{f}_{(i,r)}/J_{i-1}\cdot \tilde{f}_{(i,r)}$, it suffices to show that
(\ref{eqn BDIII chain}) also satisfies (iv). We argue along the lines of the proof of \cite[Theorem~4.2]{BDIII}: if $j\in\{1,\ldots,i-1\}$ is such that $S_j\not<_\scrJ S_i$ then $J_j\cdot \tilde{f}_{(i,r)}=J_{j-1}\cdot \tilde{f}_{(i,r)}$. So suppose that $j\in\{1,\ldots,i-1\}$ is such that $S_j<_\scrJ S_i$, and let $M$ be an indecomposable direct summand of the $A$-module $J_j\cdot \tilde{f}_{(i,r)}/J_{j-1}\cdot \tilde{f}_{(i,r)}$. Then, by \cite[Lemma~4.4]{BDIII}, there is some $s\in\{1,\ldots,l_j\}$ with $M\cong\Delta_{(j,s)}$. Arguing as above, we deduce
$f_{(j,s)}\cdot \tilde{T}_{(j,s)}=\tilde{T}_{(j,s)}\neq \{0\}$, and $e_j'\cdot \Delta_{(j,s)}\cong \tilde{T}_{(j,s)}$ as $e_j'Ae_j'$-modules. Thus
$f_{(j,s)}\cdot \Delta_{(j,s)}\neq \{0\}$, and
$$\{0\}\neq f_{(j,s)}\cdot (J_j/J_{j-1})\cdot \tilde{f}_{(i,r)}=f_{(j,s)}\cdot (J_j/J_{j-1})\cdot f_{(i,r)}\cdot \tilde{f}_{(i,r)}\,,$$
so that $f_{(j,s)}\cdot J_j\cdot f_{(i,r)}\not\subseteq J_{j-1}$. Since we are assuming $S_j<_\scrJ S_i$, this implies $(i,r)\sqsubset (j,s)$, hence also $(i,r)\lhd (j,s)$, proving (iv).
\end{proof}

\begin{nothing}\label{noth loose ends}
{\bf Costandard modules.}\, 
At the end of this section we should like to tie up some loose ends and comment on the costandard $A$-modules with respect to the
partial orders $\leq$ and $\unlhd$ on $\Lambda$.

To this end, first note that if $B$ is any finite-dimensional $k$-algebra and $f$ is a primitive idempotent of $B$ then
\begin{equation}\label{eqn left right simples}
  \Hd(Bf)^*\cong \Hd(fB) \text{ as right $B$-modules and } \Hd(fB)^*\cong \Hd(Bf) \text{ as left $B$-modules.}
\end{equation}
In other words, if $D$ is a simple left $B$-module belonging to a particular Wedderburn component of $B/J(B)$ and if $D'$ denotes the simple right $B$-module belonging to the same Wedderburn component, then $D'\cong D^*$ as right $B$-modules.

Next, observe that everything that was proven in Theorem~\ref{thm main} for left modules over the twisted split category algebra $A$ can also be proved for right $A$-modules. More precisely, for $(i,r)\in\Lambda$, consider the simple right $k_\alpha\Gamma_{e_i}$-module $T'_{(i,r)}$ belonging to the same Wedderburn component as the simple left $k_\alpha\Gamma_{e_i}$-module $T_{(i,r)}$ and let $\Ttilde'_{(i,r)}$ denote the inflation of $T'_{(i,r)}$ to $e'_iAe'_i$ with respect to the decomposition $e'_iAe'_i=k_\alpha\Gamma_{e_i}\oplus kJ_{e_i}$. Moreover, define the right $A$-modules
\begin{equation}\label{eqn Delta'}
  \Delta'_{(i,r)}:=\Ttilde'_{(i,r)}\otimes_{e'_iAe'_i} e'_iA \quad\text{and}\quad D'_{(i,r)}:=\Hd(\Delta'_{(i,r)})\,.
\end{equation}
Then, in analogy to \cite[Theorem~1.2]{LS}, the modules $D'_{(i,r)}$, $(i,r)\in\Lambda$, form a complete set of representatives of the isomorphism classes of simple right $A$-modules. Moreover, with respect to both the partial orders $\leq$ and $\unlhd$ on $\Lambda$, the standard right $A$-module labelled by $(i,r)\in\Lambda$ is given by $\Delta'_{(i,r)}$.
As well, $A$ is quasi-hereditary with respect to these partial orders in the formulation for right modules, in a symmetric sense to Definition~\ref{defi qh}. 

\smallskip
Note also that applying (\ref{eqn left right simples}) to the algebras $B=k_{\alpha}\Gamma_{e_i}$ and to $B=e_i'Ae_i'$, respectively, shows that
\begin{equation}\label{eqn T'=T^*}
  T'_{(i,r)}\cong T^*_{(i,r)} \text{ as right $k_\alpha\Gamma_{e_i}$-modules, and } 
  \Ttilde'_{(i,r)}\cong (\Ttilde_{(i,r)})^* \text{ as right $e_i'Ae_i'$-modules.}
\end{equation}
Note also that $(\Ttilde_{(i,r)})^*=\widetilde{T_{(i,r)}^*}$ as right $e_i'Ae_i'$-module, so that we may simply denote this module by 
$\Ttilde_{(i,r)}^*$, to avoid too many symbols.

Finally, assume that $\ftilde_{(i,r)}$ is a primitive idempotent of $e_i'Ae_i'$ occurring in the decomposition of the block
idempotent $f_{(i,r)}$ with $\ftilde_{(i,r)} \Ttilde_{(i,r)}\neq\{0\}$. Then also $\widetilde{T'}_{(i,r)}\ftilde_{(i,r)}\neq \{0\}$. From \cite[Theorem~4.2]{BDIII} we know that $A\ftilde_{(i,r)}$ is a projective cover of $D_{(i,r)}$, and, arguing completely analogously with right modules instead of left modules, we also deduce that $\ftilde_{(i,r)}A$ is a projective cover of $D'_{(i,r)}$. Thus, by (\ref{eqn left right simples}), we obtain
\begin{equation}\label{eqn D'=D^*}
  D'_{(i,r)}\cong D_{(i,r)}^* \text{ as right $A$-modules.}
\end{equation}
\end{nothing}

Consequently, we have:

\begin{corollary}\label{cor main}
Suppose that the group orders $|\Gamma_{e_1}|$, \ldots, $|\Gamma_{e_n}|$ are invertible in $k$. For $(i,r)\in\Lambda$, the corresponding standard right $A$-module, with respect to both $\leq$ and $\unlhd$, is given by $\Delta'_{(i,r)}$ in (\ref{eqn Delta'}). Furthermore, with respect to both $\leq$ and $\unlhd$, the costandard left $A$-modules $\nabla_{(i,r)}$ and the costandard right $A$-modules $\nabla'_{(i,r)}$ are given by 
\begin{equation*}
  \nabla_{(i,r)}\cong\Hom_{e'_iAe'_i}(e'_iA,\Ttilde_{(i,r)}) \quad\text{and}\quad 
  \nabla'_{(i,r)}\cong\Hom_{e'_iAe'_i}(Ae'_i,\Ttilde_{(i,r)}^*)\,.
\end{equation*}
Moreover,
\begin{equation*}
  \Delta'_{(i,r)}\cong \nabla_{(i,r)}^*\quad \text{and}\quad \nabla'_{(i,r)}\cong\Delta_{(i,r)}^*
\end{equation*}
as right $A$-modules.
\end{corollary}

\begin{proof}
The first statement has already been derived in \ref{noth loose ends}. 
Now let $(i,r)\in\Lambda$. Again, by the considerations in \ref{noth loose ends}, the module $\Delta'_{(i,r)}$ is the right standard $A$-modules with respect to $\leq$ and  $\unlhd$. Let $I_{(i,r)}$ denote an injective envelope of the simple left $A$-module $D_{(i,r)}$. Then $I_{(i,r)}^*$ is a projective cover of the simple right $A$-module $D_{(i,r)}^*$, and thus of $D'_{(i,r)}$, by (\ref{eqn D'=D^*}). Therefore, $\nabla_{(i,r)}^*$ is a quotient of $I^*_{(i,r)}$ such that each composition factor of $\Rad(\nabla_{(i,r)}^*)$ is of the form $D_{(j,s)}^*\cong D'_{(j,s)}$ with $(j,s)\lhd (i,r)$, and $\nabla_{(i,r)}^*$ is the largest such quotient. For otherwise, by taking duals again, $I_{(i,r)}$ would have a submodule $M$ strictly containing $\nabla_{(i,r)}$ such that  the composition factors of $M/\Soc(M)$ are of the form $D_{(j,s)}$ with $(j,s)\lhd (i,r)$, which is not the case. Therefore, both $\Delta'_{(i,r)}$ and $\nabla_{(i,r)}^*$ are standard right $A$-modules with respect to $\leq$ and $\unlhd$, with head isomorphic to $D_{(i,r)}'\cong D_{(i,r)}^*$. But this implies $\Delta'_{(i,r)}\cong \nabla_{(i,r)}^*$. Similarly, one shows that $\Delta_{(i,r)}=(\nabla'_{(i,r)})^*$.

Finally, by the usual adjunction isomorphism and (\ref{eqn T'=T^*}), we obtain that
\begin{align*}
  \nabla_{(i,r)}&\cong(\Delta'_{(i,r)})^*= \Hom_k(\Ttilde'_{(i,r)}\otimes_{e'_iAe'_i} e'_iA, k) \cong 
  \Hom_{e'_iAe'_i}(e'_iA,\Hom_k(\Ttilde'_{(i,r)},k))\\
                        & \cong \Hom_{e'_iAe'_i}(e'_iA,\Ttilde_{(i,r)})
\end{align*}
as left $A$-modules, and similarly one obtains an isomorphism $\nabla'_{(i,r)}\cong\Hom_{e'_iAe'_i}(Ae'_i,\Ttilde_{(i,r)}^*)$ of right $A$-modules.
\end{proof}

%%%%%%%%%%%%%%%%%%%%%%%%%%%%%%%%%%%%%%%%%%%%%%%%%%%%%%%%

\section{Reformulations of the relation $\sqsubseteq$}\label{sec sqsubseteq}

We retain the notation from Section~\ref{sec main}. Thus, we assume the notation and situation from \ref{nota twisted cat alg} and \ref{noth simple modules}, and also assume that, for every idempotent  endomorphism $e$ in $\catC$, the group order $|\Gamma_e|$ is invertible in $k$. Thus, the corresponding twisted group algebra $k_\alpha\Gamma_e$ will then again be semisimple.
For every $(i,r)\in\Lambda$, let $f_{(i,r)}$ denote the primitive central idempotent of $k_\alpha\Gamma_{e_i}$ satisfying 
$f_{(i,r)}T_{(i,r)}\neq \{0\}$, or equivalently, $T_{(i,r)}^*f_{(i,r)}\neq\{0\}$. The following proposition gives equivalent reformulations of one of the two conditions in Definition~\ref{defi new order}(ii), concerning the relation $\sqsubseteq$ and the resulting partial order $\unlhd$ on $\Lambda$.

\begin{proposition}\label{prop cond (ii)}
Let $i,j\in\{1,\ldots,n\}$ be such that $S_j<_\scrJ S_i$. Then, for $(i,r),(j,s)\in\Lambda$, the following are equivalent:

\smallskip
{\rm (i)}\, $f_{(i,r)}\cdot J_j\cdot f_{(j,s)}\not\subseteq J_{j-1}$;

\smallskip
{\rm (ii)}\, there is some $t\in S_j\cap e_i\circ S\circ e_j$ with $f_{(i,r)}\cdot t\cdot f_{(j,s)}\neq 0$ in $A$;

\smallskip
{\rm (iii)}\, there exists a non-zero $(k_\alpha\Gamma_{e_i}, k_\alpha\Gamma_{e_j})$-bimodule homomorphism from $f_{(i,r)}k_{\alpha}\Gamma_{e_i}\otimes f_{(j,s)}k_\alpha\Gamma_{e_j}$ to $e_i' (A/J_{j-1}) e_j'$;

\smallskip
{\rm (iv)}\,  there exists a non-zero $(k_\alpha\Gamma_{e_i}, k_\alpha\Gamma_{e_j})$-bimodule homomorphism from $T_{(i,r)}\otimes T_{(j,s)}^*$ to $e_i' (A/J_{j-1}) e_j'$;

\smallskip
{\rm (v)}\, $T_{(i,r)}\otimes T_{(j,s)}^*$ is isomorphic to a direct summand of the $(k_\alpha\Gamma_{e_i}, k_\alpha\Gamma_{e_j})$-bimodule $e_i' (A/J_{j-1}) e_j'$;

\smallskip
{\rm (vi)}\, $f_{(i,r)}\cdot (A/J_{j-1})\cdot f_{(j,s)}\neq \{0\}$.
\end{proposition}

\begin{proof}
(i) $\iff$ (ii): Note first that we have
\begin{equation*}
  f_{(i,r)}\cdot J_j\cdot f_{(j,s)}\not\subseteq J_{j-1}\Leftrightarrow 
  \exists\, t\in S_{\leq j}: f_{(i,r)}\cdot t\cdot f_{(j,s)}\notin J_{j-1}\,.
\end{equation*}
So, suppose that $t\in S_{\leq j}$ is such that $ f_{(i,r)}\cdot t\cdot f_{(j,s)}\notin J_{j-1}$. Then $t':=e_i\circ t\circ e_j$ is such that $t'\in e_i\circ S\circ e_j\cap S_{\leq j}$ and $f_{(i,r)}\cdot t'\cdot f_{(j,s)}=f_{(i,r)}\cdot t\cdot f_{(j,s)}\notin J_{j-1}$. But if $t'$ was contained in some $\scrJ$-class $S_l$ with $l<j$ then we would get $f_{(i,r)}\cdot t'\cdot f_{(j,s)}\in J_{j-1}$, which is not the case. Thus $t'\in e_i\circ S\circ e_j\cap S_j$, which gives
\begin{equation}\label{eqn s}
f_{(i,r)}\cdot J_j\cdot f_{(j,s)}\not\subseteq J_{j-1}\Leftrightarrow\exists\, t\in S_j\cap e_i\circ S\circ e_j \colon f_{(i,r)}\cdot t\cdot f_{(j,s)}\notin J_{j-1}\,.
\end{equation}
%The idempotent $f_{(i,r)}$ is a $k$-linear combination of elements in $\Gamma_{e_i}$, and the idempotent $f_{(j,s)}$ is a $k$-linear combination of elements in $\Gamma_{e_j}$. Hence if $s\in S_j\cap e_i\circ S\cap e_j$ then
Now let $t\in S_j\cap e_i\circ S\circ e_j$. Then $f_{(i,r)}\cdot t\cdot f_{(j,s)}$ is a $k$-linear combination of elements of the form
$t_i\circ t\circ t_j$, for suitable $t_i\in \Gamma_{e_i}$ and $t_j\in\Gamma_{e_j}$. Fix such elements $t_i$ and $t_j$. Then we have
$$S\circ t\circ S= S\circ e_i\circ t\circ e_j\circ S=S\circ t_i^{-1}\circ t_i\circ t\circ t_j\circ t_j^{-1}\circ S\subseteq S\circ t_i\circ t\circ t_j\circ S\subseteq S\circ t\circ S,$$
hence $\scrJ(t_i\circ t\circ t_j)=\scrJ(t)=S_j$, that is, $f_{(i,r)}\cdot t\cdot f_{(j,s)}\in kS_j$. Since $kS_j\cap J_{j-1}=\{0\}$, this finally yields
$$f_{(i,r)}\cdot J_j\cdot f_{(j,s)}\not\subseteq J_{j-1}\Leftrightarrow\exists\, t\in S_j\cap e_i\circ S\circ e_j \colon f_{(i,r)}\cdot t\cdot f_{(j,s)}\neq 0\,,$$
proving the equivalence of (i) and (ii).

\smallskip
(i) $\Rightarrow$ (iii): Let $a\in f_{(i,r)}\cdot J_j\cdot f_{(j,s)}\smallsetminus J_{j-1}$. Then the map 
$$\varphi\colon f_{(i,r)}\cdot k_\alpha\Gamma_{e_i}\otimes f_{(j,s)}\cdot  k_\alpha\Gamma_{e_j}\to f_{(i,r)}\cdot (A/J_{j-1})\cdot f_{(j,s)}\subseteq e_i'(A/J_{j-1})e_j'$$
given by $\varphi(x\otimes y)=xay +J_{j-1}$, for $x\in f_{(i,r)}\cdot k_\alpha\Gamma_{e_i}$
and $y\in f_{(j,s)}\cdot  k_\alpha\Gamma_{e_j}=k_\alpha\Gamma_{e_j}\cdot f_{(j,s)}$, is a $(k_\alpha\Gamma_{e_i},k_\alpha\Gamma_{e_j})$-bimodule homomorphism with $\varphi(f_{(i,r)}\otimes f_{(j,s)})=f_{(i,r)}\cdot a\cdot f_{(j,s)}+J_{j-1} = a + J_{j-1} \neq J_{j-1}$. 

\smallskip
(iii) $\Rightarrow$ (i): Suppose that $\varphi$ is a non-zero $(k_\alpha\Gamma_{e_i},k_\alpha\Gamma_{e_j})$-bimodule homomorphism from $f_{(i,r)}\cdot k_\alpha\Gamma_{e_i}\otimes f_{(j,s)}\cdot k_\alpha\Gamma_{e_j}$ to $e_i' (A/J_{j-1}) e_j'$. Then $0\neq\varphi(f_{(i,r)}\otimes f_{(j,s)})\in f_{(i,r)}\cdot (A/J_{j-1})\cdot f_{(j,s)}\subseteq e_i'(A/J_{j-1})e_j'$. Note that $Af_{(j,s)}\subseteq Ae_j'f_{(j,s)}\subseteq J_jf_{(j,s)}$. Thus also
$f_{(i,r)}\cdot A\cdot f_{(j,s)}=f_{(i,r)}\cdot J_j\cdot f_{(j,s)}$, and we obtain $f_{(i,r)}\cdot J_j\cdot f_{(j,s)}\not\subseteq J_{j-1}$. 

\smallskip
(iii) $\iff$ (iv) $\iff$ (v): 
Since $k_\alpha\Gamma_{e_i}$ and $k_{\alpha}\Gamma_{e_j}$ are semisimple $k$-algebras, the category of $(k_\alpha\Gamma_{e_i},k_\alpha\Gamma_{e_j})$-bimodules is semisimple and the bimodule $f_{(i,r)}k_{\alpha}\Gamma_{e_i}\otimes f_{(j,s)}k_\alpha\Gamma_{e_j}$ is isomorphic to a direct sum of copies of $T_{(i,r)}\otimes T_{(j,s)}^*$. The assertions (iii)--(v) are now clearly equivalent.

\smallskip
(i) $\iff$ (vi): This follows immediately from $f_{(i,r)}\cdot A \cdot f_{(j,s)}=f_{(i,r)}\cdot J_j \cdot f_{(j,s)}$.
\end{proof}

\noindent
Similarly to the proof of the previous proposition one proves the following symmetric result:

\begin{proposition}\label{prop cond (ii)'}
 Let again $i,j\in\{1,\ldots,n\}$ be such that $S_j<_\scrJ S_i$ and let $(i,r),(j,s)\in\Lambda$. Then the following are equivalent:

\smallskip
{\rm (i)}\, $f_{(j,s)}\cdot J_j\cdot f_{(i,r)}\not\subseteq J_{j-1}$;

\smallskip
{\rm (ii)}\, there is some $t\in S_j\cap e_j\circ S\circ e_i$ with $f_{(j,s)}\cdot t\cdot f_{(i,r)}\neq 0$ in $A$;

\smallskip
{\rm (iii)}\, there exists a non-zero $(k_\alpha\Gamma_{e_j}, k_\alpha\Gamma_{e_i})$-bimodule homomorphism from $f_{(j,s)}k_{\alpha}\Gamma_{e_j}\otimes f_{(i,r)}k_\alpha\Gamma_{e_i}$ to $e_j' (A/J_{j-1}) e_i'$;

\smallskip
{\rm (iv)}\,  there exists a non-zero $(k_\alpha\Gamma_{e_j}, k_\alpha\Gamma_{e_i})$-bimodule homomorphism from $T_{(j,s)}\otimes T_{(i,r)}^*$ to $e_j' (A/J_{j-1}) e_i'$;

\smallskip
{\rm (v)}\, $T_{(j,s)}\otimes T_{(i,r)}^*$ is isomorphic to a direct summand of the $(k_\alpha\Gamma_{e_j}, k_\alpha\Gamma_{e_i})$-bimodule $e_j' (A/J_{j-1}) e_i'$;

\smallskip
{\rm (vi)}\, $f_{(j,s)}\cdot (A/J_{j-1})\cdot f_{(i,r)}\neq \{0\}$.
\end{proposition}

\begin{remark}\label{rem trivial alpha}
Suppose that $i\in\{1,\ldots,n\}$ and that the restriction of $\alpha$ to $\Gamma_{e_i}$ is a coboundary, that is, that there exists a function $\mu\colon \Gamma_{e_i}\to k^\times$ such that $\alpha(t,t')=\mu(t)\mu(t')\mu(t\circ t')^{-1}$,  for all  $t,t'\in\Gamma_{e_i}$. Then we have a $k$-algebra isomorphism 
\begin{equation}\label{eqn trivial alpha}
k_\alpha\Gamma_{e_i}\myiso k\Gamma_{e_i}\,,\quad t\mapsto \mu(t)\cdot t\,,
\end{equation}
where $t\in\Gamma_{e_i}$, and where $k\Gamma_{e_i}$ denotes the (untwisted) group algebra of $\Gamma_{e_i}$ over $k$. In particular, this holds if the restriction of $\alpha$ to $\Gamma_{e_i}$ is the constant function with some value $a\in k^\times$. In this case $\mu$ can also be chosen to be the constant function with value $a$. 

\smallskip
If also $j\in\{1,\ldots,n\}$ and if the restriction of $\alpha$ to $\Gamma_{e_j}$ is a coboundary then we have 
$k$-algebra isomorphisms
\begin{equation}\label{eqn group algebra}
k_\alpha\Gamma_{e_i}\otimes k_\alpha\Gamma_{e_j}\cong k\Gamma_{e_i}\otimes k\Gamma_{e_j}
\cong k[\Gamma_{e_i}\times\Gamma_{e_j}]\,.
\end{equation}
We may and shall then identify every left $k_\alpha\Gamma_{e_i}\otimes k_\alpha\Gamma_{e_j}$-module with a
$k[\Gamma_{e_i}\times\Gamma_{e_j}]$-module. Moreover, we shall always identify every $(k\Gamma_{e_i},k\Gamma_{e_j})$-bimodule $M$ with the left $k[\Gamma_{e_i}\times \Gamma_{e_j}]$-module $M$ defined by $(x,y)\cdot m:= x\cdot m\cdot y^{-1}$, for $m\in M$, $x\in \Gamma_{e_i}$, $y\in\Gamma_{e_j}$. Note that, under these identifications, the $(k_\alpha\Gamma_{e_i},k_\alpha\Gamma_{e_j})$-bimodule $T_{(i,r)}\otimes T_{(j,s)}^*$ becomes the left $k[\Gamma_{e_i}\times \Gamma_{e_j}]$-module 
$T_{(i,r)}\otimes T_{(j,s)}^*$.

\smallskip
In our applications to biset functors and Brauer algebras we shall see that the restrictions of the relevant 2-cocycles to the maximal subgroups $\Gamma_{e_i}$ of the respective categories are always $2$-coboundaries (in fact, even constant). Before simplifying the first condition in Definition~\ref{defi new order}(ii) further in this situation, we introduce a last bit of notation.
\end{remark}

\begin{notation}\label{nota perm action}
Let $i,j\in\{1,\ldots,n\}$ be such that $S_j<_\scrJ S_i$. Then the set $S_j\cap e_i\circ S\circ e_j$ carries a left $\Gamma_{e_i}\times\Gamma_{e_j}$-set structure via
$$(x,y)\cdot t:=x\circ t\circ y^{-1}\quad (x\in\Gamma_{e_i},\, y\in\Gamma_{e_j},\, t\in S_j\cap e_i\circ S\circ e_j)\,.$$
In fact, clearly $x\circ t\circ y^{-1} \in e_i\circ S\circ e_j$, and $x\circ t\circ y^{-1}\in S_j=\scrJ(t)$, since $x\circ  t\circ y^{-1}\in S\circ t\circ S$ and $t=x^{-1}\circ x\circ t\circ y^{-1}\circ y\in S\circ x\circ t\circ y^{-1} \circ S$.
We denote the stabilizer of $t\in S_j\cap e_i\circ S\circ e_j$ by $\stab_{\Gamma_{e_i}\times\Gamma_{e_j}}(t)$, or simply
by $\stab(t)$ when no confusion concerning the groups is possible.

Note that, analogously, the set $S_j\cap e_j\circ S\circ e_i$ carries a left $\Gamma_{e_j}\times \Gamma_{e_i}$-module structure.
\end{notation}

\begin{corollary}\label{cor cond (ii)}
Let $i,j\in\{1,\ldots,n\}$ be such that $S_j<_\scrJ S_i$. Suppose that $\alpha$ restricts to constant
$2$-cocycles on $\Gamma_{e_i}$ and on $\Gamma_{e_j}$ with values $a_i$ and $a_j$, respectively. Moreover, let $(i,r),(j,s)\in\Lambda$.

\smallskip
{\rm (a)}\, Assume that $\alpha(x,t)=a_i$ and $\alpha(t,y)=a_j$ for all $x\in \Gamma_{e_i}$, $t\in S_j\cap e_i\circ S\circ e_j$, and $y\in \Gamma_{e_j}$. Then one has $f_{(i,r)}\cdot J_j\cdot f_{(j,s)}\not\subseteq J_{j-1}$ if and only if there is some $t\in S_j\cap e_i\circ S\circ e_j$ such that $T_{(i,r)}\otimes T_{(j,s)}^*$ is isomorphic to a direct summand of the permutation $k[\Gamma_{e_i}\times\Gamma_{e_j}]$-module $\Ind_{\stab(t)}^{\Gamma_{e_i}\times\Gamma_{e_j}}(k)$;

\smallskip
{\rm (b)}\, Assume that $\alpha(y,t)=a_j$ and $\alpha(t,x)=a_i$ for all $y\in \Gamma_{e_j}$, $t\in S_j\cap e_j\circ S\circ e_i$, and $x\in \Gamma_{e_i}$. Then one has $f_{(j,s)}\cdot J_j\cdot f_{(i,r)}\not\subseteq J_{j-1}$ if and only if there is some $t\in S_j\cap e_j\circ S\circ e_i$ such that $T_{(j,s)}\otimes T_{(i,r)}^*$ is isomorphic to a direct summand of the permutation $k[\Gamma_{e_j}\times\Gamma_{e_i}]$-module $\Ind_{\stab(t)}^{\Gamma_{e_j}\times\Gamma_{e_i}}(k)$.
\end{corollary}

\begin{proof}
We identify every $(k_\alpha\Gamma_{e_i}, k_\alpha\Gamma_{e_j})$-bimodule with a $k[\Gamma_{e_i}\times\Gamma_{e_j}]$-module as indicated in Remark~\ref{rem trivial alpha}. To prove (a), we first observe again that $J_je'_j= Ae'_j$, and we deduce that the cosets in $A/J_{j-1}$ of the elements of
$S_j\cap e_i\circ S\circ e_j$ form in fact a $k$-basis of $e_i'(A/J_{j-1})e_j'$. The hypothesis in (a) implies that under the resulting $k[\Gamma_{e_i}\times \Gamma_{e_j}]$-module structure, this $k$-basis is permuted by $\Gamma_{e_i}\times\Gamma_{e_j}$ in the same way as in \ref{nota perm action}. Thus, we obtain a $k[\Gamma_{e_i}\times\Gamma_{e_j}]$-module isomorphism 
$$e_i'(A/J_{j-1})e_j'\cong \bigoplus_t\Ind_{\stab(t)}^{\Gamma_{e_i}\times\Gamma_{e_j}}(k)\,,$$
where $t$ varies over a set of representatives of the $\Gamma_{e_i}\times\Gamma_{e_j}$-orbits on 
$S_j\cap e_i\circ S\circ e_j$. Now assertion (a) follows from Proposition~\ref{prop cond (ii)}(v).

Assertion (b) is proved analogously using Proposition~\ref{prop cond (ii)'}.
\end{proof}

%One obtains a similar result as in the above corollary for the condition $f_{(j,s)}\cdot J_j\cdot f_{(i,r)}\not\subseteq J_{j-1}$ in terms of the $\Gamma_{e_j}\times\Gamma_{e_i}$-set $S_J\cap e_j\circ S\circ e_i$ by using the version of Proposition~\ref{prop cond (ii)} formulated in Proposition~\ref{prop cond (ii)'}.

%%%%%%%%%%%%%%%%%%%%%%%%%%%%%%%%%%%%%%%%%%%%%%%%%%%%%%%%

\section{Duality in Twisted Category Algebras}\label{sec duality}

Again, we retain the notation from Section~\ref{sec main}. Thus, we assume the notation and situation from \ref{nota twisted cat alg} and \ref{noth simple modules}, and also assume that, for every idempotent  endomorphism $e$ in $\catC$, the group order $|\Gamma_e|$ is invertible in $k$, so that Theorem~\ref{thm main} applies.

%Thus, by Theorem~\ref{thm main}, the corresponding twisted group algebra $k_\alpha\Gamma_e$ will again be semisimple.

\smallskip
In this section we shall show that the partial orders $\le$ and $\unlhd$ defined in \ref{noth leq} and in Definition~\ref{defi new order} behave well under a natural notion of duality introduced in Hypotheses~\ref{hypo op}. This will allow us to apply Proposition~\ref{prop dual standard modules}. If $\alpha$ restricts to particular coboundaries on the groups $\Gamma_{e_i}$, then we shall show that the two conditions in Definition~\ref{defi new order}(ii)  are equivalent.

\smallskip
We shall need the following hypotheses, which will be satisfied in many instances, and, in particular, in the two applications we are interested in; see Sections~\ref{sec bisets} and \ref{sec brauer}.

\begin{hypotheses}\label{hypo op}
For the remainder of this section, suppose that there is a contravariant functor $-^\circ:\catC\to\catC$ satisfying the following properties:

\smallskip
(i)\, $X^\circ=X$, for every $X\in\Ob(\catC)$;

\smallskip
(ii)\, $(s^\circ)^\circ=s$,  for every $s\in \Mor(\catC)=S$;

\smallskip
(iii)\, $s\circ s^\circ \circ s=s$, for every $s\in \Mor(\catC)=S$;

%(iv)\, $(s\circ t)^\circ=t^\circ\circ s^\circ$, for every $s,t\in S$;

\smallskip
(iv)\, $\alpha(s,t)=\alpha(t^\circ,s^\circ)$,  for all $s,t\in S$;

\smallskip
(v)\, $e_i^\circ=e_i$, for $i=1,\ldots,n$.

\smallskip\noindent
Given any subset $M$ of $S$, we  set $M^\circ:=\{s^\circ: s \in M\}$.
\end{hypotheses}

\begin{remark}\label{rem op}
Note that, given a contravariant functor $-^\circ:\catC\to\catC$ with properties (i)--(iv) above, we can always choose the idempotents
$e_1,\ldots,e_n$ such that they satisfy (v), by setting $e_i:=s_i\circ s_i^\circ$, for any $s_i\in S_i$.
\end{remark}

As immediate consequences of Hypotheses~\ref{hypo op} we obtain the following, which will be used repeatedly throughout
this article.

\begin{lemma}\label{lemma anti}
{\rm (a)}\, The functor $-^\circ:\catC\to\catC$ induces a $k$-algebra anti-involution
$$-^\circ:A\to A,\quad \sum_{s\in S} a_s s\mapsto \sum_{s\in S} a_s s^\circ\,,$$
where $a_s\in k$, for $s\in S$.

\smallskip
{\rm (b)}\, For $i\in\{1,\ldots,n\}$, one has $S_i=S_i^\circ$, and thus $J_i=J_i^\circ$, where $J_i$ is the ideal in $A$ defined in (\ref{eqn J ideals}).

\smallskip
{\rm (c)}\, For $i\in\{1,\ldots,n\}$ and $x\in \Gamma_{e_i}$, one has $x^\circ=x^{-1}$.
\end{lemma}

\begin{proof}
Part~(a) is a straightforward calculation. For~(b) note that Hypothesis~\ref{hypo op}(iii) implies $\scrJ(s^\circ)\leq_{\scrJ} \scrJ(s)$, and (ii) then implies $\scrJ(s)\leq_{\scrJ} \scrJ(s^\circ)$. Thus $\scrJ(s)=\scrJ(s^\circ)$, and both assertions in (b) are immediate from this.

For (c), let $i\in\{1,\ldots,n\}$, and let $X_i\in\Ob(\catC)$ be such that $e_i\in \End_\catC(X_i)$.
Then, for $x\in\Gamma_{e_i}$, we have $x^\circ=(e_i\circ x\circ e_i)^\circ=e_i^\circ\circ x^\circ\circ e_i^\circ=e_i\circ x^\circ \circ e_i$, thus
$x^\circ\in e_i\circ \End_\catC(X_i)\circ e_i$. Since $x\in\Gamma_{e_i}$, there is some $y\in\Gamma_{e_i}$ with $x\circ y=e_i=y\circ x$.
Hence 
$$y^\circ\circ x^\circ=e_i^\circ=x^\circ\circ y^\circ\,,$$
so that also $x^\circ\in\Gamma_{e_i}$, since $e_i^\circ=e_i$. Moreover, by Hypothesis~\ref{hypo op}(iii), $x\circ x^\circ$ is then an idempotent in the group $\Gamma_{e_i}$, implying
$x\circ x^\circ=e_i$, thus $x^\circ=x^{-1}$.
\end{proof}

\begin{nothing}\label{noth duals for A}
{\bf Dual $A$-modules.}\, (a)\, We apply the conventions from \ref{noth duals} to the $k$-algebra involution $-^\circ \colon A\to A$ from Lemma~\ref{lemma anti}(a). Thus, whenever $M$ is a left $A$-module, its $k$-linear dual becomes a left $A$-module $M^\circ$, via (\ref{eqn dual module}). 

\smallskip
(b)\, Suppose that $B$ is a (not necessarily unitary) $k$-subalgebra of $A$ such that $B^\circ=B$. Then the $k$-algebra anti-involution $-^\circ \colon A\to A$ restricts to a $k$-algebra anti-involution of $B$, and thus also the $k$-linear dual of every left $B$-module $N$ becomes a left $B$-module, which we again denote by $N^\circ$. 

By our Hypothesis~\ref{hypo op}(v) and Lemma~\ref{lemma anti}(b),  this is, in particular, satisfied if $B$ is one of the algebras $e_i'Ae_i'$ or $k_\alpha\Gamma_{e_i}$, for $i=1,\ldots,n$.

%If also $C$ is a $k$-subalgebra of $A$ then one can view every left $C\otimes B$-module $M$ as a $(C,B)$-bimodule via
%\begin{equation}\label{eqn bimodules}
%c\cdot m\cdot b:=(c\otimes b^\circ)\cdot m\quad (c\in C,\, b\in B,\, m\in M)\,.
%\end{equation}
%Analogously, every $(C,B)$-bimodule can be viewed as a left $C\otimes B$-module.

\smallskip
(c)\,  Again suppose that $B$ is a $k$-subalgebra of $A$ such that $B^\circ=B$. Let $M$ be a left $B$-module, and let $f$ be a central idempotent in $B$. Then $f^\circ$ is also a central idempotent in $B$, and one easily checks that the restriction map
\begin{equation}\label{eqn dual iso}
f^\circ\cdot M^\circ\to (f\cdot M)^\circ\,, \quad \varphi\mapsto \varphi|_{f\cdot M}\,,
\end{equation}
defines a left $B$-module isomorphism.

So, in particular, if $B=k_\alpha\Gamma_{e_i}$, for some $i\in\{1,\ldots,n\}$, and if $f_{(i,r)}$ is the block idempotent of 
$k_\alpha\Gamma_{e_i}$ corresponding to the simple module $T_{(i,r)}$ then $f_{(i,r)}^\circ$ is the block idempotent of $k_\alpha\Gamma_{e_i}$ corresponding to the simple module $T_{(i,r)}^\circ$.

\smallskip
(d)\, Suppose now that $f\in A$ is an idempotent such that $f^\circ=f$, let $B:=fAf$, and let $M$ be a left $A$-module.
 Then the restriction map
\begin{equation}\label{eqn dual iso'}
  f\cdot M^\circ \mapsto (f\cdot M)^\circ\,, \quad \varphi\mapsto \varphi|_{f\cdot M}\,,
\end{equation}
is a left $B$-module isomorphism.
%Suppose now that $i\in\{1,\ldots,n\}$ and that $B=e'_iAe'_i$, and let $M$ be a left $A$-module. Then the restriction map
%\begin{equation}\label{eqn dual iso'}
  %e'_i\cdot M^\circ \mapsto (e'_i\cdot M)^\circ\,, \quad \varphi\mapsto \varphi|_{e'_i\cdot M}\,,
%\end{equation}
%is a left $B$-module isomorphism.
\end{nothing}

\begin{notation}\label{nota dual labels}
As before, for each $(j,s)\in\Lambda$, we denote by $\Delta_{(j,s)}$ and $\nabla_{(j,s)}$ the standard $A$-module and the costandard $A$-module, respectively, labelled by $(j,s)$ with respect to $(\Lambda,\leq)$ and $(\Lambda,\unlhd)$, as defined in (\ref{eqn Dir}) and determined in Corollary~\ref{cor main}.
In accordance with \ref{noth duals}, for $(i,r)\in\Lambda$, we denote by $(i,r)^\circ\in\Lambda$ the label of the simple $A$-module
$D_{(i,r)}^\circ$. Analogously, let $r^\circ\in\{1,\ldots,l_i\}$ be such that $T_{(i,r^\circ)}\cong T_{(i,r)}^\circ$ as $k_\alpha\Gamma_{e_i}$-modules.
%$\Delta_{(i,r)^*}$ and $\nabla_{(i,r)^*}$ will denote the standard module and the costandard module, respectively, labelled by $(i,r)^*$ with respect to $(\Lambda,\leq)$.

With this, we now have:
\end{notation}

\begin{proposition}\label{prop costandard A-modules}
For $(i,r),(j,s)\in \Lambda$, one has

\smallskip
{\rm (a)}\, $\Delta_{(i,r)^\circ}\cong Ae_i'\otimes_{e_i'Ae_i'} \tilde{T}_{(i,r)}^\circ\cong \Delta_{(i,r^\circ)}$, thus $(i,r)^\circ=(i,r^\circ)$;

\smallskip
{\rm (b)}\, $(i,r)\leq (j,s)$ if and only if $(i,r)^\circ\leq (j,s)^\circ$;

\smallskip
{\rm (c)}\, $(i,r)\unlhd (j,s)$ if and only if $(i,r)^\circ\unlhd (j,s)^\circ$;

\smallskip
{\rm (d)}\, $\Delta_{(i,r)^\circ}\cong \nabla_{(i,r)}^\circ$ and $\nabla_{(i,r)^\circ}\cong\Delta_{(i,r)}^\circ$.
\end{proposition}

\begin{proof}
Since $T_{(i,r)}^\circ\cong T_{(i,r^\circ)}$ as $k_\alpha\Gamma_{e_i}$-modules, we have
$M:=Ae_i'\otimes_{e_i'Ae_i'} \tilde{T}_{(i,r)}^\circ\cong \Delta_{(i,r^\circ)}$ as $A$-modules. 
In order to show that $(i,r^\circ)=(i,r)^\circ$, recall that every standard module is determined by its head, and the isomorphism class of $D_{(i,r^\circ)}=\Hd(M)$ is determined by the property
$$e_i'\cdot M\cong \tilde{T}_{(i,r)}^\circ\cong e_i'\cdot \Hd(M)$$
as $e_i'Ae_i'$-modules. Since, by  (\ref{eqn id condense}), $e_i'\cdot D_{(i,r)}\cong e_i'\cdot \Delta_{(i,r)}\cong \tilde{T}_{(i,r)}$ as $e_i'Ae_i'$-modules, we also have
$$\tilde{T}_{(i,r^\circ)}\cong\tilde{T}_{(i,r)}^\circ\cong (e_i'\cdot D_{(i,r)})^\circ\cong e_i'\cdot D_{(i,r)}^\circ\cong e_i'\cdot D_{(i,r)^\circ}$$
as $e_i'Ae_i'$-modules. Note that here we
applied (\ref{eqn dual iso'}) with $f=e_i'$ to derive the third isomorphism. So, altogether, this implies
$D_{(i,r)^\circ}\cong D_{(i,r^\circ)}$ and $\Delta_{(i,r)^\circ}\cong \Delta_{(i,r^\circ)}\cong M$, proving
(a). From this, assertion~(b) follows immediately.

\smallskip
From (a),  \ref{noth duals for A}(c), and Lemma~\ref{lemma anti}(b) we now obtain
\begin{align*}
(i,r)\sqsubset (j,s)&\Leftrightarrow S_j<_{\mathscr{J}}S_i \text{ and } (f_{(i,r)}\cdot J_j\cdot f_{(j,s)}\not\subseteq J_{j-1}\text{ or }  f_{(j,s)}\cdot J_j\cdot f_{(i,r)}\not\subseteq J_{j-1})\\
&\Leftrightarrow S_j<_{\mathscr{J}}S_i \text{ and } (f_{(j,s)}^\circ \cdot J_j\cdot f_{(i,r)}^\circ\not\subseteq J_{j-1}\text{ or }  f_{(i,r)}^\circ\cdot J_j\cdot f_{(j,s)}^\circ\not\subseteq J_{j-1})\\
&\Leftrightarrow (i,r^\circ)\sqsubset (j,s^\circ)\Leftrightarrow (i,r)^\circ\sqsubset (j,s)^\circ\,,
\end{align*}
which proves (c). 

\smallskip
Assertion~(d) follows from (a), (b), and Proposition~\ref{prop dual standard modules}.
\end{proof}

\begin{corollary}\label{cor cond (ii)'}
Let $i,j\in\{1,\ldots,n\}$ be such that $S_j<_\scrJ S_i$ and suppose that $\alpha$ restricts to constant
$2$-cocycles on $\Gamma_{e_i}$ and on $\Gamma_{e_j}$ with values $a_i$ and $a_j$, respectively. Assume further that one has $\alpha(x,t)=a_i$ and $\alpha(t,y)=a_j$ for all $x\in \Gamma_{e_i}$, $t\in S_j\cap e_i\circ S\circ e_j$, and $y\in\Gamma_{e_j}$. Then, for $(i,r),(j,s)\in\Lambda$, one has $f_{(i,r)}\cdot J_j\cdot f_{(j,s)}\not\subseteq J_{j-1}$  if and only if $f_{(j,s)}\cdot J_j\cdot f_{(i,r)}\not\subseteq J_{j-1}$.
\end{corollary}

\begin{proof}
Recall that, by Lemma~\ref{lemma anti}(b), we have
\begin{equation}\label{eqn dual condition}
  f_{(j,s)}\cdot J_j\cdot f_{(i,r)}\not\subseteq J_{j-1} \Leftrightarrow 
  f_{(i,r)}^\circ\cdot J_j\cdot  f_{(j,s)}^\circ\not\subseteq J_{j-1}\,.
\end{equation}
By Lemma~\ref{lemma anti}(c), the simple left $k\Gamma_{e_i}$-module $T_{(i,r)}^\circ$ associated with $f_{(i,r)}^\circ$  is equal to $T_{(i,r)}^*$. Thus, by Proposition~\ref{prop cond (ii)} and Remark~\ref{rem trivial alpha}, we obtain
\begin{align*}
  f_{(i,r)}^\circ\cdot J_j\cdot f_{(j,s)}^\circ\not\subseteq J_{j-1} \Leftrightarrow\
  & \Hom_{k[\Gamma_{e_i}\times\Gamma_{e_j}]}(T_{(i,r)}^*\otimes T_{(j,s)},e_i' (A/J_{j-1})e_j')\neq \{0\} \\
  \Leftrightarrow\ & \Hom_{k[\Gamma_{e_i}\times\Gamma_{e_j}]}(T_{(i,r)}\otimes T_{(j,s)}^*,(e_i'(A/J_{j-1})e_j')^*)
  \neq \{0\}\,.
\end{align*}
The last equivalence holds because $k[\Gamma_{e_i}\times\Gamma_{e_j}]$ is semisimple.

But, as in the proof of Corollary~\ref{cor cond (ii)}, the $k[\Gamma_{e_i}\times\Gamma_{e_j}]$-module $e_i' (A/J_{j-1}) e_j'$ is a permutation module, and thus self-dual. Hence $e_i' (A/J_{j-1}) e_j'\cong (e_i' (A/J_{j-1}) e_j')^*$
%\cong (e_i')^\circ\cdot (A/J_{j-1})^*\cdot (e_j')^\circ=e_i'\cdot (A/J_{j-1})^*\cdot e_j'\cong e_i'\cdot (A/J_{j-1})\cdot e_j'$, 
as $k[\Gamma_{e_i}\times\Gamma_{e_j}]$-modules.
Altogether, this implies
\begin{align*}
f_{(j,s)}\cdot J_j\cdot f_{(i,r)}\not\subseteq J_{j-1}&\Leftrightarrow  \Hom_{k[\Gamma_{e_i}\times\Gamma_{e_j}]}(T_{(i,r)}\otimes T_{(j,s)}^*,e_i'(A/J_{j-1})e_j')\neq \{0\}\\
&\Leftrightarrow f_{(i,r)}\cdot J_j\cdot f_{(j,s)}\not\subseteq J_{j-1}\,,
\end{align*}
where the last equivalence again follows from Proposition~\ref{prop cond (ii)}.
\end{proof}

%\begin{remark}\label{rem field auto}
%We remark that the two conditions in Definition~\ref{defi new order}(ii) are, for instance, also equivalent whenever there is a field automorphism $\gamma:k\to k$ such that $\gamma\circ\alpha=\alpha$ and such that $\gamma(f_{(i,r)})=f_{(i,r)}^\circ$, for all
%$(i,r)\in\Lambda$. Namely, given such $\gamma$, we firstly obtain the following ring automorphism, which we denote by $\gamma$ as well:
%$$\gamma:A\to A,\; \sum_{t\in S}a_tt\mapsto \sum_{t\in S}\gamma(a_t)t\,,$$
%where $a_t\in k$, for $t\in S$. Note that $\gamma(S_i)=S_i$ and $\gamma(J_i)=J_i$ for all $i\in\{1,\ldots,n\}$.
%
%Hence, if $(i,r),(j,s)\in\Lambda$ are such that $S_j<_\scrJ S_i$ then, by Lemma~\ref{lemma  anti}(b), we in this case have
%\begin{align*}
%  f_{(i,r)}\cdot J_j\cdot f_{(j,s)}\not\subseteq J_{j-1} 
%  & \Leftrightarrow f_{(j,s)}^\circ\cdot J_j\cdot f_{(i,r)}^\circ\not\subseteq J_{j-1}
%        \Leftrightarrow \gamma(f_{(j,s)}\cdot J_j\cdot f_{(i,r)})\not\subseteq \gamma(J_{j-1})\\
%  &\Leftrightarrow f_{(j,s)}\cdot J_j\cdot f_{(i,r)}\not\subseteq J_{j-1}\,.
%\end{align*}
%\end{remark}

%%%%%%%%%%%%%%%%%%%%%%%%%%%%%%%%%%%%%%%%%%%%%%%%%%%%%%

\section{Application I: Biset Functors}\label{sec bisets}

In this section we shall apply our results from Sections~\ref{sec main}, \ref{sec sqsubseteq} and \ref{sec duality}
to the case where the twisted category algebra is the one introduced in 
\cite[Example~5.15]{BDII}. This category algebra is closely related to the category of biset functors, as observed in \cite{BDII}. The goal of this section is to reprove, via a different approach, a result due to Webb in \cite{Webb} stating that the category of biset functors over a field of characteristic zero is a highest weight category. We shall also give an improvement on the relevant partial order on the set $\Lambda$ of isomorphism classes of simple modules. Here we only deal with the case that the underlying category of finite groups has finitely many objects. This is sufficient for many purposes, as established in \cite{Webb}. 

We begin by recalling the relevant notation as well as some results from \cite{BDII} about the category $\catC$ we need to consider. The connection to biset functors will be given in more detail in Remark~\ref{rem conn with biset functors}. From now on we suppose that $k$ is a field of characteristic 0.

\begin{notation}\label{nota ghost}
(a)\, Given finite groups $G$ and $H$, we denote by $p_1$ and $p_2$ the canonical projections $G\times H\to G$ and
$G\times H\to H$, respectively. Moreover, for every $L\leq G\times H$, we set $k_1(L):=\{g\in G\mid (g,1)\in L\}$ and
$k_2(L):=\{h\in H\mid (1,h)\in L\}$, so that $k_i(L)\unlhd p_i(L)$, for $i=1,2$.

Note that, by Goursat's Lemma, we may and shall from now on identify every subgroup $L$ of $G\times H$ with the quintuple $(p_1(L),k_1(L),\eta_L,p_2(L),k_2(L))$, where $\eta_L$ is the group isomorphism given by
$$\eta_L:p_2(L)/k_2(L)\myiso p_1(L)/k_1(L)\,,\quad hk_2(L)\mapsto gk_1(L)\,,$$
whenever $(g,h)\in L$.  The common isomorphism class of $p_1(L)/k_1(L)$ and $p_2(L)/k_2(L)$ will be denoted by 
$q(L)$.

Furthermore, in the case where $p_1(L)=G=H=p_2(L)$ and $k_1(L)=1=k_2(L)$, we have $\eta_L=\alpha$ for some automorphism
$\alpha$ of $G$, and we also denote the group $L$ by $\Delta_\alpha(G)$; in particular, for $\alpha=\mathrm{id}_G$, this gives $\Delta_\alpha(G)=\Delta(G):=\{(g,g)\mid g\in G\}$.

If $g\in G$ then the corresponding inner automorphism $G\to G$, $x\mapsto gxg^{-1}$, will be denoted by $c_g$, and 
we also set $\Delta_g(G):=\Delta_{c_g}(G)$. 

By  a {\it section} of  a finite group $G$ we understand a pair $(P,K)$ such that $K\unlhd P\leq G$. 

\smallskip
(b)\, Let $\catC$ be a category with the following properties: the objects of $\catC$ form a finite set of pairwise non-isomorphic finite groups that is {\em section-closed}, that is, whenever $G\in \Ob(\catC)$ and $(P,K)$ is a section of $G$ then there is some $H\in\Ob(\catC)$ such that $P/K\cong H$. The morphism set, for $G,H\in\Ob(\catC)$, is defined by 
$$\Hom_\catC(H,G):=\catC_{G,H}:=\{L\mid L\leq G\times H\}\,,$$
and the composition of morphisms in $\catC$ is given by
$$L\circ M:=L*M:=\{(g,k)\in G\times K\mid \exists h\in H: (g,h)\in L,\, (h,k)\in M\},\,$$
for $G,H,K\in\Ob(\catC)$, $L\in \catC_{G, H}$, $M\in\catC_{H, K}$. 
For $L\in\catC_{G,H}$, let $L^\circ:=\{(h,g)\in H\times G\mid (g,h)\in L\}\in\catC_{H,G}$.
The category $\catC$ is finite by construction, and split since $L*L^\circ*L=L$, for any $G,H\in\Ob(\catC)$ and $L\in\catC_{G,H}$; see \cite[Proposition~2.7(ii)]{BDII}. Note that, by the last statement in \ref{nota twisted cat alg}(a), the assumption that the objects of $\catC$ are pairwise non-isomorphic groups is not a significant restriction.

By \cite[Proposition~3.5]{BDII}, we have a 2-cocycle $\kappa\in Z^2(\catC,k^\times)$ defined by
\begin{equation}\label{eqn kappa}
\kappa(L,M):=\frac{|k_2(L)\cap k_1(M)|}{|H|}\,,
\end{equation}
for $G,H,K\in\Ob(\catC)$, $L\in\catC_{G,H}$, $M\in\catC_{H,K}$. 
The resulting twisted category algebra $k_\kappa\catC$ will be denoted by $A$, for the remainder of this section. Moreover, we denote the objects of $\catC$ by $G_1,\ldots,G_n$ such that $|G_i|\le|G_{i+1}|$, for $i=1,\ldots,n-1$, and we set $S:=\Mor(\catC)$. Note that mapping $L\in\catC_{G,H}$ to $L^\circ\in\catC_{H,G}$ gives rise to a contravariant functor $-^\circ\colon \catC \to\catC$ satisfying the properties (i)--(v) in Hypotheses~\ref{hypo op} with respect to the 2-cocycle $\kappa$ of $\catC$. Concrete idempotents $e_1,\ldots,e_n$
will be determined in Proposition~\ref{prop J-classes of C} below.
\end{notation}

\begin{remark}\label{rem conn with biset functors}
Biset functors on $\catC$ over $k$ are related to the twisted category algebra $A=k_\kappa\catC$ as follows. For each $i=1,\ldots,n$, we set $\varepsilon_i:= \sum_{g\in G_i} \Delta_g(G_i)=|Z(G_i)|\cdot \sum_{\alpha\in\Inn(G_i)}\Delta_\alpha(G_i)\in e'_i A e'_i$, and we set $\varepsilon:=\varepsilon_{\catC}:=\sum_{i=1}^n \varepsilon_i$. Then $\varepsilon_1,\ldots,\varepsilon_n$ are pairwise orthogonal idempotents of $A$, $\varepsilon$ is an idempotent of $A$, and the left module category of the $k$-algebra $\varepsilon A\varepsilon$ is equivalent to the category of biset functors on $\catC$ over $k$; see \cite[Example~5.15(c)]{BDII} for more detailed explanations. By Theorem~\ref{thm main}, we know that $A$ is quasi-hereditary with respect to $(\Lambda,\unlhd)$, using the notation from Sections~\ref{sec twisted cat} and \ref{sec main}. 

Our goal is to show that also the condensed $k$-algebra $\varepsilon A\varepsilon$ is quasi-hereditary. Recall from Green's idempotent condensation theory (see \cite[Section~6.2]{Gr}) that the simple modules of $\varepsilon A\varepsilon$ are of the form $\varepsilon \cdot D_{(i,r)}$, with $(i,r)\in\Lambda$ such that $\varepsilon\cdot D_{(i,r)}\neq\{0\}$ and that any two distinct such indices $(i,r)$ result in non-isomorphic simple $\varepsilon A\varepsilon$-modules. Thus, the labelling set $\Lambda'$ of the isomorphism classes of simple $\varepsilon A\varepsilon$-modules can be considered as a subset of $\Lambda$ in a natural way. Moreover, by Proposition~\ref{prop qh condensed}, it suffices to show that the idempotent $\varepsilon$ satisfies the following property: If $(i,r)\unlhd(j,s)$ are elements in $\Lambda$ and if $\varepsilon\cdot D_{(i,r)}\neq\{0\}$ then also $\varepsilon\cdot D_{(j,s)}\neq\{0\}$. This will be done in Theorem~\ref{thm condense}. The main reason for introducing the partial order $\unlhd$ in Section~\ref{sec main} is that
this property is not satisfied for the partial order $\leq$, as we shall see in Example~\ref{expl S_4} below.
% we could not prove this property for the partial order $\leq$ in this particular example.
\end{remark}

The following proposition establishes quickly the set $\Lambda$ for the finite split category algebra $A=k_\kappa\catC$ and the subset $\Lambda'\subseteq\Lambda$.

\begin{proposition}\label{prop J-classes of C}
{\rm (a)}\, For $L,M\in S$, one has $\scrJ(L)=\scrJ(M)$ if and only if $q(L)=q(M)$. In particular, the elements $e_i:=\Delta(G_i)\in \catC_{G_i,G_i}\subseteq S$, with $i=1,\ldots,n$, form a set of representatives of the $\scrJ$-classes of $\catC$. 
For $i=1,\ldots,n$, we have $e_i^\circ =e_i$, where $-^\circ:\catC\to \catC$ is the functor in \ref{nota ghost}(b).

\smallskip
{\rm (b)}\, For $i\in\{1,\ldots,n\}$, the element $e_i$ is an idempotent and
$\Gamma_{e_i}=\{\Delta_\alpha(G_i)\mid \alpha\in\Aut(G_i)\}$; in particular,
$\Aut(G_i)\cong\Gamma_{e_i}$ via the map $\alpha\mapsto\Delta_\alpha(G_i)$. Moreover, the $2$-cocycle $\kappa\in Z^2(\catC,k^\times)$ restricts to a constant $2$-cocycle
on $\Gamma_{e_i}$ with value $|G_i|^{-1}$, and the $k$-linear map
\begin{equation}\label{eqn untwisted iso}
k_\kappa\Gamma_{e_i}\to k\Aut(G_i),\; \Delta_\alpha(G_i)\mapsto |G_i|^{-1}\cdot \alpha\,,\quad (\alpha\in \Aut(G_i))
\end{equation}
defines a $k$-algebra isomorphism. If also $j\in \{1,\ldots,n\}$ and $L\in \catC_{G_i,G_j}$ then $\kappa(e_i,L)=|G_i|^{-1}$.

\smallskip
{\rm (c)}\, For $(i,r)\in\Lambda$, one has $\varepsilon\cdot D_{(i,r)}\neq\{0\}$ if and only if $\Inn(G_i)$ acts trivially on $T_{(i,r)}$, when viewed as $k\Aut(G_i)$-module via the isomorphism in (b).

\smallskip
{\rm (d)}\, For $i=1,\ldots,n$, we set $S_i:=\scrJ(e_i)$. Then, for $i,j\in\{1,\ldots,n\}$, one has $S_i\leq_\scrJ S_j$ if and only if there is a section $(P,K)$ of $G_j$ with $G_i\cong P/K$. In particular, the ordering $G_1,\ldots,G_n$ has the property that if $\scrJ(e_i)\leq_\scrJ \scrJ(e_j)$ then $i\le j$, as required in \ref{nota twisted cat alg}(b).
\end{proposition}

\begin{proof}
Assertions (a) and (b) follow immediately from \cite[Proposition~6.3, Proposition~6.4]{BDII}, \cite[Lemma~2.1]{LS}, and the definition of $\kappa$.

\smallskip
Part (c) follows immediately from \cite[Corollary~7.4]{BDII} and its proof.

\smallskip
To prove part (d), note that, in the notation of \ref{nota ghost}, we have 
$e_j=(G_j,1,\id,G_j,1)$.
Suppose first that $S_i\leq_\scrJ S_j$, that is, $S*e_i*S_i\subseteq S*e_j*S$, by \ref{nota twisted cat alg}(b).
Thus $e_i=L*e_j*M$, for some $L,M\in S$. But this implies that $G_i$ is isomorphic to a subquotient of $G_j$, see \cite[Lemma~2.7]{BDII}.

Conversely, suppose that there is a section $(P,K)$ of $G_j$ such that $P/K\cong G_i$. Then $e:=(P,K,\id,P,K)$ is an idempotent in $\catC$ with $\scrJ(e)=\scrJ(e_i)=S_i$, by part~(a). Moreover, $e=e*e_j*e$, thus $S*e*S\subseteq S*e_j*S$, implying
$S_i=\scrJ(e_i)=\scrJ(e)\leq_\scrJ\scrJ(e_j)=S_j$.
\end{proof}

\begin{notation}\label{nota J-classes of C}
(a)\, For $i\in\{1,\ldots,n\}$, let $e_i:=\Delta(G_i)$ and set $S_i:=\scrJ(e_i)$, as in Proposition~\ref{prop J-classes of C}.
In consequence of Proposition~\ref{prop J-classes of C}, $S_1,\ldots,S_n$ are then precisely the distinct $\scrJ$-classes of $\catC$, and $e_i$ is an idempotent endomorphism in $S_i$, for $i\in\{1,\ldots,n\}$.

Also, for $i\in\{1,\ldots,n\}$, we shall from now on identify the group $\Gamma_{e_i}$ with the automorphism group
$\Aut(G_i)$, and the twisted group algebra $k_\kappa\Gamma_{e_i}$ with the untwisted group algebra $k\Aut(G_i)$, via the isomorphisms in Proposition~\ref{prop J-classes of C}(b). In particular, every $k_\kappa\Gamma_{e_i}$-module can and will from now on
be viewed as a $k\Aut(G_i)$-module.

\smallskip
(b)\, Suppose that $i,j\in\{1,\ldots,n\}$ are such that $S_j<_\scrJ S_i$, so that we have the $\Gamma_{e_i}\times\Gamma_{e_j}$-action
on $S_j\cap e_i*S*e_j$ introduced in \ref{nota perm action}. Thus, via the isomorphisms $\Gamma_{e_i}\cong \Aut(G_i)$ and
$\Gamma_{e_j}\cong \Aut(G_j)$, we also have a left $\Aut(G_i)\times \Aut(G_j)$-action on
$S_j\cap e_i*S*e_j$ via
\begin{equation}\label{eqn conj L}
{}^{(\alpha,\beta)}L:=\Delta_\alpha(G_i)*L*\Delta_{\beta^{-1}}(G_j)\quad (\alpha\in\Aut(G_i),\beta\in\Aut(G_j),L\in S_j\cap e_i*S*e_j)\,.
\end{equation}
As before, we shall denote the stabilizer of $L\in S$ in $\Aut(G_i)\times\Aut(G_j)$ simply by $\stab(L)$, whenever $i$ and $j$ are apparent from the context.
%Furthermore, we retain our general notation introduced in the previous sections.
\end{notation}

\begin{remark}\label{rem untwisted}
(a)\, 
%In consequence of Proposition~\ref{prop J-classes of C} and Lemma~\ref{lemma op on C}, $\catC$ is a finite split category, and the twisted category algebra $A=k_\kappa\catC$ satisfies the hypotheses in \ref{hypo op}. In particular, Theorem~\ref{thm main} applies. 
By Proposition~\ref{prop J-classes of C}(b), we are able to apply Corollary~\ref{cor cond (ii)'}.
So, suppose that $(i,r),(j,s)\in\Lambda$ are such that $S_j<_\scrJ S_i$. Then, by Corollary~\ref{cor cond (ii)'},
Proposition~\ref{prop cond (ii)} and Proposition~\ref{prop cond (ii)'}, the following are equivalent:

\smallskip
\quad (i)\, $f_{(i,r)}\cdot J_j\cdot f_{(j,s)}\not\subseteq J_{j-1}$;

\smallskip
\quad (ii)\, $f_{(j,s)}\cdot J_j\cdot f_{(i,r)}\not\subseteq J_{j-1}$;

\smallskip
\quad (iii)\, there is some $L\in S_j\cap e_i*S*e_j$ with $f_{(i,r)}\cdot L\cdot f_{(j,s)}\neq 0$;

\smallskip
\quad (iv)\, there is some $M\in S_j\cap e_j*S*e_i$ with $f_{(j,s)}\cdot M\cdot f_{(i,r)}\neq 0$.

\smallskip
Note also that the set $S_j\cap e_i*S*e_j$ consists precisely of those subgroups $(P,K,\eta,G_j,1)$ of $G_i\times G_j$, where $(P,K)$ is a section of $G_i$, and $\eta:G_j\to P/K$ is an isomorphism.

\smallskip
(b)\, In the proof of Lemma~\ref{lemma inner} below we shall often only be interested to see whether certain products of elements in $A=k_\kappa\catC$ are non-zero, without determining the coefficients  at the standard basis elements explicitly. Therefore, given
$a,b\in A$,  we shall write $a\sim b$ if there is some $\lambda\in k^\times$ such that $a=\lambda b$.

With this convention we, in particular, deduce the following description of the block idempotent $f_{(i,r)}$ of $k_\kappa\Gamma_{e_i}$:
view $T_{(i,r)}$ as a simple $k\Aut(G_i)$-module via the isomorphism (\ref{eqn untwisted iso}), and let $\chi_{(i,r)}$ be the
character of $\Aut(G_i)$ afforded by $T_{(i,r)}$. Then the corresponding
block idempotent of $k\Aut(G_i)$ is
$$f'_{(i,r)}:=\frac{\chi_{(i,r)}(1)}{u^2v|\Aut(G_i)|}\sum_{\alpha\in \Aut(G_i)}\chi_{(i,r)}(\alpha^{-1})\alpha\,,$$
where $\chi_{(i,r)}=u(\psi_1+\cdots +\psi_v)$ is a decomposition into absolutely irreducible characters over a suitable extension field of $k$. Thus, applying (\ref{eqn untwisted iso}) again, we get 
\begin{equation}\label{eqn f_ir}
f_{(i,r)}= \frac{|G_i|\cdot \chi_{(i,r)}(1)}{u^2v|\Aut(G_i)|}\sum_{\alpha\in \Aut(G_i)}\chi_{(i,r)}(\alpha^{-1})\Delta_\alpha(G_i)    \sim \sum_{\alpha\in\Aut(G_i)}\chi_{(i,r)}(\alpha^{-1})\Delta_\alpha(G_i)\,.
\end{equation}
\end{remark}

The next lemma will be the key step towards establishing Theorem~\ref{thm condense}, our main result of this section.

\begin{lemma}\label{lemma inner}
Let $(i,r),(j,s)\in \Lambda$ be such that $S_j<_\scrJ S_i$, and suppose that $L\in S_j\cap e_i* S* e_j$ is such that $f_{(i,r)}\cdot L\cdot f_{(j,s)}\neq 0$. Then one has

\smallskip
{\rm (a)}\, $\stab_{\Aut(G_i)\times\Aut(G_j)}(L)\cdot(1\times \Inn(G_j))\leq (\Inn(G_i)\times 1)\cdot\stab_{\Aut(G_i)\times\Aut(G_j)}(L)$;

\smallskip
{\rm (b)}\, if $\Inn(G_i)$ acts trivially on the simple $k\Aut(G_i)$-module $T_{(i,r)}$ then $\Inn(G_j)$ acts trivially on the simple $k\Aut(G_j)$-module $T_{(j,s)}$.
\end{lemma}

\begin{proof}
For ease of notation, set $A_i:=\Aut(G_i)$, $A_j:=\Aut(G_j)$, $I_i:=\Inn(G_i)$, and $I_j:=\Inn(G_j)$.

Since $L\in S_j\cap e_i*S*e_j$, we deduce from Proposition~\ref{prop J-classes of C} that $L=(P,K,\eta,G_j,1)$, for some $1\leq K\unlhd P\leq G_i$.

\smallskip
To prove (a), note first that $\stab(L)(1\times I_j)$ and $(I_i\times 1)\stab(L)$ are indeed subgroups of $A_i\times A_j$, since $I_i\unlhd A_i$ and $I_j\unlhd A_j$. Note further that it suffices to show that $1\times I_j\leq (I_i\times 1)\stab(L)$.

For $(\alpha,\beta)\in A_i\times A_j$, we have
\begin{align*}
(\alpha,\beta)\in \stab(L)&\Leftrightarrow \Delta_\alpha(G_i)*L*\Delta_{\beta^{-1}}(G_j)=L\\
&\Leftrightarrow (\alpha(P),\alpha(K),\bar{\alpha}\circ \eta\circ \beta^{-1},G_j,1)=(P,K,\eta,G_j,1)\,,
\end{align*}
where $\bar{\alpha}$ is the isomorphism $P/K\to \alpha(P)/\alpha(K)$ induced by $\alpha$.

Now, given $\beta\in\Inn(G_j)$, there is some $g\in G_j$ with $\beta=c_g$. Let 
$h\in P\leq G_i$ be such that $\eta(g)=hK$, and set $\alpha:=c_h\in\Inn(G_i)$. Since $h\in P$ and $K\unlhd P$, we get $\alpha(P)=P$,
$\alpha(K)=K$ as well as 
$$(\bar{\alpha}\circ\eta\circ\beta^{-1})(x)=\bar{\alpha}(\eta(g^{-1}xg))=\bar{\alpha}(h^{-1}K\cdot \eta(x)\cdot hK)=hK\cdot h^{-1}K\cdot \eta(x)\cdot hK\cdot h^{-1}K =\eta(x)\,,$$
for all $x\in G_j$. Thus $(\alpha,\beta)\in\stab(L)$, and 
$$(1,\beta)=(\alpha,\beta)\cdot (\alpha^{-1},1)\in \stab(L)(I_i\times 1)\,,$$
implying $1\times I_j\leq \stab(L)(I_i\times 1)$. This proves assertion (a).

\smallskip
To prove assertion~(b), recall from (\ref{eqn f_ir}) that
$$f_{(i,r)}\sim\sum_{\alpha\in A_i}\chi_{(i,r)}(\alpha^{-1})\Delta_\alpha(G_i)\quad\text{ and }\quad f_{(j,s)}\sim\sum_{\beta\in A_j}\chi_{(j,s)}(\beta^{-1})\Delta_\beta(G_j)\,,$$
and recall from Remark~\ref{rem conn with biset functors} that
\begin{equation}\label{eqn epsilon}
\varepsilon_i=|Z(G_i)|\cdot\sum_{\alpha\in I_i}\Delta_\alpha(G_i)
\end{equation}
is an idempotent in $k_\kappa\Gamma_{e_i}$ that, up to a  non-zero scalar, corresponds under the isomorphism in (\ref{eqn untwisted iso}) to the principal block idempotent of $kI_i$. Since $I_i\unlhd A_i$, the element $\varepsilon_i$, viewed in $kA_i$, is stable under $A_i$-conjugation. Thus, $\varepsilon_i$ is a central idempotent of $k_\kappa \Gamma_{e_i}$. Similarly, $\varepsilon_j$ is an idempotent in $Z(k_\kappa\Gamma_{e_j})$.

\smallskip
Now, assume that $I_i$ acts trivially on $T_{(i,r)}$, but $I_j$ does not act trivially on $T_{(j,s)}$. Then we get 
$0\neq \varepsilon_i\cdot T_{(i,r)}=\varepsilon_i f_{(i,r)}\cdot T_{(i,r)}$, thus $\varepsilon_if_{(i,r)}\neq 0$ and
$$\varepsilon_if_{(i,r)}\sim f_{(i,r)}\sim f_{(i,r)}\varepsilon_i\,.$$
On the other hand, we have 
$$\varepsilon_jf_{(j,s)}=0=f_{(j,s)}\varepsilon_j\,;$$ 
for otherwise we would have $ \varepsilon_jf_{(j,s)}\sim f_{(j,s)}\sim f_{(j,s)}\varepsilon_j$,  and so $I_j$ would act trivially on $T_{(j,s)}=f_{(j,s)}\cdot T_{(j,s)}=\varepsilon_jf_{(j,s)}\cdot T_{(j,s)}$, contradicting our assumption.

Therefore, we also have $0=f_{(i,r)}\cdot L\cdot f_{(j,s)}\varepsilon_j$ and $0\neq f_{(i,r)}\cdot L\cdot f_{(j,s)}\sim \varepsilon_if_{(i,r)}\cdot L\cdot f_{(j,s)}$. 

\smallskip
Our final step will be to show that
\begin{equation}\label{eqn char prod}
(\chi_{(i,r)}^*\times \chi_{(j,s)},1)_{\stab(L)(1\times I_j)}=0\neq (\chi_{(i,r)}^*\times \chi_{(j,s)},1)_{(I_i\times 1)\stab(L)}\,,
\end{equation}
which will then, by (a), lead to a contradiction completing the proof of (b).
Here $1$ simply denotes the trivial character of $\stab(L)(1\times I_j)$ and $(I_i\times 1)\stab(L)$, respectively.
%Since, by part~(a), we already know that $\stab(L)(1\times I_j)\leq (I_i\times 1)\stab(L)$, this will then lead to a contradiction, and complete the proof of (b).

\smallskip
By (\ref{eqn f_ir}) and (\ref{eqn epsilon}), we have
\begin{align}\nonumber
0&=f_{(i,r)}\cdot L\cdot f_{(j,s)}\varepsilon_j\sim\sum_{\alpha\in A_i}\sum_{\beta\in A_j}\sum_{g\in G_j}\chi_{(i,r)}(\alpha^{-1})\chi_{(j,s)}(\beta^{-1}) \cdot {}^{(\alpha,\beta^{-1})}L*\Delta_{g}(G_j)\\\nonumber
&=\sum_{\alpha\in A_i}\sum_{\beta\in A_j}\sum_{g\in G_j}\chi_{(i,r)}(\alpha^{-1})\chi_{(j,s)}(\beta^{-1}) \cdot  {}^{(\alpha,c_{g^{-1}}\circ\beta^{-1})}L\\\label{eqn fLf}
&=\sum_{\alpha\in A_i}\sum_{\beta\in A_j}\sum_{g\in G_j}\chi_{(i,r)}(\alpha^{-1})\chi_{(j,s)}(\beta) \cdot  {}^{(\alpha,c_{g^{-1}}\circ\beta)}L\,.
\end{align}
Fixing $(\alpha_0,\beta_0)\in A_i\times A_j$, the coefficient at ${}^{(\alpha_0,\beta_0)}L$ in (\ref{eqn fLf}) equals
\begin{align*}
\sum_{\substack{\alpha\in A_i\\ \beta\in A_j\\ g\in G_j\\ {}^{(\alpha,c_g^{-1}\circ\beta)}L={}^{(\alpha_0,\beta_0)}L}}\chi_{(i,r)}(\alpha^{-1})\chi_{(j,s)}(\beta)&=\sum_{g\in G_j}\sum_{(\sigma,\tau)\in\stab(L)}\chi_{(i,r)}(\sigma^{-1}\circ\alpha_0^{-1})\chi_{(j,s)}(c_g\circ\beta_0\circ\tau)\\
&=\sum_{g\in G_j}\sum_{(\sigma,\tau)\in\stab(L)}\chi_{(i,r)}(\sigma^{-1}\circ\alpha_0^{-1})\chi_{(j,s)}(\beta_0\circ\tau\circ c_g)\\
&\sim\sum_{(\sigma,\tau)\in\stab(L)(1\times I_j)}\chi_{(i,r)}(\sigma^{-1}\circ\alpha_0^{-1})\chi_{(j,s)}(\beta_0\circ\tau)\\
&=(\chi_{(i,r)}^*\times\chi_{(j,s)})\left((\alpha_0,\beta_0)\cdot (\stab(L)(1\times I_j))^+\right)\,.
\end{align*}
Here $(\stab(L)(1\times I_j))^+:=\sum_{(\sigma,\tau)\in\stab(L)(1\times I_j)}(\sigma,\tau)\in k[A_i\times A_j]$.

Hence, altogether this yields
$$0=f_{(i,r)}\cdot L\cdot f_{(j,s)}\varepsilon_j\sim\sum_{\substack{(\alpha,\beta)\in\\ [A_i\times A_j/\stab(L)]}}(\chi_{(i,r)}^*\times\chi_{(j,s)})\left((\alpha,\beta)\cdot (\stab(L)(1\times I_j))^+\right)\cdot {}^{(\alpha,\beta)}L\,,$$
where $[A_i\times A_j/\stab(L)]$ denotes a set of representatives of the left cosets $A_i\times A_j/\stab(L)$.
Thus,  we have $0=f_{(i,r)}\cdot L\cdot f_{(j,s)}\varepsilon_j$ if and only if $(\chi_{(i,r)}^*\times\chi_{(j,s)})\left((\alpha,\beta)\cdot (\stab(L)(1\times I_j))^+\right)=0$, for all $(\alpha,\beta)\in A_i\times A_j$. By \cite[Lemma~7.3]{BDII}, the latter condition is in turn satisfied if and only if 
$(\chi_{(i,r)}^*\times\chi_{(j,s)})\left((\stab(L)(1\times I_j))^+\right)=0$, that is, if and only if 
$(\chi_{(i,r)}^*\times \chi_{(j,s)},1)_{\stab(L)(1\times I_j)}=0$.

\smallskip
A completely analogous calculation gives
$$0\neq \varepsilon_i f_{(i,r)}\cdot L\cdot f_{(j,s)}\sim\sum_{\substack{(\alpha,\beta)\in\\ [A_i\times A_j/\stab(L)]}}(\chi_{(i,r)}^*\times\chi_{(j,s)})\left((\alpha,\beta)((I_i\times 1)\stab(L))^+\right)\cdot {}^{(\alpha,\beta)}L\,,$$
so that $0\neq \varepsilon_i f_{(i,r)}\cdot L\cdot f_{(j,s)}$ holds if and only if there is some $(\alpha,\beta)\in A_i\times A_j$ with
$(\chi_{(i,r)}^*\times\chi_{(j,s)})\left((\alpha,\beta)((I_i\times 1)\stab(L))^+\right)\neq 0$. By \cite[Lemma~7,3]{BDII} again, this is equivalent
to $(\chi_{(i,r)}^*\times\chi_{(j,s)})\left(((I_i\times 1)\stab(L))^+\right)\neq 0$, which is equivalent to
$(\chi_{(i,r)}^*\times \chi_{(j,s)},1)_{(I_i\times 1)\stab(L)}\neq 0$.

\smallskip
To summarize, we have now established (\ref{eqn char prod}), which completes the proof of assertion~(b).
\end{proof}

\begin{theorem}\label{thm condense}
For $i\in\{1,\ldots,n\}$, let $\varepsilon_i:=|Z(G_i)|\cdot\sum_{\alpha\in\Inn(G_i)}\Delta_\alpha(G_i)$, and let further $\varepsilon:=\varepsilon_\catC:=\sum_{i=1}^n\varepsilon_i$ (see Remark~\ref{rem conn with biset functors}). Then the following hold:

\smallskip
{\rm (a)}\, The twisted category algebra $A=k_\kappa\catC$ is quasi-hereditary, both with respect to $(\Lambda,\leq)$ and to $(\Lambda,\unlhd)$.
For $(i,r)\in\Lambda$, the corresponding standard and costandard $A$-modules $\Delta_{(i,r)}$ and $\nabla_{(i,r)}$ with respect to both $\leq$ and $\unlhd$ satisfy
\begin{equation*}
  \Delta_{(i,r)}\cong A e_i'\otimes_{e_i'Ae_i'}\Ttilde_{(i,r)}\quad \text{and}\quad
  \nabla_{(i,r)}\cong \Hom_{e_i'Ae_i'}(e_i'A,\Ttilde_{(i,r)}) \cong (Ae_i'\otimes_{e_i'Ae_i'}\Ttilde_{(i,r)}^\circ)^\circ
\end{equation*}
with respect to the anti-involution $-^\circ\colon A\to A$ from \ref{nota ghost}(b).

\smallskip
{\rm (b)}\, Suppose that $(i,r),(j,s)\in\Lambda$ are such that $(i,r)\lhd (j,s)$. If $\varepsilon\cdot D_{(i,r)}\neq \{0\}$ then also $\varepsilon\cdot D_{(j,s)}\neq \{0\}$.

\smallskip
{\rm (c)}\, The element $\varepsilon$ is an idempotent in $A$. Moreover, the condensed algebra $\varepsilon A \varepsilon$ is quasi-hereditary with respect to the partial order induced by $\unlhd$ on $\Lambda':=\{(i,r)\in\Lambda\mid \varepsilon\cdot D_{(i,r)}\neq \{0\}\}$. The corresponding standard and costandard modules are precisely the modules  $\Delta_{(i,r)}':=\varepsilon\cdot \Delta_{(i,r)}$ and $\nabla_{(i,r)}':=\varepsilon\cdot \nabla_{(i,r)}$, respectively, for $(i,r)\in\Lambda'$. For every $(i,r)\in\Lambda'$, one also has an isomorphism $(\nabla_{(i,r)}')^\circ\cong \Delta_{(i,r)^\circ}'=\Delta_{(i,r^\circ)}'$ of $\varepsilon A\varepsilon$-modules.
\end{theorem}

\begin{proof}
Assertion~(a) is immediate from Theorem~\ref{thm main}, Corollary~\ref{cor main}, Proposition~\ref{prop costandard A-modules}, and the properties of the duality functor $-^\circ:\catC\to\catC$, see the last paragraph of \ref{nota ghost}(b). \smallskip
To prove (b), let $(i,r),(j,s)\in\Lambda$ be such that $(i,r)\lhd (j,s)$, that is, there exist $m\in\mathbb{N}$ and 
suitable $(i_0,r_0)=(i,r),(i_1,r_1),\ldots,(i_{m-1},r_{m-1}),(i_m,r_m)=(j,s)\in\Lambda$ such that
$$S_{i_{q+1}}<_\scrJ S_{i_q}\quad\text{ and }\quad f_{(i_q,r_q)}\cdot J_{i_{q+1}}\cdot f_{(i_{q+1},r_{q+1})}\not\subseteq J_{i_{q+1}-1}\,,$$
for all $q=0,\ldots,m-1$. 
Since $\varepsilon\cdot D_{(i,r)}\neq 0$, we deduce from \cite[Corollary~7.6]{BDII} that $\Inn(G_i)$ acts trivially on $T_{(i,r)}$.
By Proposition~\ref{prop cond (ii)}(b) and Lemma~\ref{lemma inner}(b), $\Inn(G_{i_1})$ acts trivially on $T_{(i_1,r_1)}$. Thus
$\varepsilon\cdot D_{(i_1,r_1)}\neq 0$, by \cite[Corollary~7.6]{BDII} again. Iteration of this argument implies $\varepsilon\cdot D_{(j,s)}\neq 0$, as claimed.

\smallskip
As already mentioned in Remark~\ref{rem conn with biset functors} and shown in \cite{BDII}, $\varepsilon$ is an idempotent in $A$.
Moreover, for every $i=1,\ldots,n$, we have $\varepsilon_i^\circ=\sum_{g\in G_i}\Delta_g(G_i)^\circ=\sum_{g\in G_i}\Delta_{g^{-1}}(G_i)=\varepsilon_i$, and thus also 
$\varepsilon^\circ=\varepsilon$.
Assertion~(c) now follows immediately from (a), (b), Proposition~\ref{prop qh condensed}, (\ref{eqn dual iso}), and Proposition~\ref{prop costandard A-modules}(a).
\end{proof}

\begin{remark}\label{rem not section closed}
(a) Theorem~\ref{thm condense}(a) remains true if one only requires that the morphisms of $\catC$ satisfy the slightly technical condition (10) in \cite{BDII}. 
But the assumption of $\catC$ being section-closed was needed in \cite[Corollary~7.6]{BDII}, and thus in the proofs of Lemma~\ref{lemma inner} and Theorem~\ref{thm condense}(b)--(c). 
So if $\catC$  satisfies condition (10) in \cite{BDII} and is section-closed then
parts (b) and (c) of Theorem~\ref{thm condense} still remain true.

\smallskip
(b) Part~(c) of Theorem~\ref{thm condense} gives a different proof of one of  the main results of Webb in \cite[Section~7]{Webb} in the case that $\catC$ is finite. Lifting the finite case to the infinite case is possible using standard techniques. In fact, if $\catC'\subseteq\catC$ is a full subcategory whose objects are again closed under taking subsections then all the constructions for $\catC'$ arise from those of $\catC$ by multiplying with the idempotent $\varepsilon_{\catC'}$. 

In order to compare our approach with the one in \cite{Webb},
we first claim that if $\varepsilon D_{(i,r)}\neq \{0\}$ then this module corresponds to the simple biset functor $S_{G_i,T_{(i,r)}}$ (in the notation in \cite{Webb}). In fact, by \cite[Theorem~4.3.10 and Lemma~4.3.9]{Bouc}, $S_{G_i, T_{(i,r)}}$ is characterized by the
following two properties:

\smallskip

(i) $G_i$ has minimal order among all group $G_j\in\Ob(\catC)$ with the property that $S_{G_i,T_{(i,r)}}(G_j)\neq \{0\}$, and
%$G_i$ is minimal with respect to the subquotient relation among all group $G_j\in\catC$ with the property that $S_{G_i,T_{(i,r)}}(G_j)\neq \{0\}$, and

(ii) $S_{G_i,T_{(i,r)}}(G_i)\cong T_{(i,r)}$ as $k\Out(G_i)$-modules.\\
To prove the claim, recall from \cite{Webb} that evaluation of a biset functor at a group $G_j$ translates into multiplying the corresponding $\varepsilon A\varepsilon$-module with $\varepsilon_j$. 
Suppose that $\varepsilon\cdot D_{(i,r)}\neq \{0\}$.
Then, by \cite[Corollary~7.6]{BDII}, $\Inn(G_i)$ acts trivially on $T_{(i,r)}$, so that $T_{(i,r)}$ can be viewed as a simple
$k\Out(G_i)$-module.
Now, $D_{(i,r)}$ satisfies property (ii), since 
\begin{equation*}
  \varepsilon_i\varepsilon D_{(i,r)} =\varepsilon_i  D_{(i,r)}= \varepsilon_i e_i D_{(i,r)} \cong \varepsilon_i \Ttilde_{(i,r)} \cong T_{(i,r)}
\end{equation*}
as $k\Out(G_i)$-modules, by (\ref{eqn id condense}). 
%In order to show Property~(i), it suffices to show that $\varepsilon_j\varepsilon\Delta_{(i,r)}=\{0\}$ for all $j\in\{1,\ldots,n\}$ with $S_j\le_{\scrJ} S_i$. But in this case, we have $\varepsilon_j\varepsilon\Delta_{(i,r)}=\varepsilon_j Ae'_i\otimes_{e'_iAe'_i}\Ttilde_{(i,r)} =\varepsilon_j e'_jAe_i'\otimes_{e'_iAe'_i}\Ttilde_{(i,r)}$. Since $S_j\le_{\scrJ} S_i$, there exist $a,b\in S$ such that $e_j=a\circ e_i\circ b$. Thus, $e_j=e_j\circ e_j = a\circ e_i\circ e_i\circ b\circ e_j$ and each element of $e_j'Ae_i'$ is a $k$-linear combination of elements of the form $(a\circ e_i)\circ (e_i\circ b\circ e_j\circ c\circ e_i)$ with $c\in S$. But $e_i\circ b\circ e_j\circ c\circ e_i\in e_i\circ S\circ e_i\cap S_{\le i-1}=J_{e_i}$, by \cite[Equation~(6)]{BDIII}. Since $kJ_{e_i}$ annihilates $\Ttilde_{(i,r)}$, the claim is proven.
In order to prove (i), suppose that $|G_j|<|G_i|$, so that $S_i\not\leq_{\scrJ}S_j$. Thus $S_j\cdot D_{(i,r)}=\{0\}$, by \cite[Theorem~1.2, Proposition~5.1]{LS}. From this we get 
$$\varepsilon_j\varepsilon D_{(i,r)}= \varepsilon_j D_{(i,r)}=\varepsilon_je_j D_{(i,r)}=\{0\}\,.$$

Therefore, also the standard modules and costandard modules constructed in \cite{Webb} must coincide with ours, since the ones constructed in \cite{Webb} are the standard modules with respect to the partial order we call $\leq$, and since $\leq$ is a refinement of $\unlhd$. However, knowing that they are also the standard modules with respect to $\unlhd$ is an improvement, since it imposes restrictions on possible composition factors occurring in the standard modules (see Example~\ref{expl S_4}). 

\smallskip
(c) The standard modules $\varepsilon\cdot\Delta_{(i,r)}$ appear also in \cite{BST} as the functors $\Lbar_{H,V}$, see the paragraph preceding \cite[Lemma~4.3]{BST} for a definition. There they play an important role in the determination of simple biset functors, but without any investigation of quasi-hereditary structures. That our standard modules coincide with these functors can also be seen from their definition using (\ref{eqn Dir}).
\end{remark}

\begin{example}\label{expl S_4}
To conclude this section, we shall illustrate some of our previous results by an explicit example. We shall, in particular, see that the relation $\sqsubseteq$ in Definition~\ref{defi new order} is in general not transitive. Furthermore, we shall show that the partial order $\leq$ on $\Lambda$ is in general a proper refinement of $\unlhd$. Throughout this example, let $k:=\mathbb{C}$.

\smallskip
(a)\, With the notation as in \ref{nota ghost}, we consider the category $\catC$ whose objects are the following finite groups:
$$G_1:=\{1\},\; G_2:=C_2,\; G_3:=C_3,\, G_4:=C_4,\; G_5:=C_2\times C_2\,,$$
$$G_6:=\mathfrak{S}_3,\; G_7:=D_8,\; G_8:=\mathfrak{A}_4,\; G_9:=\mathfrak{S}_4\,.$$
Here, $D_8$ denotes the dihedral group of order 8. In particular, $\Ob(\catC)$ is a section-closed set. 
We shall determine the relation $\sqsubseteq$ via Corollary~\ref{cor cond (ii)}(a), and encode it in Table~\ref{tab S_4}.
To this end, we first list, for each $G\in\Ob(\catC)$,
the isomorphism type of $\Aut(G)$ as well as the ordinary  irreducible $\Aut(G)$-characters. 
The $\Aut(G)$-characters that restrict trivially to $\Inn(G)$ 
are set in boldface, as these are precisely the characters leading to simple $A$-modules not annihilated by $\varepsilon$.

\smallskip
\begin{center}
\begin{tabular}{|c|c|c|l|} \hline
$i$ & $G_i$ & $\Aut(G_i)$ & $\mathrm{Irr}(\Aut(G_i))$\\\hline\hline
$1$ & $\{1\}$& $\{1\}$ & $\boldsymbol{\chi_{(1,1)}}:=1$\\\hline
$2$ & $C_2$&$\{1\}$ & $\boldsymbol{\chi_{(2,1)}}:=1$\\\hline
$3$ & $C_3$&$C_2$&$\boldsymbol{\chi_{(3,1)}}:=1, \boldsymbol{\chi_{(3,2)}}:=\sgn$\\\hline
$4$ & $C_4$& $C_2$&$\boldsymbol{\chi_{(4,1)}}:=1,\boldsymbol{\chi_{(4,2)}}:=\sgn$\\\hline
$5$ & $C_2\times C_2$&$\mathfrak{S}_3$&$\boldsymbol{\chi_{(5,1)}}:=1,\boldsymbol{\chi_{(5,2)}}:=\sgn,\boldsymbol{\chi_{(5,3)}}:=\nu_2$\\\hline
$6$ & $\mathfrak{S}_3$&$\mathfrak{S}_3$&$\boldsymbol{\chi_{(6,1)}}:=1,\chi_{(6,2)}:=\sgn,\chi_{(6,3)}:=\nu_2$\\\hline
$7$ & $D_8$&$D_8$&$\boldsymbol{\chi_{(7,1)}}:=1,\chi_{(7,2)}:=\tau,\boldsymbol{\chi_{(7,3)}}:=\mu,\chi_{(7,4)}:=\mu',\chi_{(7,5)}:=\chi$\\\hline
$8$ & $\mathfrak{A}_4$&$\mathfrak{S}_4$&$\boldsymbol{\chi_{(8,1)}}:=1,\boldsymbol{\chi_{(8,2)}}:=\sgn,\chi_{(8,3)}:=\chi_2,\chi_{(8,4)}:=\chi_3,\chi_{(8,5)}:=\chi_3'$\\\hline
$9$ & $\mathfrak{S}_4$&$\mathfrak{S}_4$&$\boldsymbol{\chi_{(9,1)}}:=1,\chi_{(9,2)}:=\sgn,\chi_{(9,3)}:=\chi_2,\chi_{(9,4)}:=\chi_3,\chi_{(9,5)}:=\chi_3'$\\\hline
\end{tabular}
\end{center}

\smallskip
Here, by abuse of notation, $1$ always denotes the trivial character, and $\sgn$ denotes the sign character, for each of the relevant groups. Moreover, $\nu_2$ is the natural
character of $\mathfrak{S}_3$ of degree 2, $\chi_3$ is the natural character of $\mathfrak{S}_4$ of degree 3, $\chi_3'=\chi_3\cdot\sgn$,
and $\chi_2$ is the unique irreducible $\mathfrak{S}_4$-character of degree 2.

As for the characters of $\Aut(D_8)\cong D_8$, we have $\deg(\tau)=\deg(\mu)=\deg(\mu')=1$ and
$\deg(\chi)=2$. Moreover, taking $D_8$ to be the Sylow 2-subgroup of $\mathfrak{S}_4$ generated by $(1,2)$ and $(1,3)(2,4)$, 
an explicit isomorphism $D_8\myiso \Aut(D_8)$ is given by the map
$$\varphi:D_8\to \Aut(D_8)\,; \begin{cases}  (1,2)\mapsto \begin{cases} (1,2)\mapsto (1,2)\\ (1,3)(2,4)\mapsto (1,4)(2,3) \end{cases}\\
                                                                  (1,3)(2,4)\mapsto\begin{cases}    (1,2)\mapsto (1,3)(2,4)\\ (1,3)(2,4)\mapsto (1,2)\,. \end{cases}
 \end{cases}$$
 With this convention, we get $\ker(\mu)=\langle  \varphi((1,2)),\varphi((3,4))  \rangle$,
 $\ker(\mu')=\langle  \varphi((1,2)(3,4)), \varphi((1,3)(2,4))\rangle$, and $\ker(\tau)$ is the unique cyclic subgroup of $\Aut(D_8)$ of order 4.

\smallskip
Now, whenever $i,j\in\{1,\ldots,9\}$ are such that $S_j<_\scrJ S_i$, that is, $G_j$ is isomorphic to a subquotient
of $G_i$, we proceed as follows, according to Corollary~\ref{cor cond (ii)}: we determine a set of representatives of the $\Aut(G_i)\times\Aut(G_j)$-orbits on
$$S_j\cap e_i*S*e_j=\{(P,K,\eta,G_j,1)\mid K\unlhd P\leq G_i: P/K\cong G_j\}\,.$$
For every such representative $L$, we decompose the permutation character $\Ind_{\stab(L)}^{\Aut(G_i)\times\Aut(G_j)}(1)$ into
a sum of irreducible characters. This determines the relation $\sqsubset$: if $\chi_{(i,r)}\times\chi_{(j,s)}^*$ is a constituent of the above permutation character then $(i,r)\sqsubset(j,s)$. We indicate this with an entry $1$ in Table~\ref{tab S_4}, and with $\cdot$ otherwise. 

As above, if $G_i\in\Ob(\catC)$ and if $\chi_{(i,r)}\in\mathrm{Irr}(\Aut(G_i))$ restricts trivially to $\Inn(G_i)$ then
we set the entries involving $\chi_{(i,r)}$ in boldface, since these characters lead precisely to the simple $A$-modules not 
annihilated by $\varepsilon$.
Note that all characters in the above table are self-dual. So, in our particular example, we have $\chi_{(j,s)}^*=\chi_{(j,s)}$, for all $(j,s)\in\Lambda$.

\begin{landscape}
\begin{table}
{\small \begin{tabular}{|c|c|c||c|c|cc|cc|ccc|ccc|ccccc|ccccc|ccccc|}\cline{3-30}
\multicolumn{2}{c|}{}&$i$&$1$&$2$&\multicolumn{2}{|c|}{$3$}&\multicolumn{2}{|c|}{$4$}&\multicolumn{3}{|c|}{$5$}&\multicolumn{3}{|c|}{$6$}&\multicolumn{5}{|c|}{$7$}&\multicolumn{5}{|c|}{$8$}&\multicolumn{5}{|c|}{$9$}\\\cline{3-30}

\multicolumn{2}{c|}{}&$G_i$&$\{1\}$&$C_2$&\multicolumn{2}{|c|}{$C_3$}&\multicolumn{2}{|c|}{$C_4$}&\multicolumn{3}{|c|}{$C_2^2$}&\multicolumn{3}{|c|}{$\mathfrak{S}_3$}&\multicolumn{5}{|c|}{$D_8$}&\multicolumn{5}{|c|}{$\mathfrak{A}_4$}&\multicolumn{5}{|c|}{$\mathfrak{S}_4$}\\\cline{3-30}

\multicolumn{2}{c|}{}&$\Aut(G_i)$&$\{1\}$&$\{1\}$&\multicolumn{2}{|c|}{$C_2$}&\multicolumn{2}{|c|}{$C_2$}&\multicolumn{3}{|c|}{$\mathfrak{S}_3$}&\multicolumn{3}{|c|}{$\mathfrak{S}_3$}&\multicolumn{5}{|c|}{$D_8$}&\multicolumn{5}{|c|}{$\mathfrak{S}_4$}&\multicolumn{5}{|c|}{$\mathfrak{S}_4$}\\\hline

$G_i$&$\Aut(G_i)$&$r$&{\bf 1}&{\bf 1}&{\bf 1}&{\bf  2}&{\bf 1}&{\bf 2}&{\bf  1}&{\bf 2}&{\bf 3}&{\bf 1}&2&3&{\bf 1}&2&{\bf 3}&4&5&{\bf 1}&{\bf 2}&3&4&5&{\bf 1}&2&3&4&5\\\hline\hline

$\{1\}$&$\{1\}$&{\bf 1}&{\bf 1}&{\bf 1}&{\bf 1}&$\cdot$&{\bf 1}&$\cdot$ &{\bf 1}&$\cdot$ &{\bf 1} & {\bf 1}&$\cdot$ &1 & {\bf 1}&$\cdot$ &{\bf 1}&$\cdot$& 1& {\bf 1}&$\cdot$ &1&1&$\cdot$&{\bf 1}&$\cdot$&1&1&$\cdot$\\\hline

$C_2$&$\{1\}$&{\bf 1}&$\cdot$& {\bf 1}&{\bf 1}&$\cdot$&{\bf 1}&$\cdot$&{\bf 1}&$\cdot$&{\bf 1}&{\bf 1}&$\cdot$& 1&{\bf 1}&$\cdot$&{\bf 1}&$\cdot$&1&{\bf 1}&$\cdot$ &1&$\cdot$&$\cdot$&{\bf 1}&$\cdot$&1&1&$\cdot$\\\hline

\multirow{2}{*}{$C_3$}&\multirow{2}{*}{$C_2$}&{\bf 1}&$\cdot$&$\cdot$&{\bf 1}&$\cdot$&$\cdot$&$\cdot$&$\cdot$&$\cdot$&$\cdot$&{\bf 1}&$\cdot$&$\cdot$&$\cdot$&$\cdot$&$\cdot$&$\cdot$&$\cdot$&{\bf 1}&$\cdot$&$\cdot$&1&$\cdot$&{\bf 1}&$\cdot$&$\cdot$&1&$\cdot$\\
&&{\bf 2}&$\cdot$&$\cdot$&$\cdot$&{\bf 1}&$\cdot$&$\cdot$&$\cdot$&$\cdot$&$\cdot$&$\cdot$&1&$\cdot$&$\cdot$&$\cdot$&$\cdot$&$\cdot$&$\cdot$&$\cdot$&{\bf 1}&$\cdot$&$\cdot$&1&$\cdot$&1&$\cdot$&$\cdot$&1\\\hline

\multirow{2}{*}{$C_4$}&\multirow{2}{*}{$C_2$}&{\bf 1}&$\cdot$&$\cdot$&$\cdot$&$\cdot$&{\bf 1}&$\cdot$&$\cdot$&$\cdot$&$\cdot$&$\cdot$&$\cdot$&$\cdot$&{\bf 1}&$\cdot$&$\cdot$&$\cdot$&$\cdot$&$\cdot$&$\cdot$&$\cdot$&$\cdot$&$\cdot$&{\bf 1}&$\cdot$&1&$\cdot$&$\cdot$\\
&&{\bf 2}&$\cdot$&$\cdot$&$\cdot$&$\cdot$&$\cdot$&{\bf 1}&$\cdot$&$\cdot$&$\cdot$&$\cdot$&$\cdot$&$\cdot$&$\cdot$&1&$\cdot$&$\cdot$&$\cdot$&$\cdot$&$\cdot$&$\cdot$&$\cdot$&$\cdot$&$\cdot$&$\cdot$&$\cdot$&$\cdot$&1\\\hline

\multirow{3}{*}{$C_2^2$}&\multirow{3}{*}{$\mathfrak{S}_3$}&{\bf 1}&$\cdot$&$\cdot$&$\cdot$&$\cdot$&$\cdot$&$\cdot$&{\bf 1}&$\cdot$&$\cdot$&$\cdot$&$\cdot$&$\cdot$&{\bf 1}&$\cdot$&{\bf 1}&$\cdot$&$\cdot$&{\bf 1}&$\cdot$&$\cdot$&$\cdot$&$\cdot$&{\bf 1}&$\cdot$&1&$\cdot$&$\cdot$\\
&&{\bf 2}&$\cdot$&$\cdot$&$\cdot$&$\cdot$&$\cdot$&$\cdot$&$\cdot$&{\bf 1}&$\cdot$&$\cdot$&$\cdot$&$\cdot$&$\cdot$&$\cdot$&{\bf 1}&$\cdot$&1&$\cdot$&{\bf 1}&$\cdot$&$\cdot$&$\cdot$&{\bf 1}&1&1&1&$\cdot$\\
&&{\bf 3}&$\cdot$&$\cdot$&$\cdot$&$\cdot$&$\cdot$&$\cdot$&$\cdot$&$\cdot$&{\bf 1}&$\cdot$&$\cdot$&$\cdot$&{\bf 1}&$\cdot$&{\bf 1}&$\cdot$&1&$\cdot$&$\cdot$&1&$\cdot$&$\cdot$&{\bf 1}&$\cdot$&1&1&$\cdot$\\\hline

\multirow{3}{*}{$\mathfrak{S}_3$}&\multirow{3}{*}{$\mathfrak{S}_3$}&{\bf 1}&$\cdot$&$\cdot$&$\cdot$&$\cdot$&$\cdot$&$\cdot$&$\cdot$&$\cdot$&$\cdot$&{\bf 1}&$\cdot$&$\cdot$&$\cdot$&$\cdot$&$\cdot$&$\cdot$&$\cdot$&$\cdot$&$\cdot$&$\cdot$&$\cdot$&$\cdot$&{\bf 1}&$\cdot$&$\cdot$&1&$\cdot$\\
&& 2&$\cdot$&$\cdot$&$\cdot$&$\cdot$&$\cdot$&$\cdot$&$\cdot$&$\cdot$&$\cdot$&$\cdot$&1&$\cdot$&$\cdot$&$\cdot$&$\cdot$&$\cdot$&$\cdot$&$\cdot$&$\cdot$&$\cdot$&$\cdot$&$\cdot$&$\cdot$&1&$\cdot$&$\cdot$&1\\
&&3&$\cdot$&$\cdot$&$\cdot$&$\cdot$&$\cdot$&$\cdot$&$\cdot$&$\cdot$&$\cdot$&$\cdot$&$\cdot$&1&$\cdot$&$\cdot$&$\cdot$&$\cdot$&$\cdot$&$\cdot$&$\cdot$&$\cdot$&$\cdot$&$\cdot$&$\cdot$&$\cdot$&1&1&1\\\hline

\multirow{5}{*}{$D_8$}&\multirow{5}{*}{$D_8$}&{\bf 1}&$\cdot$&$\cdot$&$\cdot$&$\cdot$&$\cdot$&$\cdot$&$\cdot$&$\cdot$&$\cdot$&$\cdot$&$\cdot$&$\cdot$&{\bf 1}&$\cdot$&$\cdot$&$\cdot$&$\cdot$&$\cdot$&$\cdot$&$\cdot$&$\cdot$&$\cdot$&{\bf 1}&$\cdot$&1&$\cdot$&$\cdot$\\
&&2&$\cdot$&$\cdot$&$\cdot$&$\cdot$&$\cdot$&$\cdot$&$\cdot$&$\cdot$&$\cdot$&$\cdot$&$\cdot$&$\cdot$&$\cdot$&1&$\cdot$&$\cdot$&$\cdot$&$\cdot$&$\cdot$&$\cdot$&$\cdot$&$\cdot$&$\cdot$&$\cdot$&$\cdot$&$\cdot$&1\\
&&{\bf 3}&$\cdot$&$\cdot$&$\cdot$&$\cdot$&$\cdot$&$\cdot$&$\cdot$&$\cdot$&$\cdot$&$\cdot$&$\cdot$&$\cdot$&$\cdot$&$\cdot$&{\bf 1}&$\cdot$&$\cdot$&$\cdot$&$\cdot$&$\cdot$&$\cdot$&$\cdot$&{\bf 1}&$\cdot$&1&$\cdot$&$\cdot$\\
&&4&$\cdot$&$\cdot$&$\cdot$&$\cdot$&$\cdot$&$\cdot$&$\cdot$&$\cdot$&$\cdot$&$\cdot$&$\cdot$&$\cdot$&$\cdot$&$\cdot$&$\cdot$&1&$\cdot$&$\cdot$&$\cdot$&$\cdot$&$\cdot$&$\cdot$&$\cdot$&$\cdot$&$\cdot$&$\cdot$&1\\
&&5&$\cdot$&$\cdot$&$\cdot$&$\cdot$&$\cdot$&$\cdot$&$\cdot$&$\cdot$&$\cdot$&$\cdot$&$\cdot$&$\cdot$&$\cdot$&$\cdot$&$\cdot$&$\cdot$&1&$\cdot$&$\cdot$&$\cdot$&$\cdot$&$\cdot$&$\cdot$&1&1&1&$\cdot$\\\hline

\multirow{5}{*}{$\mathfrak{A}_4$}&\multirow{5}{*}{$\mathfrak{S}_4$}&{\bf 1}&$\cdot$&$\cdot$&$\cdot$&$\cdot$&$\cdot$&$\cdot$&$\cdot$&$\cdot$&$\cdot$&$\cdot$&$\cdot$&$\cdot$&$\cdot$&$\cdot$&$\cdot$&$\cdot$&$\cdot$&{\bf 1}&$\cdot$&$\cdot$&$\cdot$&$\cdot$&{\bf 1}&$\cdot$&$\cdot$&$\cdot$&$\cdot$\\
&&{\bf 2}&$\cdot$&$\cdot$&$\cdot$&$\cdot$&$\cdot$&$\cdot$&$\cdot$&$\cdot$&$\cdot$&$\cdot$&$\cdot$&$\cdot$&$\cdot$&$\cdot$&$\cdot$&$\cdot$&$\cdot$&$\cdot$&{\bf 1}&$\cdot$&$\cdot$&$\cdot$&$\cdot$&1&$\cdot$&$\cdot$&$\cdot$\\
&&3&$\cdot$&$\cdot$&$\cdot$&$\cdot$&$\cdot$&$\cdot$&$\cdot$&$\cdot$&$\cdot$&$\cdot$&$\cdot$&$\cdot$&$\cdot$&$\cdot$&$\cdot$&$\cdot$&$\cdot$&$\cdot$&$\cdot$&1&$\cdot$&$\cdot$&$\cdot$&$\cdot$&1&$\cdot$&$\cdot$\\
&&4&$\cdot$&$\cdot$&$\cdot$&$\cdot$&$\cdot$&$\cdot$&$\cdot$&$\cdot$&$\cdot$&$\cdot$&$\cdot$&$\cdot$&$\cdot$&$\cdot$&$\cdot$&$\cdot$&$\cdot$&$\cdot$&$\cdot$&$\cdot$&1&$\cdot$&$\cdot$&$\cdot$&$\cdot$&1&$\cdot$\\
&&5&$\cdot$&$\cdot$&$\cdot$&$\cdot$&$\cdot$&$\cdot$&$\cdot$&$\cdot$&$\cdot$&$\cdot$&$\cdot$&$\cdot$&$\cdot$&$\cdot$&$\cdot$&$\cdot$&$\cdot$&$\cdot$&$\cdot$&$\cdot$&$\cdot$&1&$\cdot$&$\cdot$&$\cdot$&$\cdot$&1\\\hline

\multirow{5}{*}{$\mathfrak{S}_4$}&\multirow{5}{*}{$\mathfrak{S}_4$}&{\bf 1}&$\cdot$&$\cdot$&$\cdot$&$\cdot$&$\cdot$&$\cdot$&$\cdot$&$\cdot$&$\cdot$&$\cdot$&$\cdot$&$\cdot$&$\cdot$&$\cdot$&$\cdot$&$\cdot$&$\cdot$&$\cdot$&$\cdot$&$\cdot$&$\cdot$&$\cdot$&{\bf 1}&$\cdot$&$\cdot$&$\cdot$&$\cdot$\\
&&2&$\cdot$&$\cdot$&$\cdot$&$\cdot$&$\cdot$&$\cdot$&$\cdot$&$\cdot$&$\cdot$&$\cdot$&$\cdot$&$\cdot$&$\cdot$&$\cdot$&$\cdot$&$\cdot$&$\cdot$&$\cdot$&$\cdot$&$\cdot$&$\cdot$&$\cdot$&$\cdot$&1&$\cdot$&$\cdot$&$\cdot$\\
&&3&$\cdot$&$\cdot$&$\cdot$&$\cdot$&$\cdot$&$\cdot$&$\cdot$&$\cdot$&$\cdot$&$\cdot$&$\cdot$&$\cdot$&$\cdot$&$\cdot$&$\cdot$&$\cdot$&$\cdot$&$\cdot$&$\cdot$&$\cdot$&$\cdot$&$\cdot$&$\cdot$&$\cdot$&1&$\cdot$&$\cdot$\\
&&4&$\cdot$&$\cdot$&$\cdot$&$\cdot$&$\cdot$&$\cdot$&$\cdot$&$\cdot$&$\cdot$&$\cdot$&$\cdot$&$\cdot$&$\cdot$&$\cdot$&$\cdot$&$\cdot$&$\cdot$&$\cdot$&$\cdot$&$\cdot$&$\cdot$&$\cdot$&$\cdot$&$\cdot$&$\cdot$&1&$\cdot$\\
&&5&$\cdot$&$\cdot$&$\cdot$&$\cdot$&$\cdot$&$\cdot$&$\cdot$&$\cdot$&$\cdot$&$\cdot$&$\cdot$&$\cdot$&$\cdot$&$\cdot$&$\cdot$&$\cdot$&$\cdot$&$\cdot$&$\cdot$&$\cdot$&$\cdot$&$\cdot$&$\cdot$&$\cdot$&$\cdot$&$\cdot$&1\\\hline
\end{tabular}}
\caption{The relation $\sqsubseteq$ in Example~\ref{expl S_4}}
\label{tab S_4}
\end{table}
\end{landscape}

For instance, letting $i:=9$ and $j:=6$, we have $G_i=\mathfrak{S}_4$ and $G_j=\mathfrak{S}_3$, thus
$S_j<_\scrJ S_i$. Representatives of the $\Aut(G_i)\times \Aut(G_j)$-orbits 
of $S_j\cap e_i*S*e_j$ are given by $L:=\Delta(\mathfrak{S}_3)$ and $M:=(\mathfrak{S}_4,V_4,\eta,\mathfrak{S}_3,1)$,
where $V_4$ is the normal Klein four-group in $\mathfrak{S}_4$ and $\eta:\mathfrak{S}_3\to\mathfrak{S}_4/V_4$ is
any fixed isomorphism. 
We have $\Aut(G_i)=\Inn(G_i)\cong G_i$ and $\Aut(G_j)=\Inn(G_j)\cong G_j$.
Identifying $G_i$ with $\Aut(G_i)$ and $G_j$ with $\Aut(G_j)$, we get $\stab(L)=\Delta(\mathfrak{S}_3)$
and $\stab(M)=\Delta(\mathfrak{S}_3)(1\times V_4)$. Furthermore,
$$\Ind_{\stab(L)}^{\mathfrak{S}_4\times\mathfrak{S}_3}(1)=1\times 1+\sgn\times \sgn+\chi_3\times 1+\chi_3'\times\sgn+\chi_2\times \nu_2+\chi_3\times \nu_2+\chi_3'\times\nu_2\,,$$
and
$$\Ind_{\stab(M)}^{\mathfrak{S}_4\times\mathfrak{S}_3}(1)=\chi_2\times\nu_2+1\times 1+\sgn\times \sgn\,.$$

\medskip
(b)\, Now Table~\ref{tab S_4} shows that the relation $\sqsubseteq$ on $\Lambda$
is not transitive, since $(9,2)\sqsubset (7,5)$
and $(7,5)\sqsubset (1,1)$, whereas $(9,2)\not\sqsubset (1,1)$.

\medskip
(c)\, From Table~\ref{tab S_4} we, moreover, see that the partial order $\unlhd$ on $\Lambda$ is indeed coarser than $\leq$:
consider $i=8$ and $j=3$, so that $G_i=\mathfrak{A}_4$ and $G_j=C_3$. The group $G_j$ can either be realized as a maximal
subgroup or as a minimal quotient group of $\mathfrak{A}_4$. So we infer that $(i,r)\sqsubset (j,s)$ if and only if
$(i,r)\lhd (j,s)$, for $i=8$, $j=3$, $1\leq r\leq 5$, $1\leq s\leq 2$. By Table~\ref{tab S_4}, we have
$$(i,r)\sqsubset (j,s)\Leftrightarrow (r,s)\in\{(1,1),(1,4),(2,2),(2,5)\}\,,$$
whereas $(i,r)<(j,s)$, for all $1\leq r\leq 5$, $1\leq s\leq 2$. 

After condensation with $\varepsilon$, we obtain $(i,r)\lhd (j,s)$ in $(\Lambda',\unlhd)$ if and only if $(r,s)\in\{(1,1),(2,2)\}$, but $(i,r)<(j,s)$ for all combinations of $r,s\in\{1,2\}$.

\medskip

(d)\, As indicated in Remark~\ref{rem conn with biset functors}, Theorem~\ref{thm condense}(b) does, in general, not hold
with $\leq$ instead of $\unlhd$. To see this in our current example, take $(i,r):=(9,1)$ and $(j,s):=(8,3)$. Then $(i,r)<(j,s)$, since 
$G_j=\mathfrak{A}_4$ is isomorphic to a subquotient of $G_i=\mathfrak{S}_4$. From Proposition~\ref{prop J-classes of C}(c)
we deduce that $\varepsilon\cdot D_{(i,r)}\neq \{0\}$, but $\varepsilon\cdot D_{(j,s)}= \{0\}$.
\end{example}

%%%%%%%%%%%%%%%%%%%%%%%%%%%%%%%%%%%%%%%%%%%%%%%%%%%%%%%%%%%%%%%%

\section{Application II: Brauer Algebras}\label{sec brauer}

As mentioned in the introduction, several classes of diagram algebras arise naturally as twisted category algebras. 
In fact, in all these examples one deals with monoid algebras, that is, the underlying category has only one object.

Throughout this section, let $n\in\mathbb{N}$, and let $\mathfrak{S}_n$ be the symmetric group of degree $n$. Permutations
in $\mathfrak{S}_n$ will always be composed from right to left, so that, for instance, we have $(1,2)(2,3)=(1,2,3)\in\mathfrak{S}_n$, whenever $n\geq 3$.
Moreover, let $k$ be a field such that $n!$ is invertible in $k$, and let
$\delta\in k^\times$. 

%{\tt We only state the results for the `single' Brauer algebra. If one wanted to see the analogy to
%the biset functor example, one could, e.g., take a category whose objects form a finite set of natural numbers. 
%To make the notation as smooth as possible, one could suppose that all these natural numbers have the same residue
%modulo 2. A morphism between $n$ and $m$ would then be an $(n,m)$-Brauer diagram. 
%If $n$ is the maximum of the objects then we could take the `same' $(n,n)$-Brauer diagram idempotents 
%as for the Brauer algebra; these would still be representatives of the possible J-classes. (Here we need the assumption
%on the common parity of all objects!) With this all the arguments below should go through. 
%However, it is apparently not clear whether this generalized twisted category algebra is of much interest to people
%studying Brauer algebras.
%}

\begin{nothing}\label{noth brauer}
{\bf Brauer algebras.}\, 
Consider the set $S$ of {\it $(n,n)$-Brauer diagrams}; each of these consists of $n$ {\it northern nodes}, labelled by
$1,\ldots,n$, and
$n$ {\it southern} nodes, labelled by $\bar{1},\ldots,\bar{n}$, and each node is connected by an edge to precisely one other node. Egdes connecting a pair of nothern or southern nodes are called {\it arcs}, and edges connecting a northern with a southern node are called {\it propagating lines}. In other words, the elements of $S$ can be viewed as equivalence relations of the set
$\{1,\ldots,n\}\cup\{\bar{1},\ldots,\bar{n}\}$ whose equivalence classes contain precisely two elements.

Given $(n,n)$-Brauer diagrams $t$ and $t'$, their composition $t\circ t'$ is defined by first taking the {\it concatenation of $t$ above
$t'$}, and then deleting all cycles from the resulting diagram.

\smallskip

For instance, suppose that $n=6$, $t=\{\{1,\bar{1}\},\{2,\bar{4}\},\{3,\bar{2}\},\{4,6\},\{5,\bar{3}\},\{\bar{5},\bar{6}\}\}$,
and $t'=\{\{1,\bar{1}\},\{2,\bar{2}\},\{3,4\},\{5,6\},\{\bar{3},\bar{6}\},\{\bar{4},\bar{5}\}\}$. Then we get:

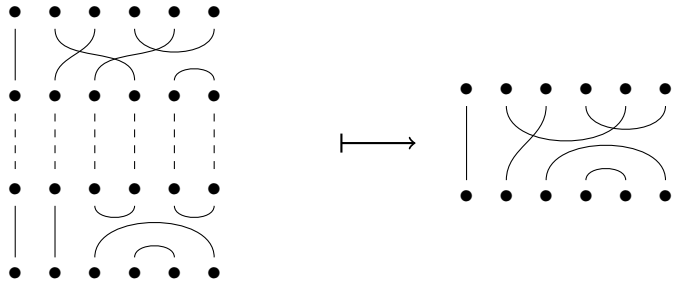
\begin{center}
\begin{tikzpicture}[scale=1,transform shape,every node/.style={scale=1}]
    \begin{scope}
      \matrix (A) [matrix of math nodes,column sep=0.8 mm] 
      { \bullet & \bullet & \bullet & \bullet & \bullet & \bullet  \\[0.7cm]
        \bullet & \bullet & \bullet & \bullet & \bullet & \bullet  \\[0.8cm]
        \bullet & \bullet & \bullet & \bullet & \bullet & \bullet  \\[0.7cm]
        \bullet & \bullet & \bullet & \bullet & \bullet & \bullet  \\
      };

      \draw (A-1-1) to [out=270,in=90] (A-2-1);
      \draw (A-1-2) to [out=270,in=90]  (A-2-4);
      \draw (A-1-3) to [out=270,in=90]  (A-2-2);
      \draw (A-1-4) to [out=270,in=270] (A-1-6);
      \draw (A-1-5) to [out=270,in=90] (A-2-3);
      \draw (A-2-5) to [out=90,in=90]  (A-2-6);

      \draw[dashed] (A-2-1) -- (A-3-1);
      \draw[dashed] (A-2-2) -- (A-3-2);
      \draw[dashed] (A-2-3) -- (A-3-3);
      \draw[dashed] (A-2-4) -- (A-3-4);
      \draw[dashed] (A-2-5) -- (A-3-5);
      \draw[dashed] (A-2-6) -- (A-3-6);

      \draw (A-3-1) to (A-4-1);
      \draw (A-3-2) to (A-4-2);
      \draw (A-3-3) to [out=270,in=270] (A-3-4);
      \draw (A-3-5) to [out=270,in=270] (A-3-6);
      \draw (A-4-3) to [out=90,in=90] (A-4-6);
      \draw (A-4-4) to [out=90,in=90] (A-4-5);
    \end{scope}

   \draw[thick,|->] (3,0) -- (4,0);

    \begin{scope}[xshift=6cm]
      \matrix (B) [matrix of math nodes,column sep=0.8 mm] 
      { \bullet & \bullet & \bullet & \bullet & \bullet & \bullet  \\[1cm]
        \bullet & \bullet & \bullet & \bullet & \bullet & \bullet  \\
      };

      \draw (B-1-1) to (B-2-1);
      \draw (B-1-2) to [out=270,in=270] (B-1-5);
      \draw (B-1-3) to [out=270,in=90]  (B-2-2);
      \draw (B-1-4) to [out=270,in=270] (B-1-6);

      \draw (B-2-3) to [out=90,in=90] (B-2-6);
      \draw (B-2-4) to [out=90,in=90] (B-2-5);

          \end{scope}
  \end{tikzpicture}
\end{center}

\smallskip

Hence $t\circ t'=\{\{1,\bar{1}\},\{2,5\},\{3,\bar{2}\},\{4,6\},\{\bar{3},\bar{6}\},\{\bar{4},\bar{5}\}\}$.

\bigskip

In this way, $S$ becomes a  finite monoid whose identity element is the diagram that connects
each pair of opposite nodes by a propagating line.

The map
\begin{equation}\label{eqn delta}
\alpha:S\times S\to k^\times,\; (t,t')\mapsto \delta^{m(t,t')}\,,
\end{equation}
where $m(t,t')$ is the number of cycles in the concatenation of $t$ above $t'$, defines a 2-cocycle of the monoid $S$ with values 
in $k^\times$. The resulting twisted monoid algebra
\begin{equation}\label{eqn brauer alg}
B_n(\delta):=k_\alpha S
\end{equation}
is called the {\it Brauer algebra} of degree $n$ over $k$ with parameter $\delta$.

\medskip
The $\scrJ$-classes of the monoid $S$ have been determined by Mazorchuk in \cite{Maz}; we shall recall the result in Proposition~\ref{prop brauer} below. In order to do so, it will be convenient to use the following notation: let $d:=\lfloor\frac{n}{2}\rfloor$, and for
$i=1,\ldots,d+1$, let $e_i$ be the  diagram each of whose $n-2(d-i+1)$ leftmost northern nodes is joined
to its opposite southern node by a propagating line; the remaining $2(d-i+1)$
edges of $e_i$ are arcs, each connecting a pair of
neighbouring northern or southern nodes. That is, $e_i$ has shape

\medskip
\begin{center}
\begin{tikzpicture}[scale=1,transform shape,every node/.style={scale=1}]
   \begin{scope}
      \matrix (A) [matrix of math nodes,column sep=1 mm] 
      { \bullet & \bullet & \cdots &\bullet& \bullet & \bullet & \bullet &\bullet&\cdots&\bullet&\bullet \\[0.75cm]
        \bullet & \bullet & \cdots &\bullet & \bullet & \bullet & \bullet  &\bullet&\cdots&\bullet&\bullet \\
              };

      \draw (A-1-1) to (A-2-1);
      \draw (A-1-2) to   (A-2-2);
      \draw (A-1-4) to   (A-2-4);
      \draw (A-1-5) to [out=270,in=270] (A-1-6);
      \draw (A-1-7) to [out=270,in=270] (A-1-8);
       \draw (A-1-10) to [out=270,in=270] (A-1-11);
      \draw (A-2-5) to [out=90,in=90] (A-2-6);
       \draw (A-2-7) to [out=90,in=90] (A-2-8);
       \draw (A-2-10) to [out=90,in=90] (A-2-11);     
   \end{scope} 
 \end{tikzpicture}
\end{center}
\end{nothing}

\begin{proposition}[\protect{\cite[Theorem~7]{Maz},\cite[Section~8]{Wil}}]\label{prop brauer} Keep the notation as in \ref{noth brauer}. Then, one has the following:

\smallskip

{\rm (a)}\, Brauer diagrams $t,t'\in S$ belong to the same $\scrJ$-class if and only if they have the same number of propagating lines.
Moreover, the diagrams $e_1,\ldots,e_{d+1}$ are idempotents in $S$, and they form a set of representatives
of the distinct $\scrJ$-classes of $S$. Furthermore,
$$\scrJ(e_1)<_\scrJ \scrJ(e_2)<_\scrJ\cdots <_\scrJ \scrJ(e_{d+1})\,.$$

\smallskip
{\rm (b)}\, For $i=1,\ldots,d+1$, the group $\Gamma_{e_i}$ consists of those diagrams $t\in S$ with arcs
$$\{n-2j+1,n-2j+2\},\, \{\overline{n-2j+1},\overline{n-2j+1}\}\quad (j=1,\ldots, d-i+1)\,,$$
and whose remaining northern and southern nodes are connected by propagating lines. In parti\-cular, 
there is a group isomorphism
\begin{equation}\label{eqn Gamma iso}
\mathfrak{S}_{n_i}\to\Gamma_{e_i},\; \sigma\mapsto t_\sigma\,,
\end{equation}
where $n_i:=n-2(d-i+1)$, and $t_\sigma$ is the Brauer diagram in $\Gamma_{e_i}$ with propagating lines
$\{\sigma(1),\bar{1}\},\ldots,\{\sigma(n_i),\overline{n_i}\}$.

\smallskip
{\rm (c)}\, For $i\in\{ 1,\ldots, d+1\}$ and $x,x'\in \Gamma_{e_i}$, one has
$\alpha(x,x')=\alpha(e_i,e_i)=\delta^{d-i+1}$;
in particular, $\alpha$ restricts to a constant $2$-cocycle on $\Gamma_{e_i}$, and the map
\begin{equation}\label{eqn S_n iso}
k_\alpha\Gamma_{e_i}\to k\Gamma_{e_i},\; t\mapsto \delta^{d-i+1}\cdot t
\end{equation}
is a $k$-algebra isomorphism.

\smallskip
{\rm (d)}\, For each $t\in S$, let $t^\circ$ be the diagram that is obtained 
by reflecting $t$ about the horizontal axis. Then the resulting map $-^\circ:S\to S$, $t\mapsto t^\circ$, satisfies 
Hypotheses~\ref{hypo op}(ii)--(v),
with respect to the $2$-cocycle $\alpha$  in (\ref{eqn delta}).
In particular, $S$ is a regular monoid. 
\end{proposition}

\begin{remark}\label{rem brauer}
In accordance with our notation in Section~\ref{sec twisted cat}, we again set $S_i:=\scrJ(e_i)$, for $i=1,\ldots,d+1$, so that, by 
Proposition~\ref{prop brauer}(a), $S_1,\ldots,S_{d+1}$ are the distinct $\scrJ$-classes of $S$.

Furthermore, for $i=1,\ldots,d+1$, we may identify $\Gamma_{e_i}$ with the symmetric group $\mathfrak{S}_{n_i}$, via the
isomorphism (\ref{eqn Gamma iso}), and the twisted group algebra
$k_\alpha\Gamma_{e_i}$ with the untwisted group algebra $k\mathfrak{S}_{n_i}$ via the isomorphism in 
(\ref{eqn S_n iso}). Note that, since we are assuming $n!\in k^\times$, we also ensure that the group orders 
$|\mathfrak{S}_{n_1}|,\ldots,|\mathfrak{S}_{n_{d+1}}|$ are invertible in $k$. Hence, the isomorphism classes of simple
$k\mathfrak{S}_{n_i}$-modules are parametrized by the partitions of $n_i$. More precisely, suppose that $\{\lambda_{(i,1)},\ldots,\lambda_{(i,l_i)}\}$ is the set of partitions of $n_i$. Then, for $r=1,\ldots,l_i$, the simple $k\mathfrak{S}_{n_i}$-module
$T_{(i,r)}$ can be chosen to be the {\it Specht $k\mathfrak{S}_{n_i}$-module} $S^{\lambda_{(i,r)}}$.
For details concerning the representation theory of symmetric groups, we refer to \cite{J} and \cite{JK}.

\smallskip

Now, Theorem~\ref{thm main} and Corollary~\ref{cor main} as well as the results from Sections~\ref{sec sqsubseteq} and \ref{sec duality} apply, and we obtain the following result.
Here we again set $d:=\lfloor\frac{n}{2}\rfloor$, so that $\Lambda=\{(i,r)\mid 1\leq i\leq d+1,\, 1\leq r\leq l_i\}$. 
\end{remark}

\begin{theorem}\label{thm brauer qh}
The Brauer algebra $B_\delta(n)$ is quasi-hereditary with respect to $(\Lambda,\unlhd)$. 
The corresponding standard $B_\delta(n)$-module labelled by $(i,r)\in\Lambda$ is
isomorphic to $\Delta_{(i,r)}:=B_\delta(n)e_i'\otimes_{e_i'B_\delta(n)e_i'} \tilde{T}_{(i,r)}$, and the
costandard $B_\delta(n)$-module labelled by $(i,r)$ is isomorphic to $\nabla_{(i,r)}\cong \Hom_{e_i'B_\delta(n)e_i'}(e_i'B_\delta(n), \Ttilde_{(i,r)})\cong  \Delta_{(i,r)}^\circ$. 
Moreover, every simple $B_\delta(n)$-module is self-dual with respect to 
the map $-^\circ:S\to S$ in Proposition~\ref{prop brauer} and 
the resulting duality introduced in \ref{noth duals}; in particular, $B_\delta(n)$ is a BGG-algebra.
\end{theorem}

\begin{proof}
By Theorem~\ref{thm main} and Proposition~\ref{prop brauer}, $B_\delta(n)$ is quasi-hereditary with respect to $(\Lambda,\unlhd)$, and the standard modules
are as claimed. Moreover, the costandard modules are as claimed, by Corollary~\ref{cor main} and Proposition~\ref{prop costandard A-modules}, taking into account that $k_\alpha\Gamma_{e_i}\cong k\mathfrak{S}_{n_i}$ and that every simple $k\mathfrak{S}_{n_i}$-module
is self-dual; see \cite[Theorem~11.5]{J}. Finally, since, by Proposition~\ref{prop costandard A-modules} again,
$D_{(i,r)}$ is the head of $\Delta_{(i,r)}$, and $D_{(i,r)}^\circ$ is isomorphic to the head of
$\Delta_{(i,r)^\circ}\cong \Delta_{(i,r^\circ)}\cong \Delta_{(i,r)}$, we get $D_{(i,r)}\cong D_{(i,r)}^\circ$.
\end{proof}

\begin{remark}\label{rem young}
It has been known that, under our assumptions on $k$ and $\delta$, the $k$-algebra $B_\delta(n)$ is quasi-hereditary, see~\cite[Theorem~1.3]{KX}. The underlying partial order on the set $\Lambda$ that is usually considered is the one in which $(i,r)$ is strictly smaller than $(j,s)$ if and only if $j<i$. As shown in Proposition~\ref{prop order brauer} (a), this partial order coincides with the partial order $\leq$ from (\ref{eqn leq}). Part~(c) of Proposition~\ref{prop order brauer} determines the partial order $\unlhd$ from Definition~\ref{defi new order} explicitly. Here, for $i=1,\ldots,d+1$, we again
identify the subgroup $\Gamma_{e_i}$ of $S$ with the symmetric group $\mathfrak{S}_{n_i}$ via the
isomorphism (\ref{eqn Gamma iso}). Furthermore, if $j\in\{1,\ldots,i-1\}$ is such that $j<i$ then we denote by
$\mathfrak{S}_{n_j}\times (\mathfrak{S}_2)^{i-j}$ the standard Young subgroup
$$\mathfrak{S}_{n_j}\times \langle (n_j+1,n_j+2)\rangle\times \langle (n_j+3,n_j+4)\rangle\times\cdots\times\langle (n_i-1,n_i)\rangle$$
of $\mathfrak{S}_{n_i}$.
\end{remark}

\begin{proposition}\label{prop order brauer}
Let $(i,r),(j,s)\in\Lambda$. Then one has the following:

\smallskip
{\rm (a)}\, $(i,r)<(j,s)$ if and only if $j<i$.

\smallskip
{\rm (b)}\, For all $x\in \Gamma_{e_i}$, $t\in e_i\circ S$ and $u\in S\circ e_i$, one has $\alpha(x,t)=\delta^{d-i+1}=\alpha(u,x)$.
If $j<i$, then $S_j\cap e_i\circ S\circ e_j$ is a transitive $\mathfrak{S}_{n_i}\times \mathfrak{S}_{n_j}$-set
via the action defined in \ref{nota perm action}
and the isomorphism (\ref{eqn Gamma iso}). Moreover, in this case one also has $e_j\in S_j\cap e_i\circ S\circ e_j$.

\smallskip
%{\rm (c)}\, One has 
%$$(i,r)\sqsubset (j,s)\Leftrightarrow j<i\text{ and } \left(\chi_{(j,s)}\times\chi_{(i,r)},\Ind_{L_{(j,i)}}^{\mathfrak{S}_{n_j}\times\mathfrak{S}_{n_i}}(1)\right)_{\mathfrak{S}_{n_j}\times\mathfrak{S}_{n_i}}\neq 0\,,$$
%where $L_{(j,i)}:=\stab_{\mathfrak{S}_{n_j}\times \mathfrak{S}_{n_i}}(e_j)$.
%\smallskip
{\rm (c)}\, One has
$$(i,r)\lhd (j,s)\Leftrightarrow j<i\text{ and } T_{(i,r)}\big\arrowvert M_{(j,s)}^{(i)}\,,$$
where $M_{(j,s)}^{(i)}$ denotes the $k\mathfrak{S}_{n_i}$-module
$$M_{(j,s)}^{(i)}:=\Ind_{\mathfrak{S}_{n_j}\times\mathfrak{S}_2\times\cdots\times\mathfrak{S}_2}^{\mathfrak{S}_{n_i}}(T_{(j,s)}\times k\times\cdots\times k)\,.$$
\end{proposition}

Before proving the proposition, we mention (without proof) the following well-known lemma that will be used repeatedly in the proof below. As usual, given finite groups $G$ and $H$, we identify left $k[G\times H]$-modules with $(kG,kH)$-bimodules, and vice versa.

\begin{lemma}\label{lemma bisets}
Let $G$ and $H$ be finite groups such that $|G|$ and $|H|$ are invertible in $k$ and assume that $k$ is a splitting field for $G\times H$. Let $M$ be a left $k[G\times H]$-module and assume that $X$ is an irreducible left $kG$-module and $Y$ is an irreducible left $kH$-module. Then, $X\otimes_k Y^*$ is a constituent of $M$ if and only if $X$ is a constituent of $M\otimes_{kH} Y$.
\end{lemma}

We are now prepared to prove Proposition~\ref{prop order brauer}:

\medskip
\begin{proof}{\bf of Proposition~\ref{prop order brauer}:}\;
Assertion~(a) is clear, by the description of the $\scrJ$-classes of $S$ in Proposition~\ref{prop brauer}(a).

\smallskip

To prove the first assertion in (b), let $x\in \Gamma_{e_i}$ and $t\in e_i\circ S$. Note that cycles in the construction of $x\circ t$ can only result from southern arcs of $x$ and northern arcs of $t$. Since $x\in \Gamma_{e_i}$, and $t\in e_i\circ S$, the diagram of $x$ has precisely $d-i+1$ southern arcs connecting consecutive nodes starting from the right, and $t$ has the $d-i+1$ matching northern arcs (and possibly more, which are irrelevant). Thus there are precisely $d-i+1$ cycles, so that $\alpha(x,t)=\delta^{d-i+1}$. Similarly, we obtain $\alpha(u,x)=\delta^{d-i+1}$, for $u\in S\circ e_i$.

Now suppose that $j<i$, so that also $n_j=n-2(d-j+1)<n-2(d-i+1)=n_i$; in particular, $\mathfrak{S}_{n_j}<\mathfrak{S}_{n_i}$.
Since $e_j=e_i\circ e_j$, we have $e_j\in S_j\cap e_i\circ S\circ e_j$. Now let $t\in S_j\cap e_i\circ S\circ e_j$ be arbitrary.
Then $t$ has precisely $n_j$ propagating lines, each connecting one of the $n_j$ leftmost southern nodes with one
of the $n_i$ leftmost northern nodes. In other words, there is an injection
$\iota:\{1,\ldots,n_j\}\to\{1,\ldots,n_i\}$ such that  $t$ has the following propagating lines:
$$\{\iota(1),\bar{1}\},\ldots,\{\iota(n_j),\bar{n}_j\}\,.$$
Multiplying $t$ from the left by a suitable permutation in $\mathfrak{S}_{n_i}$, we may suppose that
$\iota(q)\leq n_j$, for all $q=1,\ldots,n_j$. Then, for $\sigma\in\mathfrak{S}_{n_i}$ with
$\sigma(\iota(m))=m$ for $m=1,\ldots,n_j$ and $\sigma(m)=m$ for $m=n_j+1,\ldots,n_i$, we get
$e_j=\sigma\cdot t$. Therefore, $S_j\cap e_i\circ S\circ e_j$ is actually a transitive left $\mathfrak{S}_{n_i}$-set, thus
also a transitive left $\mathfrak{S}_{n_i}\times\mathfrak{S}_{n_j}$-set.

\medskip
It remains to verify assertion~(c). To this end, we first determine when $(i,r)\sqsubset (j,s)$ holds.  Note that the hypothesis of Corollaries~\ref{cor cond (ii)} and \ref{cor cond (ii)'} are satisfied by Proposition~\ref{prop brauer}(c) and by part~(b). Therefore, we obtain
\begin{equation}\label{eqn sq}
(i,r)\sqsubset (j,s)\Leftrightarrow j<i \text{ and } T_{(i,r)}\otimes T_{(j,s)}\big\arrowvert \Ind_{L_{i,j}}^{\mathfrak{S}_{n_i}\times\mathfrak{S}_{n_j}}(k)\,,
\end{equation}
%$(i,r)\sqsubset (j,s)$ if and only if $j<i$ and 
%$$T_{(i,r)}\otimes T_{(j,s)}\big\arrowvert \Ind_{L_{i,j}}^{\mathfrak{S}_{n_i}\times\mathfrak{S}_{n_j}}(k)\,$$
where $L_{i,j}:=\stab_{\mathfrak{S}_{n_i}\times\mathfrak{S}_{n_j}}(e_j)$. Note that here we again used the fact that the
simple $k\mathfrak{S}_{n_j}$-module $T_{(j,s)}$ is self-dual. So, by Lemma~\ref{lemma bisets}, we infer that
\begin{equation}\label{eqn ind Lij}
  (i,r)\sqsubset (j,s)\Leftrightarrow j<i\text{ and } T_{(i,r)}\big\arrowvert 
  k[(\mathfrak{S}_{n_i}\times\mathfrak{S}_{n_j})/L_{i,j}]\otimes_{k\mathfrak{S}_{n_j}}T_{(j,s)}\,.
\end{equation}
Now suppose that $j<i$, so that $n_i-n_j=2(i-j)$. In order to describe $L_{i,j}$, let first $W_{i,j}$ denote the subgroup of $\mathfrak{S}_{n_i}$ defined by
\begin{align}\label{eqn Wij}
  W_{i,j}:&=\langle (n_j+1,n_j+2),\, (n_j+1,n_j+3)(n_j+2,n_j+4),\\ \nonumber
  &(n_j+1,n_j+3,\ldots,n_j+2(i-j)-1)(n_j+2,n_j+4,\ldots,n_j+2(i-j))\rangle\,.
\end{align}
Then $W_{i,j}$ is isomorphic to the wreath product $\mathfrak{S}_2\wr\mathfrak{S}_{i-j}$, and we have
$L_{i,j}=\Delta(\mathfrak{S}_{n_j})\cdot (W_{i,j}\times 1)\leq\mathfrak{S}_{n_i}\times\mathfrak{S}_{n_j}$. Thus, writing
$L_{i,j}$ as a quintuple as in \ref{nota ghost}(b), this gives
\begin{equation}\label{eqn Lij}
L_{i,j}=(\mathfrak{S}_{n_j}\times W_{i,j},W_{i,j},\eta_{i,j},\mathfrak{S}_{n_j},1)\,,
\end{equation}
where $\eta_{i,j}:\mathfrak{S}_{n_j}\myiso (\mathfrak{S}_{n_j}\times W_{i,j})/W_{i,j},\;\sigma\mapsto (\sigma,1)W_{i,j}$.

\smallskip
Consequently, we have shown the following:
\begin{align}\label{eqn lhd}
(i,r)&\lhd (j,s)\Leftrightarrow \exists\, q\in\mathbb{N},\,  (i_0,r_0),\ldots,(i_q,r_q)\in\Lambda:\\\nonumber
                    &\quad\quad (i,r)=(i_0,r_0)\sqsubset (i_1,r_1)\sqsubset\cdots\sqsubset (i_q,r_q)=(j,s)\\\nonumber
                     &\Leftrightarrow\exists\, q\in\mathbb{N},\,  (i_0,r_0),\ldots,(i_q,r_q)\in\Lambda:\\\nonumber
                     &\quad \quad j=i_q<\ldots< i_1<i_0=i \text{ and }\\\nonumber
                     & T_{(i_p,r_p)}\big\arrowvert k[(\mathfrak{S}_{n_{i_p}}\times\mathfrak{S}_{n_{i_{p+1}}})/L_{i_p,i_{p+1}}]\otimes_{k\mathfrak{S}_{n_{i_{p+1}}}}T_{(i_{p+1},r_{p+1})}\; \;(0\leq p\leq q-1)\\\nonumber
                      &\Leftrightarrow\exists\, q\in\mathbb{N},\,  (i_0,r_0),\ldots,(i_q,r_q)\in\Lambda:\\\nonumber
                     &\quad \quad j=i_q<\ldots< i_1<i_0=i \text{ and }\\\nonumber
                     & T_{(i_0,r_0)}\big\arrowvert k[(\mathfrak{S}_{n_{i_0}}\times\mathfrak{S}_{n_{i_1}})/L_{i_0,i_1}]\otimes_{k\mathfrak{S}_{n_{i_1}}}\cdots\otimes_{k\mathfrak{S}_{n_{i_{q-1}}}}    k[(\mathfrak{S}_{n_{i_{q-1}}}\times\mathfrak{S}_{n_{i_{q}}})/L_{i_{q-1},i_{q}}]\otimes_{k\mathfrak{S}_{n_{i_q}}}T_{(i_{q},r_{q})}\,.
\end{align}

Suppose that $(i_0,r_0),\ldots,(i_q,r_q)\in\Lambda$ are such that $j=i_q<\ldots< i_1<i_0=i$.
Then the $(k\mathfrak{S}_{n_i},k\mathfrak{S}_{n_j})$-bimodule $k[\mathfrak{S}_{n_{i_0}}\times\mathfrak{S}_{n_{i_1}}/L_{i_0,i_1}]\otimes_{k\mathfrak{S}_{n_{i_1}}}\cdots\otimes_{k\mathfrak{S}_{n_{i_{q-1}}}}    k[\mathfrak{S}_{n_{i_{q-1}}}\times\mathfrak{S}_{n_{i_{q}}}/L_{i_{q-1},i_{q}}]$ is isomorphic to $kX$, where $X$ is the $(\mathfrak{S}_{n_i},\mathfrak{S}_{n_j})$-biset 
$$X:=(\mathfrak{S}_{n_{i_0}}\times\mathfrak{S}_{n_{i_1}}/L_{i_0,i_1})\times_{\mathfrak{S}_{n_{i_1}}}\cdots\times_{\mathfrak{S}_{n_{i_{q-1}}}} (\mathfrak{S}_{n_{i_{q-1}}}\times\mathfrak{S}_{n_{i_{q}}}/L_{i_{q-1},i_{q}})\,.$$
For a precise definition of the tensor product of bisets, see \cite[Definition~2.3.11]{Bouc}.
Since $p_2(L_{i_p,i_{p+1}})=\mathfrak{S}_{n_{i_{p+1}}}$, for $p=0,\ldots,q-1$, the Mackey formula for tensor products
of bisets in \cite[Lemma~2.3.24]{Bouc} gives $X\cong \mathfrak{S}_{n_i}\times\mathfrak{S}_{n_j}/L$, where
\begin{equation}\label{eqn L}
L:=(\mathfrak{S}_{n_j}\times W_{i_{q-1},i_q}\times W_{i_{q-2},i_{q-1}}\times\cdots\times W_{i_0,i_1},
W_{i_{q-1},i_q}\times W_{i_{q-2},i_{q-1}}\times\cdots\times W_{i_0,i_1},\eta,\mathfrak{S}_{n_j},1)
\end{equation}
and  $\eta(\sigma):=(\sigma,1,\ldots,1)(W_{i_{q-1},i_q}\times\cdots\times W_{i_0,i_1})$, for $\sigma\in\mathfrak{S}_{n_j}$.
Set $W:=W_{i_{q-1},i_q}\times\cdots\times W_{i_0,i_1}$.
Then, by \cite[Lemma~2.3.26]{Bouc}, we also know that, as a functor from $k\mathfrak{S}_{n_j}$\textbf{-mod}
to $k\mathfrak{S}_{n_i}$\textbf{-mod}, tensoring with $kX$ over $k\mathfrak{S}_{n_j}$
is equivalent to $\Ind_{\mathfrak{S}_{n_j}\times W}^{\mathfrak{S}_{n_i}}\circ\Inf_{\mathfrak{S}_{n_j}}^{\mathfrak{S}_{n_j}\times W}$. Therefore, this implies

\begin{align}\nonumber
(i,r)\lhd (j,s)&\Leftrightarrow \exists\, q\in\mathbb{N},\,   j=j_q<\ldots <i_1<i_0=i:\;
                    T_{(i,r)}\big\arrowvert \Ind_{\mathfrak{S}_{n_j}\times W}^{\mathfrak{S}_{n_i}}(\Inf_{\mathfrak{S}_{n_j}}^{\mathfrak{S}_{n_j}\times W}(T_{(j,s)}))\,,\\ \label{eqn W}
                   & \quad\quad \text{for } W:=W_{i_{q-1},i_q}\times\cdots\times W_{i_0,i_1}\,.
\end{align}
Recall again that, by Lemma~\ref{lemma bisets}, the condition on the right-hand side of (\ref{eqn W}) holds if and only if
$$T_{(i,r)}\otimes T_{(j,s)}\big\arrowvert \Ind_L^{\mathfrak{S}_{n_i}\times\mathfrak{S}_{n_j}}(k)\,;$$
%or equivalently, if and only if
%$$k\big\arrowvert \Res_L^{\mathfrak{S}_{n_i}\times\mathfrak{S}_{n_j}}(T_{(i,r)}\otimes T_{(j,s)})\,,$$
%by Frobenius Reciprocity. 
here $L$ is the group in (\ref{eqn L}), which contains the subgroup
\begin{equation}\label{eqn M}
M:=(\mathfrak{S}_{n_j}\times (\mathfrak{S}_2)^{i-j},(\mathfrak{S}_2)^{i-j},\eta',\mathfrak{S}_{n_j},1)\,,
\end{equation}
where $\eta'(\sigma)=(\sigma,1,\ldots,1)\mathfrak{S}_{2}^{i-j}$, for $\sigma\in\mathfrak{S}_{n_j}$.
Hence, if
$T_{(i,r)}\otimes T_{(j,s)}\big\arrowvert \Ind_L^{\mathfrak{S}_{n_i}\times\mathfrak{S}_{n_j}}(k)$
then we have
$T_{(i,r)}\otimes T_{(j,s)}\big\arrowvert \Ind_M^{\mathfrak{S}_{n_i}\times\mathfrak{S}_{n_j}}(k)$
as well.
Conversely, if $T_{(i,r)}\otimes T_{(j,s)}\big\arrowvert \Ind_M^{\mathfrak{S}_{n_i}\times\mathfrak{S}_{n_j}}(k)$
then $T_{(i,r)}\big\arrowvert \Ind_{\mathfrak{S}_{n_j}\times (\mathfrak{S}_2)^{i-j}}^{\mathfrak{S}_{n_i}}(\Inf_{\mathfrak{S}_{n_j}}^{\mathfrak{S}_{n_j}\times (\mathfrak{S}_2)^{i-j}}(T_{(j,s)}))$, again by Lemma~\ref{lemma bisets} and \cite[Lemma~2.3.26]{Bouc}.
But then we may consider the chain $j<j+1<j+2<\ldots< i-1<i$ and the group $W:=W_{j+1,j}\times\cdots\times W_{i,i-1}=(\mathfrak{S}_2)^{i-j}$. So our above considerations imply $(i,r)\lhd (j,s)$.

\smallskip

To summarize, we have now shown that
\begin{equation}\label{eqn Y}
(i,r)\lhd (j,s)\Leftrightarrow j<i\quad \text{ and }\quad T_{(i,r)}\big\arrowvert \Ind_{\mathfrak{S}_{n_j}\times (\mathfrak{S}_2)^{i-j}}^{\mathfrak{S}_{n_i}}(\Inf_{\mathfrak{S}_{n_j}}^{\mathfrak{S}_{n_j}\times (\mathfrak{S}_2)^{i-j}}(T_{(j,s)}))\,.
\end{equation}
Since $\Inf_{\mathfrak{S}_{n_j}}^{\mathfrak{S}_{n_j}\times (\mathfrak{S}_2)^{i-j}}(T_{(j,s)})\cong T_{(j,s)}\otimes k\otimes \cdots\otimes k$,
for $j<i$, this completes the proof of (c).
\end{proof}

The next example will show that, also in the case where the twisted category algebra is a Brauer algebra,
the partial order $\leq$ on $\Lambda$ is a proper refinement of $\unlhd$, and that the relation
$\sqsubseteq$ in Definition~\ref{defi new order} is not transitive.

\begin{example}\label{expl brauer sq}
Let $n:=6$, so that $d=\frac{6}{2}=3$; in particular, there are four $\scrJ$-classes of $S$. Moreover, 
for simplicity let $k:=\mathbb{C}$.

\smallskip

(a)\, If $i=3$ and $j=2$ then 
$n_i=4$ and $n_j=2$. The isomorphism classes of simple $k\mathfrak{S}_4$-modules
are labelled by the partitions of 4, and the isomorphism classes of simple $k\mathfrak{S}_2$-modules are labelled by
the partitions of 2. We thus choose our notation in such a way that the simple $k\mathfrak{S}_{n_i}$-module
$T_{(i,r)}$ corresponds to partition $\lambda_r$, and the simple $k\mathfrak{S}_{n_j}$-module $T_{(j,s)}$ corresponds
to the partition $\lambda_s$:

\smallskip

\begin{center}
\begin{tabular}{|c|c|c|c|c|c||c|c|c|}\hline
$r$ & 1& 2 &3 & 4 & 5& $s$& 1& 2\\\hline
$\lambda_r$& $(4)$ & $(3,1)$ & $(2^2)$ & $(2,1^2)$& $(1^4)$&$\lambda_s$& $(2)$ &$(1^2)$\\\hline
\end{tabular}
\end{center}
By the Littlewood--Richardson Rule \cite[Theorem~16.4]{J}, we obtain 
$$
\Ind_{\mathfrak{S}_2\times\mathfrak{S}_2}^{\mathfrak{S}_4}(T_{(j,1)}\otimes k)\cong T_{(i,1)}\oplus T_{(i,2)}\oplus T_{(i,3)}\; \text{ and }\;
\Ind_{\mathfrak{S}_2\times\mathfrak{S}_2}^{\mathfrak{S}_4}(T_{(j,2)}\otimes k)\cong T_{(i,2)}\oplus T_{(i,4)}\,,
$$
so that $(i,r)\lhd (j,s)$ if and only if $(r,s)\in\{(1,1),(2,1),(3,1),(2,2),(4,2)\}$, by Proposition~\ref{prop order brauer}\l c\r.
On the other hand, $(i,r)< (j,s)$, for all $1\leq r\leq 5$ and $1\leq s\leq 2$.

\medskip

(b)\, Now let $l:=4$, so that $n_l=6$.  Again, we choose our labelling such that the simple $k\mathfrak{S}_6$-module
$T_{(l,t)}$ corresponds to the partition $\lambda_t$ of 6:

\smallskip

\begin{center}
\begin{tabular}{|c|c|c|c|c|c|c|c|c|c|c|c|}\hline
$t$ & 1&2 &3&4&5&6&7&8&9&10&11\\\hline
$\lambda_t$& $(6)$& $(5,1)$& $(4,2)$& $(4,1^2)$& $(3^2)$& $(3,2,1)$& $(3,1^3)$& $(2^3)$& $(2^2,1^2)$& $(2,1^4)$& $(1^4)$\\\hline
\end{tabular}
\end{center}

\smallskip

Using the notation introduced in the proof of Proposition~\ref{prop order brauer}, we have
\begin{align*}
L_{4,2}&=\stab_{\mathfrak{S}_6\times\mathfrak{S}_2}(e_2)=\Delta(\mathfrak{S}_2)(W_{4,2}\times 1)\text{ with } W_{4,2}\cong \mathfrak{S}_2\wr\mathfrak{S}_2\,,\\
L_{4,3}&=\stab_{\mathfrak{S}_6\times\mathfrak{S}_4}(e_3)=\Delta(\mathfrak{S}_4)(W_{4,3}\times 1)\text{ with } W_{4,3}\cong \mathfrak{S}_2\,,\\
L_{3,2}&=\stab_{\mathfrak{S}_4\times\mathfrak{S}_2}(e_2)=\Delta(\mathfrak{S}_2)(W_{3,2}\times 1)\text{ with } W_{3,2}\cong \mathfrak{S}_2\,.
\end{align*}
Recall that, by (\ref{eqn sq}), we have $(i,r)\sqsubset (j,s)$ if and only if $T_{(i,r)}\otimes T_{(j,s)}\big\arrowvert \Ind_{L_{i,j}}^{\mathfrak{S}_{n_i}\times\mathfrak{S}_{n_j}}(k)$.
By a character computation with MAGMA \cite{MAGMA}, we infer that $(i,2)\sqsubset (j,1)$ and $(l,5)\sqsubset (i,2)$, whereas $(l,5)\not\sqsubset (j,1)$.
%{\tt Should we insert a table recording the relation $\sqsubseteq$, for $n=6$?}
\end{example}

\begin{remark}\label{rem CdVM}
Via the partial order $\unlhd$ we, in particular, obtain information on the decomposition numbers of the Brauer algebra $B_\delta(n)$, that is, on the composition factors of standard modules. It should be pointed out that further information on decomposition numbers
of Brauer algebras over fields of characteristic 0 can, for instance,  be found in \cite{CdVM}.
\end{remark}

%{\tt Comment on \cite{CdVM}??? E.g., using our Proposition~\ref{prop dual standard modules} and
%Proposition~4.2 in \cite{CdVM}, one can define yet another partial order on $\Lambda$ that also
%induces the same quasi-hereditary structure in the Brauer algebra. For some parameters such as 
%$\delta=1$, this partial order is coarser than
%our partial order $\unlhd$ (thus better). For other parameters, e.g., $\delta=2$, the order $\unlhd$ apparently yields
%stronger results concerning decomposition numbers. }

%%%%%%%%%%%%%%%%%%%%%%%%%%%%%%%%%%%%%%%%%%%%%%%%%%%%%%%%%%%%

\pagebreak
%%%%%%%%%%%%%% BIBLIOGRAPHY %%%%%%%%%%%%%%%%%%%%%%%%%%%%%%%%%%


\begin{thebibliography}{00}


\bibitem{BDII} {\sc R.~Boltje, S.~Danz:} A ghost algebra of the double Burnside 
algebra in characteristic zero. J. Pure Appl. Algebra {\bf 217} (2013), 608--635.

\bibitem{BDIII} {\sc R.~Boltje, S.~Danz:} Twisted split category algebras as quasi-hereditary algebras.
Arch. Math. {\bf 99} (2012), 589--600.

\bibitem{MAGMA} {\sc W.~Bosma, J.~Cannon, C.~Playoust:} The Magma algebra system. I. The user language, J.~Symbolic Comput. {\bf 24} (3-4) (1997), 235--265.

\bibitem{Bouc}{\sc S.~Bouc:} Biset functors for finite groups. Lecture Notes in Mathematics, vol. 1990, Springer-Verlag, Berlin, 2010.

\bibitem{BST}{\sc S.~Bouc, R.~Stancu, J.~Th\'evenaz:} Simple biset functors and double Burnside rings. J. Pure Appl. Algebra {\bf 217} (2013), 546--566.

%\bibitem{BCdV} {\sc C.~Bowman, A.~Cox, M.~De Visscher:} Decomposition
%numbers for the cyclotomic Brauer algebras in characteristic zero. preprint, 2012.
%{\sf arXiv:1205.3345v1 [math.RT]}.



\bibitem{CPS} {\sc E.~Cline, B.~Parshall, L.~Scott:} Finite-dimensional
algebras and highest weight categories. J. Reine Angew. Math. \textbf{391}
(1988), 85--99.

\bibitem{CPS2} {\sc E.~Cline, B.~Parshall, L.~Scott:} Duality in highest weight categories. Contemp. Math. {\bf 82} (1989), 7--22.

\bibitem{CdVM} {\sc A.~Cox, M.~De Visscher, P.~Martin:} The blocks of the Brauer algebra in characteristic zero. 
Represent.~Theory {\bf 13} (2009), 272--308. 

\bibitem{CR} {\sc C.~W.~Curtis, I.~Reiner:} Methods of representation
theory, vol. 1, J.~Wiley \& Sons, New York, 1981.

\bibitem{DR} {\sc V.~Dlab, C.~M.~Ringel:} 
   Quasi-hereditary algebras. 
   Illinois J.~Math.~{\textbf 33} (1989), 280--291. 

\bibitem{D} {\sc S.~Donkin:}
   The $q$-Schur algebra.
   London Mathematical Society Lecture Note Series, 253. Cambridge University Press, Cambridge, 1998.
   
\bibitem{GMS} {\sc O.~Ganyushkin, V.~Mazorchuk, B.~Steinberg:} On the 
irreducible
representations of a finite semigroup, Proc. Amer. Math. Soc. \textbf{137} (2009), 
3585--3592.

%\bibitem{GAP} {\sc The GAP Group}, GAP---Groups, Algorithms, Programming, a system for computational
%discrete algebra, version 4.7.2 (2013), {\sf http://www.gap-system.org}.

\bibitem{GL} {\sc J.~J.~Graham, G.~I.~Lehrer:} Cellular algebras,
Invent. Math. {\bf 123} (1996), 1--34.

\bibitem{Gr1}{\sc J.~A.~Green:} On the structure of semigroups. Ann. of Math. \textbf{54} (1951), 163--172.

\bibitem{Gr} {\sc J.~A.~Green:} Polynomial representations of
$\mathrm{GL}_n$. Second corrected and augmented edition.
 With an appendix on Schensted correspondence and Littelmann paths by K.
Erdmann, J.~A. Green and M. Schocker.
 Lecture Notes in Mathematics, 830. Springer-Verlag, Berlin, 2007.


\bibitem{I} {\sc R.~S.~Irving:} BGG-algebras and the BGG reciprocity principle. J. Algebra {\bf 135} (1990), 363--380.

\bibitem{J} {\sc G.~D.~James:} The representation theory of the symmetric groups. Lecture Notes in Mathematics, 
682. Springer, Berlin, 1978.

\bibitem{JK} {\sc G.~D.~James, A.~Kerber:} The representation theory of the
symmetric group, Encyclopedia Math. Appl., vol. 16,
Addison-Wesley, Reading, 1981.


\bibitem{KX} {\sc S.~K\"onig, C.~Xi:} When is a cellular algebra quasi-hereditary? Math. Ann. {\bf 315} (1999), no. 2, 281--293.

\bibitem{LS} {\sc M.~Linckelmann, M.~Stolorz:} On simple
modules over twisted finite category algebras. Proc. Amer. Math. Soc. {\bf 140} (2012),
no.~11, 3725--3737.

\bibitem{LS2} {\sc M.~Linckelmann, M.~Stolorz:} Quasi-hereditary twisted category algebras. J.~Algebra {\bf 385} (2013), 1--13.

\bibitem{M} {\sc P.~Martin:} The structure of the partition algebras, J. Algebra
{\bf 183} (1996), 219--358.

\bibitem{Maz} {\sc V:~Mazorchuk:} On the structure of Brauer semigroup and its partial analogue. Problems in Algebra {\bf 13} (1998),
29--45.

\bibitem{NT} {\sc H.~Nagao, Y.~Tsushima:} Representations of finite groups, Academic Press, Boston, 1989.

\bibitem{Putcha} {\sc M.~S.~Putcha:} 
   Complex representations of finite monoids. II. Highest weight categories and quivers. 
   J.~Algebra {\textbf 205} (1998), 53--76. 

%\bibitem{RX} {\sc H.~Rui, C.~Xi:} The representation theory of the cyclotomic
%Temperley--Lieb algebras, Comment. Math. Helvet. {\bf 79} (2004), 427--450.

%\bibitem{RY} {\sc H.~Rui, W.~Yu:} On the semi-simplicity of the cyclotomic
%Brauer algebras, J. Algebra {\bf 277} (2004), 187--221.

\bibitem{Webb} {\sc P.~Webb:} Stratifications and Mackey functors II: globally defined Mackey functors.
      J. K-Theory {\textbf 6} (2010), 99--170.

%\bibitem{West} {\sc B.~W.~Westbury:} The representation theory of the
%Temperley--Lieb algebras, Math. Z. {\bf 219} (1995), 539--565.


\bibitem{Wil} {\sc S.~Wilcox:} Cellularity of diagram algebras as twisted semigroup algebras. J.~Algebra {\bf 309} (2007), 10--31.

%\bibitem{X} {\sc C.~Xi:} Partition algebras are cellular, Compositio Math. {\bf 119} (1999), 99--109.


\end{thebibliography}
\end{document}